\documentclass[ejsv2,noshowframe,preprint]{imsart}

\RequirePackage[colorlinks,citecolor=blue,urlcolor=blue]{hyperref}
\RequirePackage{graphicx}

\RequirePackage{amsthm,amsmath}
\RequirePackage[numbers,sort&compress]{natbib}

\usepackage{euscript,bbm,xcolor,epstopdf}
\usepackage{mathrsfs}

\usepackage{bm,euscript,bbm}
\usepackage{verbatim}
\usepackage{multirow}
\usepackage{epstopdf}
\usepackage{color}
\usepackage{pstricks,pst-node,pst-plot}
\usepackage{rotating}
\usepackage{enumitem}
\usepackage{booktabs}

\startlocaldefs
\theoremstyle{plain}

\newtheorem{theorem}{Theorem}[section]
\newtheorem{lemma}[theorem]{Lemma}
\newtheorem{corollary}[theorem]{Corollary}
\newtheorem{proposition}[theorem]{Proposition}

\theoremstyle{definition}
\newtheorem{definition}[theorem]{Definition}
\newtheorem{example}[theorem]{Example}
\newtheorem{remark}[theorem]{Remark}
\newtheorem{assumption}[theorem]{Assumption}

\def\R{{\mathbb R}}
\def\E{{\mathbb E}}
\def\P{{\mathbb P}}
\def\Z{{\mathbb Z}}

\endlocaldefs

\begin{document}
	\begin{frontmatter}
		\title{Cluster size distributions of discrete random fields}
		\runtitle{Cluster size distribution}
		
		\begin{aug}
			\author{\fnms{Dan}~\snm{Cheng}\ead[label=e1]{chengdan@asu.edu}}
			\and
			\author{\fnms{John}~\snm{Ginos}\ead[label=e2]{jginos@asu.edu}}
			\address{School of Mathematical and Statistical Sciences, Arizona State University\printead[presep={,\ }]{e1,e2}}
		\end{aug}
		
		\begin{abstract}
			We study size distributions of finite excursion clusters of
			discrete random fields on the integer lattice $\mathbb Z^d$.
			For stationary fields, we assign a unique root to each finite
			cluster and derive an exact representation of the root Palm
			size law through joint exceedance and exterior nonexceedance
			probabilities for complete connected shapes. Root sampling
			gives each finite cluster equal weight and remains well
			defined even when infinite components are present.
			We also study the cluster size law
			conditional on a specified site being a local maximum above
			the threshold $u$, providing a local formulation for nonstationary
			fields. Under stationarity and almost sure cluster finiteness,
			an exact change of measure shows that peak sampling reweights
			the root law by the number of local maxima in each cluster.
			For stationary Gaussian fields satisfying suitable conditions,
			the total variation distance between these laws decreases
			exponentially in $u^2$. Guarded spatial Monte
			Carlo yields unbiased estimators of cluster intensities,
			estimates of sampling uncertainty from independent realizations,
			and a finite sample certificate for root truncation error.
			We further establish sufficient exponential cluster size tail
			bounds at a fixed threshold and exact logarithmic Gaussian
			decay rates at a fixed cluster size as $u\to \infty$.
			The latter rates are characterized by quadratic programs
			retaining the exterior constraints needed to identify complete
			components. Numerical illustrations compare direct probability
			evaluations with Monte Carlo estimates and assess sampling
			uncertainty and truncation error. The general representations
			and estimators apply to both Gaussian and non-Gaussian fields, providing a probabilistic foundation for quantifying spatial extent in discretely sampled data.
		\end{abstract}
		
		
		
	\end{frontmatter}
	
	\section{Introduction}	
	Random fields play a central role in modern spatial statistics, stochastic modeling, and signal detection. They provide a natural framework for characterizing spatially correlated data across a wide range of scientific disciplines, including medical imaging, geoscience, environmental studies, and cosmology \cite{AT07,Cheng:2020,Chiles:2012,Goeman:2023,MP11,Taylor:2007}.
	When a random field is observed over a discrete lattice such as $\Z^d$, the resulting process is referred to as a \textit{discrete random field}. Let $X=\{X_t:t\in\Z^d\}$ be a real-valued discrete random field. For a threshold $u\in\R$, denote the excursion set by
	\[
	E_u=\{t\in\Z^d:X_t>u\}.
	\] 
	Then $E_u$ decomposes into connected components, called excursion clusters. Their sizes, or cardinalities, describe the spatial extent of threshold exceedances.
	
	Understanding the statistical distribution of cluster sizes in random fields is fundamental for both theory and applications. In spatial statistics and random field theory, the cluster size distribution under a null (pure noise) model provides a benchmark for assessing the significance of observed spatial patterns. For example, in neuroimaging, random field theory underlies cluster-based inference procedures for detecting activation regions in fMRI data \cite{Bansal:2018,Chumbley:2010,Friston:1994,Goeman:2023,Poline:1997,Worsley:1996a,Zhang:2009}. Similar principles arise in geoscience, genomics and cosmology, where cluster size distributions describe the spatial extent of anomalies in correlated data  \cite{Adler81,Bardeen:1986,Benjamini:2019,Lindgren82,Vanmarcke:2010}. Recent work in climate and environmental statistics has emphasized the importance of spatial extent inference, which aims to quantify the areal footprint of extreme events \cite{Tan:2021,Zhong:2025}. In parallel, inference for threshold exceedances in random fields has been applied to region detection \cite{Benjamini:2019} and spatiotemporal extreme value analysis \cite{Fondeville:2018,Huser:2014}. 
	
	When the threshold indicators $Y_t=\mathbbm{1}_{\{X_t>u\}}$
	are i.i.d. Bernoulli variables, excursion clusters are the occupied
	components of Bernoulli site percolation. Cluster number densities,
	lattice animal enumeration, and perimeter formulas are classical in that
	setting \cite{BroadbentHammersley57,SykesEssam64,Grimmett99,StaufferAharony92}.
	Sharp exponential decay results are also well established for
	independent percolation \cite{AizenmanBarsky87,AntunovicVeselic08}. Threshold
	excursions of dependent Gaussian fields have likewise been studied as
	Gaussian percolation on $\Z^d$ and $\R^d$, with emphasis on critical levels,
	sharpness, and connectivity probabilities \cite{Muirhead24}. The present
	paper instead targets normalized finite-component size laws and uses a direct
	sufficient bound tailored to truncation control.
	
	The natural language for a typical stationary cluster is Palm theory.
	Assigning each finite cluster a translation-equivariant canonical root
	produces a stationary point process of roots, and its Palm distribution is
	the law seen from a typical root. On the lattice, this Palm law is ordinary
	conditioning on a root at the origin. General background on Palm
	distributions and stationary partitions is given in
	\cite{CoeurjollyEtAl17,Last06}. Palm duality for random subsets of integer
	grids is developed in \cite{Thorisson07}, while invariant transports and
	mass stationarity are treated in \cite{LastThorisson09}. In extreme value
	theory, limiting cluster size inference for stationary sequences is studied
	in \cite{Robert09}, anchored Palm representations of typical extremal
	clusters in stationary regularly varying time series and random fields are
	developed in \cite{Planinic23}, and spatial extremal cluster counting is
	studied in \cite{RonnNielsenStehr23}. In particular, \cite{Planinic23}
	distinguishes a typical extremal cluster from the size-biased cluster viewed
	from a typical exceedance. These works concern asymptotic high-threshold
	clusters under regular variation, mixing, or cluster counting conditions.
	
	The starting point of this paper is an explicit representation of the
	root Palm cluster size law through complete excursion shape probabilities
	at a fixed threshold (Theorem~\ref{thm:CSD_d}). For each rooted connected
	shape, exceedances are imposed at every site in the shape and
	nonexceedances at every site in its exterior neighbor set. The exterior
	constraints identify an entire connected component. Summing these
	mutually exclusive event probabilities over shapes of size $k$ gives
	the intensity $w_k$; normalizing by the total cluster intensity gives
	the typical cluster size distribution. The Palm
	theory has an established probabilistic foundation. Our
	representation makes the sampling law explicit in joint threshold
	events for general stationary fields, without requiring Gaussianity,
	independence, regular variation, or mixing, subject to
	Assumption~\ref{ass:finite-cluster} and positive root intensity.
	
	The principal consequences concern how this representation can be used.
	First, it permits an exact comparison of cluster sampling rules: a typical
	peak samples a cluster in proportion to its number of local maxima,
	whereas a typical root gives each cluster one vote. Second, it supplies
	finite-region scores for estimating the intensities without
	misclassifying clusters cut by the simulation boundary. Independent
	simulated realizations then provide sampling uncertainty estimates, and
	an overflow score certifies the error from truncating the root
	normalizer. Third, it turns Gaussian high-threshold probabilities into
	quadratic minimization problems whose constraints retain the exterior
	of the complete component. These consequences connect the
	representation to estimation and quantitative analysis, rather than
	leaving it as a formal sum over shapes.
	
	For $t\in E_u$, let \(C_u(t)\) denote the connected component
	of \(E_u\) containing \(t\), and set \(C_u(t)=\varnothing\) when
	\(t\notin E_u\). The size $|C_u(t)|$ is defined as the number of lattice points in $C_u(t)$. Unless explicitly stated otherwise, stationarity means strict stationarity.  We use the following standing assumption at the threshold under consideration.
	
	\begin{assumption}[Finite clusters]\label{ass:finite-cluster}
		At the threshold $u$ under consideration,
		\begin{equation*}
			\mathbb P\left(
			|C_u(t)|<\infty \text{ for every }t\in E_u
			\right)=1.
		\end{equation*}
	\end{assumption}
	
	The assumption excludes percolating excursion components and ensures
	that the normalized size laws have total mass one on the positive
	integers. Finite-size shape probabilities and guarded classifications
	through a fixed cutoff remain meaningful without this assumption;
	their interpretation is stated separately when needed. If the assumption
	holds at $u$, then it automatically holds for $u'>u$. For a stationary field, Assumption~\ref{ass:finite-cluster}
	is equivalent to $\mathbb P\bigl(|C_u(t_0)|=\infty\bigr)=0$ for some $t_0$. Illustrations of clusters and their sizes in one and two dimensions are provided in Figures \ref{fig:simul-1D} and \ref{fig:simul-2D}.
	
	Our main results are as follows.
	\begin{enumerate}[label=(\roman*)]
		\item For general stationary discrete random fields, we derive an explicit,
		exact representation of the root Palm cluster size distribution at a
		fixed threshold in terms of joint exceedance and exterior nonexceedance
		probabilities for complete connected shapes
		(Theorem~\ref{thm:CSD_d}). The underlying Palm sampling convention
		assigns a unique root to each cluster and conditions on the origin being
		a cluster root. We establish the empirical interpretation through cluster
		counts in expanding observation windows. Under ergodicity, the empirical
		size distribution obtained from a single realization converges almost
		surely to this Palm distribution as the observation window grows (Theorem \ref{thm:empirical-cluster-counts}).
		\item We define and derive the peak-based cluster size distribution for stationary and
		nonstationary fields by conditioning on a specified site being a 
		local maximum above $u$ (Theorem \ref{thm:peak-CSD}). For stationary fields, peak
		sampling weights each cluster by its number of local maxima,
		whereas root sampling gives each cluster one vote. We express this
		relationship through an exact change-of-measure identity. Using the
		total variation between two discrete measures
		\[
		d_{\rm TV}(\mu,\nu):=\frac12\sum_{k\geq1}|\mu_k-\nu_k|,
		\]
		we obtain a bound for the distance between two size distributions. For Gaussian fields, this
		distance decreases exponentially in $u^2$ as $u\to\infty$. This convergence
		is driven by concentration on singleton clusters: conditional on an
		exceedance at a site, the probability that any neighboring site also
		exceeds $u$ becomes exponentially small, so both sampling laws
		concentrate on clusters containing a single site.
		\item We construct spatial Monte Carlo estimators by simulating the field on
		an enlarged region surrounding an interior observation window. The
		surrounding guard region allows us to determine exactly whether the
		full excursion cluster containing each scored exceedance site has size
		at most $K$ or greater than $K$, preventing misclassification caused by
		the simulation boundary. We obtain unbiased estimators of the
		unnormalized root and peak intensities and establish consistency as
		the number of independent simulated realizations increases. We also
		provide variance and covariance estimators based on independent
		realizations and a finite-sample certificate for the root truncation
		error (Proposition~\ref{prop:finite-sample-truncation}). These are
		distinct sources of error: increasing the number of realizations
		reduces sampling uncertainty, while increasing $K$ reduces the
		omitted cluster mass. The spatial variance formula
		(Proposition~\ref{prop:mc-window-variance}) describes how the scoring
		window affects precision.
		\item We analyze two complementary tail regimes. First, we fix the threshold
		$u$ and let the cluster size increase. A union bound over possible
		cluster shapes, called lattice animals, yields exponential
		bounds on cluster size tails under sufficient conditions, including a covariance summability criterion for Gaussian fields. These bounds can certify finiteness and control the mass omitted from computation when their sufficient conditions hold. Second, we fix the cluster size $k$ and let $u\to\infty$. Finite-dimensional Gaussian large deviation theory yields the exact
		logarithmic decay rate for each complete cluster shape through quadratic
		minimization, subject to exceedance constraints inside the shape and
		nonexceedance constraints on its exterior boundary
		(Theorem~\ref{thm:ld-high-threshold}). 
	\end{enumerate}
	
	For a specified lattice noise model, these results provide a concrete
	computational route. Probabilities for moderate-size shapes can be evaluated
	directly. The guarded spatial Monte Carlo method estimates larger-size weights and quantifies sampling and truncation errors. The complete shape events also yield Gaussian quadratic programs for comparing rare component
	shapes at high thresholds.
	
	In many applications, data are observed on discrete grids rather than as continuous fields. The framework developed in this paper thus provides a unified theoretical foundation for cluster modeling and inference in discrete random fields. Beyond its intrinsic theoretical interest, the results have practical implications for signal detection and spatial extent inference in correlated lattice data. Applications include neuroimaging \cite{Bansal:2018,Chumbley:2010,Friston:1994,Goeman:2023,Worsley:1996a}, geophysical anomaly detection \cite{Chiles:2012,Vanmarcke:2010}, and climate data analysis \cite{Cooley:2007,Davison:2012}, where discretely sampled random fields are ubiquitous. Further, the developed methodology applies to broad classes of both Gaussian and non-Gaussian random fields.
	
	The paper is organized as follows. Section~\ref{sec:exact-cluster} defines the root typical-cluster law and its empirical version. Section~\ref{sec:peak-based-cluster} develops the peak-sampled law and compares the sampling schemes. Section~\ref{sec:MC} gives the guarded Monte Carlo construction. Section~\ref{sec:large-deviations} studies fixed-threshold cluster tails and fixed-size Gaussian rates for high thresholds. Section~\ref{sec:simulation} presents numerical illustrations, and Section~\ref{sec:discussion} discusses interpretation, limitations, and further work.

	\section{The root cluster size distribution}\label{sec:exact-cluster}
	
	We begin with a stationary random field on $\Z^d$ at a fixed threshold.
	Assigning one translation equivariant root to every finite excursion
	cluster produces a stationary point process, whose Palm law defines a
	typical cluster. Last's stationary partition framework \cite{Last06}
	provides the Palm interpretation; the related distinction between typical
	clusters and clusters viewed from an exceedance site is central to
	Planini\'c's extremal theory \cite{Planinic23}. The contribution of this section is an explicit,
	exact representation of its size distribution through complete excursion
	shape probabilities: joint exceedances inside each rooted connected
	shape and nonexceedances on its exterior neighbor set. We then justify
	the empirical interpretation through root counts in expanding windows,
	including the effect of discarding clusters near the observation boundary.
	The value of this representation is that the same complete-component
	events support several tasks: changing the sampling rule, estimating
	intensities without enumerating every shape, certifying omitted probability
	mass, and computing Gaussian rates that account for the exterior boundary.
	The subsequent sections develop these consequences.
	
	\subsection{General $d$-dimensional formulation}
	
	Let $\{X_t:t\in\Z^d\}$ be a real-valued stationary random field, let
	$u\in\R$ be fixed, and suppose Assumption~\ref{ass:finite-cluster} holds.  Recall
	that $E_u=\{t\in\Z^d:X_t>u\}$ and that $C_u(t)$ denotes the connected component of $E_u$ containing $t$,
	with $C_u(t)=\varnothing$ when $t\notin E_u$.
	
	\subsubsection{Connectivity, roots, and rooted cluster shapes}
	
	Denote the origin by $o=(0,\ldots,0)$, and let
	$e_1=(1,0,\ldots,0),\ldots,e_d=(0,\ldots,0,1)$ be the standard basis vectors.
	We consider either of the following two neighborhood systems.  The
	\emph{nearest neighbors} of $t\in\Z^d$ are
	\begin{equation*}
		\{t'\in\Z^d:\|t'-t\|=1\}
		=\{t\pm e_i:i=1,\ldots,d\},
	\end{equation*}
	so each site has $2d$ nearest neighbors.  The \emph{Moore neighbors} of $t$ are
	\begin{equation*}
		\{t'\in\Z^d:\|t'-t\|_\infty=1\},
	\end{equation*}
	where $\|(t_1,\ldots,t_d)\|_\infty=\max_{1\le i\le d}|t_i|$.  Thus, the Moore
	neighborhood also includes diagonal sites and contains $3^d-1$ sites. We fix
	one of these two neighborhood systems throughout each argument.

	For a finite set $D\subset\Z^d$, let $\mathcal{N}(D)$ denote its
	\emph{exterior neighbor set}:
	\begin{equation}\label{eq:neighbor}
		\mathcal{N}(D)=
		\begin{cases}
			\{t'\in\Z^d\setminus D:\|t'-t\|=1
			\text{ for some }t\in D\}, & \text{nearest neighbors},\\
			\{t'\in\Z^d\setminus D:\|t'-t\|_\infty=1
			\text{ for some }t\in D\}, & \text{Moore neighbors}.
		\end{cases}
	\end{equation}
	For a singleton, we write $\mathcal{N}(t)=\mathcal{N}(\{t\})$.  We also use
	$X_D>u$ to mean $X_t>u$ for every $t\in D$, and define $X_D\le u$
	analogously.
	
	To select one representative location for each cluster, introduce the
	lexicographic order on $\Z^d$.  For distinct
	$t=(t_1,\ldots,t_d)$ and $t'=(t'_1,\ldots,t'_d)$, write
	\[
	t<_{\rm lex}t'
	\quad\Longleftrightarrow\quad
	\text{for some $j$, $t_j<t'_j$ and $t_i=t'_i$ for every $i<j$}.
	\]
	Thus the first coordinate where the two vectors differ determines the ordering. For example, in $\Z^2$,
	\[
	(0,0)<_{\rm lex}(0,1)<_{\rm lex}(1,-1)<_{\rm lex}(1,0)<_{\rm lex}(1,1).
	\]
	Every finite nonempty subset of $\Z^d$ has a unique lexicographically smallest
	point.  We call this point the \emph{root} of the cluster.  For every
	$t\in\Z^d$, define the root event
	\begin{equation}\label{eq:translated-root-event}
		R_u(t)
		=
		\{X_t>u,\ t<_{\rm lex}s
		\text{ for every }s\in C_u(t)\setminus\{t\}\}.
	\end{equation}
	Thus, $R_u(t)$ is the event that $t$ is the root of its excursion cluster.
	In particular, denote by $R_u(o)$ the event that the origin is a cluster root and let $\rho_u:=\P(R_u(o))$ be its probability. The lexicographic order is translation invariant, so the root counting
	measure
	\begin{equation*}
		\Xi_u^{\rm root}
		:=\sum_{t\in\Z^d}\mathbbm{1}\{R_u(t)\}\,\delta_t
	\end{equation*}
	is stationary whenever $X$ is stationary.  Its intensity is
	$\rho_u$.

	For $k\ge1$, define the deterministic collection of size-$k$ connected shapes rooted at $o$ by
	\begin{equation*}
		\mathcal{C}_k^{\rm root}
		=
		\left\{
		D\subset\Z^d:
		|D|=k,\ D\text{ is connected with root at } o
		\right\}.
	\end{equation*}
	
	\begin{figure}[t]
		\begin{center}
			\begin{tabular}{c}
				\includegraphics[trim=8 8 25 5,clip,width=5in]{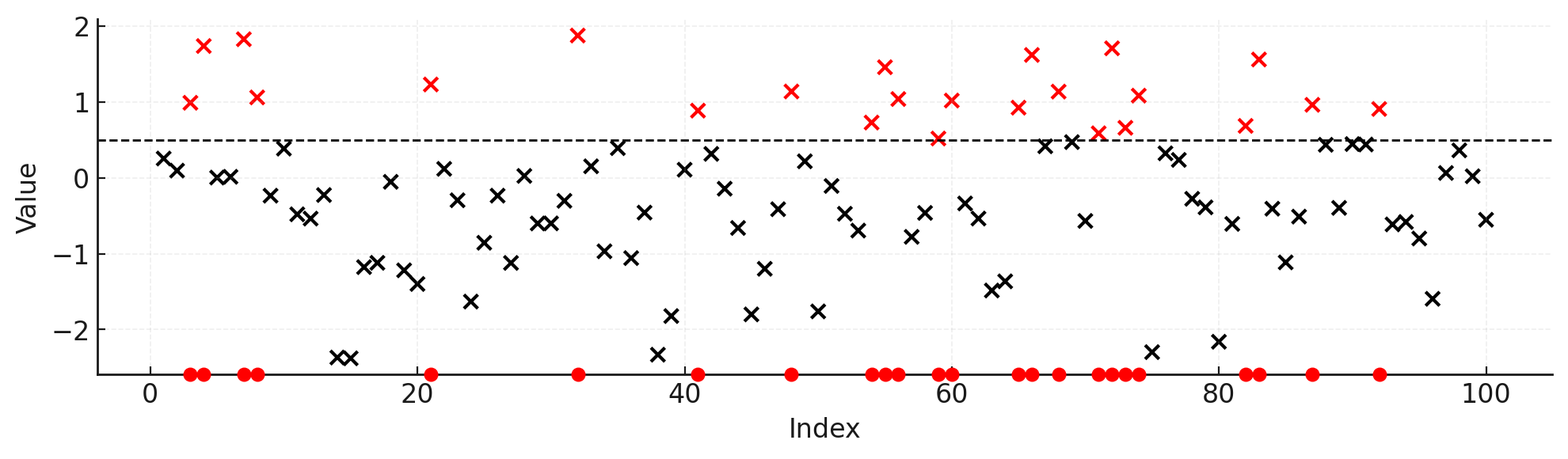}\\[.5em]
				{(a) Simulated Gaussian process and clusters above $u=0.5$}
				\vspace{4mm}\\
				\includegraphics[trim=5 10 50 5,clip,width=3in]
				{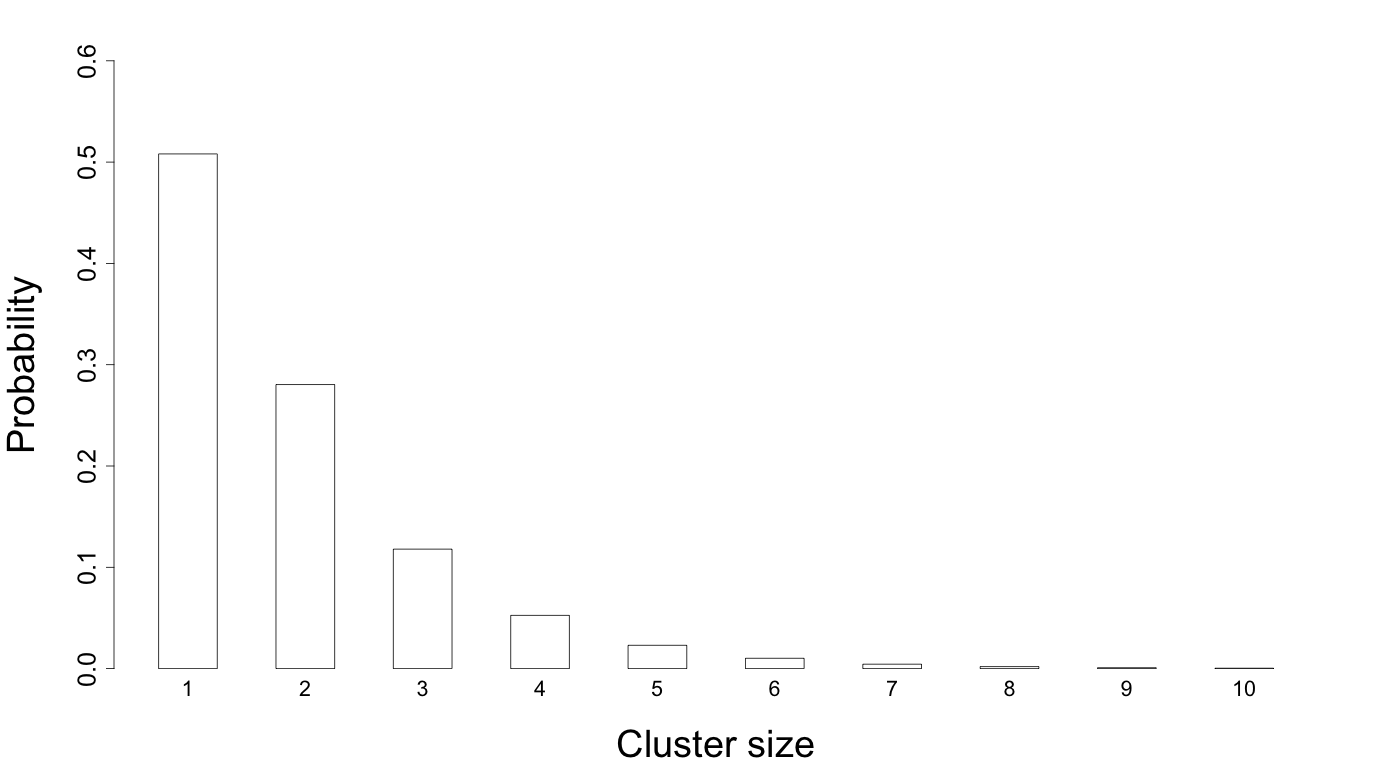}\\[.5em]
				{(b) Empirical cluster size distribution}
			\end{tabular}
			\caption{\label{fig:simul-1D}
				Panel (a) is a scatter plot of a simulated realization of a centered
				stationary Gaussian process with covariance
				$C(t,s)=e^{-(t-s)^2}$.  Exceedances above $u=0.5$ and the corresponding
				clusters are marked by consecutive red dots along $\Z$.  This realization
				has 14 observed clusters: 7 of size one, 5 of size two, 1 of size three,
				and 1 of size four.  Panel (b) displays the empirical cluster size
				distribution obtained from repeated simulations.}
		\end{center}
	\end{figure}
	
	We present some examples of rooted shapes. In one dimension, the nearest and Moore neighbors coincide. For every $k\geq1$, there is exactly one rooted connected shape,
	\begin{equation}\label{eq:root-shapes-1d}
		D_k=\{0,1,\ldots,k-1\},
		\qquad
		\mathcal C_k^{\rm root}=\{D_k\},
		\qquad
		\mathcal N(D_k)=\{-1,k\}.
	\end{equation}
	In two dimensions with nearest-neighbor connectivity, the first three rooted
	collections are
	\begin{equation}\label{eq:root3}
		\begin{split}
			\mathcal{C}_1^{\rm root}
			&=\bigl\{\{o\}\bigr\},\qquad
			\mathcal{C}_2^{\rm root}
			=\bigl\{\{o,e_1\},\{o,e_2\}\bigr\},\\
			\mathcal{C}_3^{\rm root}
			&=\bigl\{
			\{o,e_1,2e_1\},
			\{o,e_1,e_1+e_2\},
			\{o,e_1,e_1-e_2\},\\
			&\hspace{18mm}
			\{o,e_1,e_2\},
			\{o,e_2,2e_2\},
			\{o,e_2,e_1+e_2\}
			\bigr\}.
		\end{split}
	\end{equation}
	Under Moore connectivity, $\mathcal C_2^{\rm root}$ additionally contains
	$\{o,e_1+e_2\}$ and $\{o,e_1-e_2\}$, and
	$\mathcal C_3^{\rm root}$ contains further shapes because diagonal sites are
	also adjacent. For $D\in\mathcal{C}_k^{\rm root}$, the event that $D$ is exactly the excursion cluster rooted at $o$ is
	\begin{equation}\label{eq:C_u(o)}
		\{C_u(o)=D\}=\{X_D>u,\ X_{\mathcal{N}(D)}\le u\}.
	\end{equation}
	
	A typical excursion cluster is sampled by assigning one representative
	to each finite component. We choose the lexicographically smallest site
	of each cluster as its root, so every cluster contributes exactly one
	point to the root process. Since lexicographic order is preserved by
	translations, the root selection rule is translation equivariant.
	Consequently, for a stationary field, the root process is stationary,
	and its Palm law is obtained by conditioning on a root at the origin.
	This gives the following precise definition.
	
	\begin{definition}[Root cluster law]\label{def:cluster-size} 
		Let $X$ be stationary, suppose Assumption~\ref{ass:finite-cluster} holds,
		and assume $\rho_u=\P(R_u(o))>0$. The typical cluster size $S_u$ has distribution
		\begin{equation}\label{eq:CSD_def}
			\P(S_u=k) = \P\bigl(|C_u(o)|=k\mid R_u(o)\bigr).
		\end{equation}
		Thus, $S_u$ is the size of the cluster rooted at the origin under the
		Palm law of the root process.
	\end{definition}
	By stationarity, Definition~\ref{def:cluster-size} does not depend on a
	particular lattice location.  It gives one vote to every cluster, through its
	unique root; under ergodicity, this is precisely the limiting empirical law
	obtained by sampling clusters in expanding windows. Figures \ref{fig:simul-1D} and \ref{fig:simul-2D} provide visual illustrations for cases $d=1$ and $d=2$, respectively. 
	
	\begin{figure}[t]
		\begin{center}
			\begin{tabular}{cc}
				\includegraphics[trim=20 20 0 0,clip,width=1.5in]{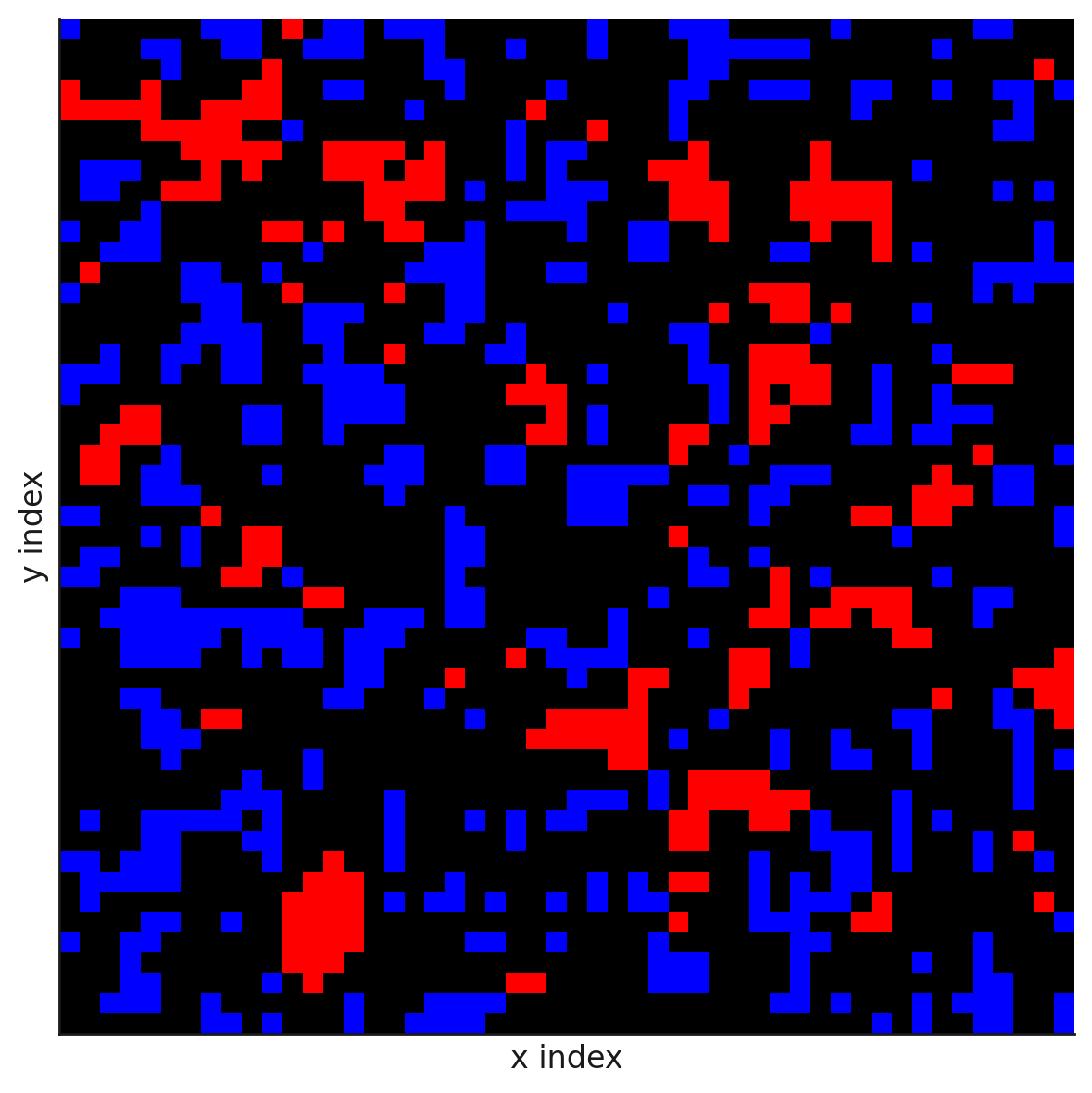} &
				\includegraphics[trim=0 0 0 0,clip,width=3in]
				{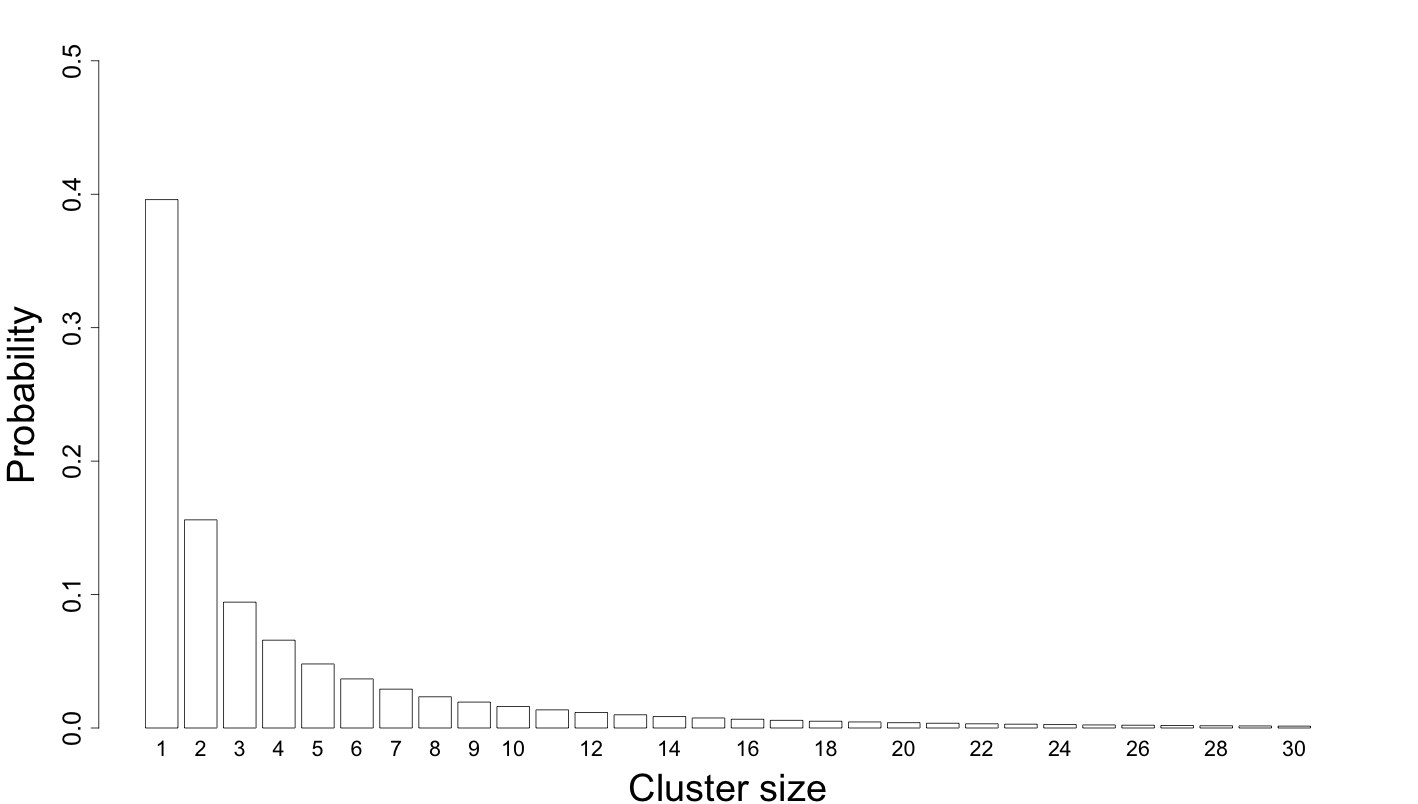}\\[.5em]
				(a) Clusters (nearest neighbors) &
				(b) Empirical distribution (nearest neighbors)
			\end{tabular}
			
			\vspace{3.9mm}
			
			\begin{tabular}{cc}
				\includegraphics[trim=20 20 0 0,clip,width=1.5in]{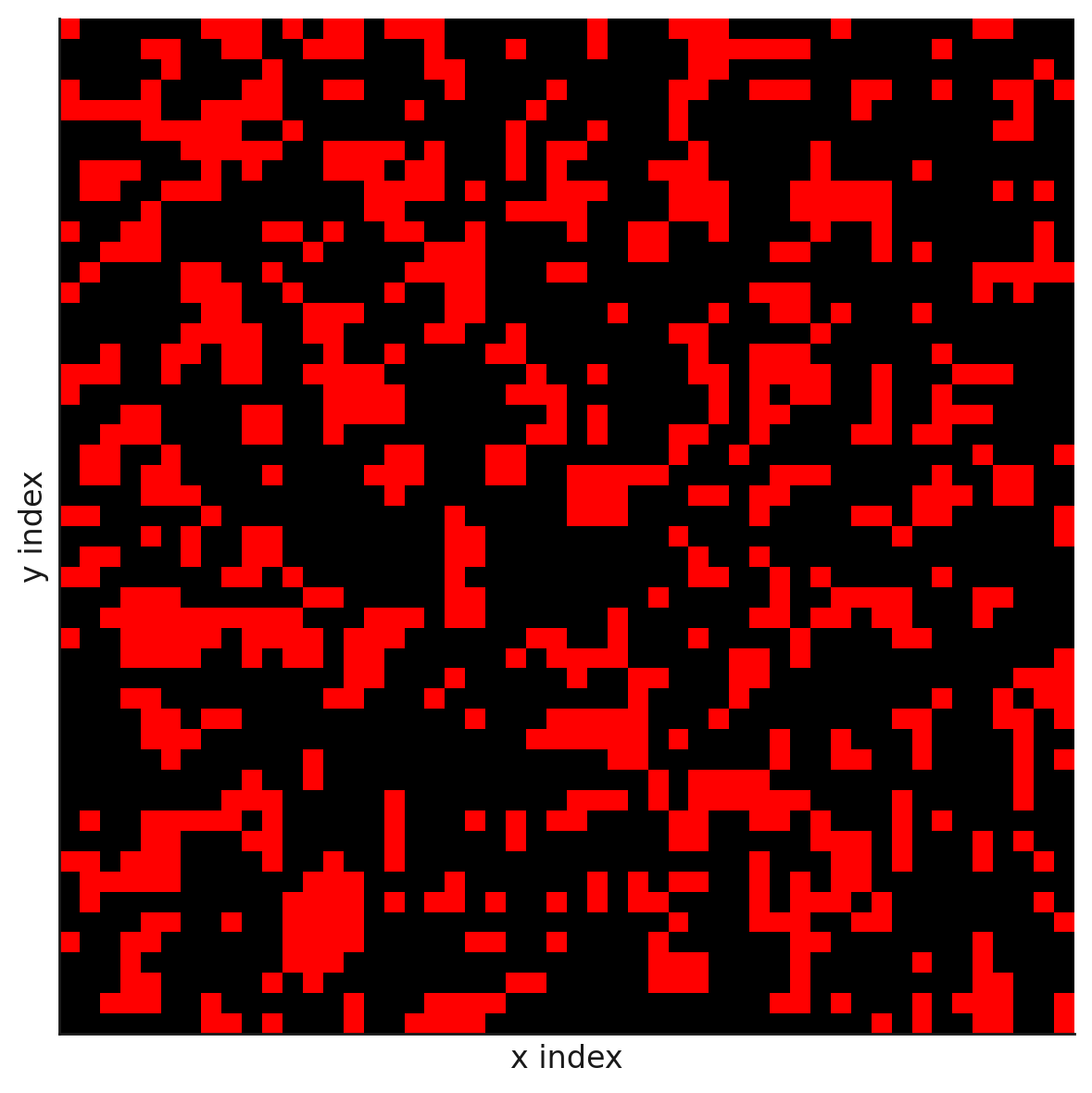} &
				\includegraphics[trim=0 0 0 0,clip,width=3in]
				{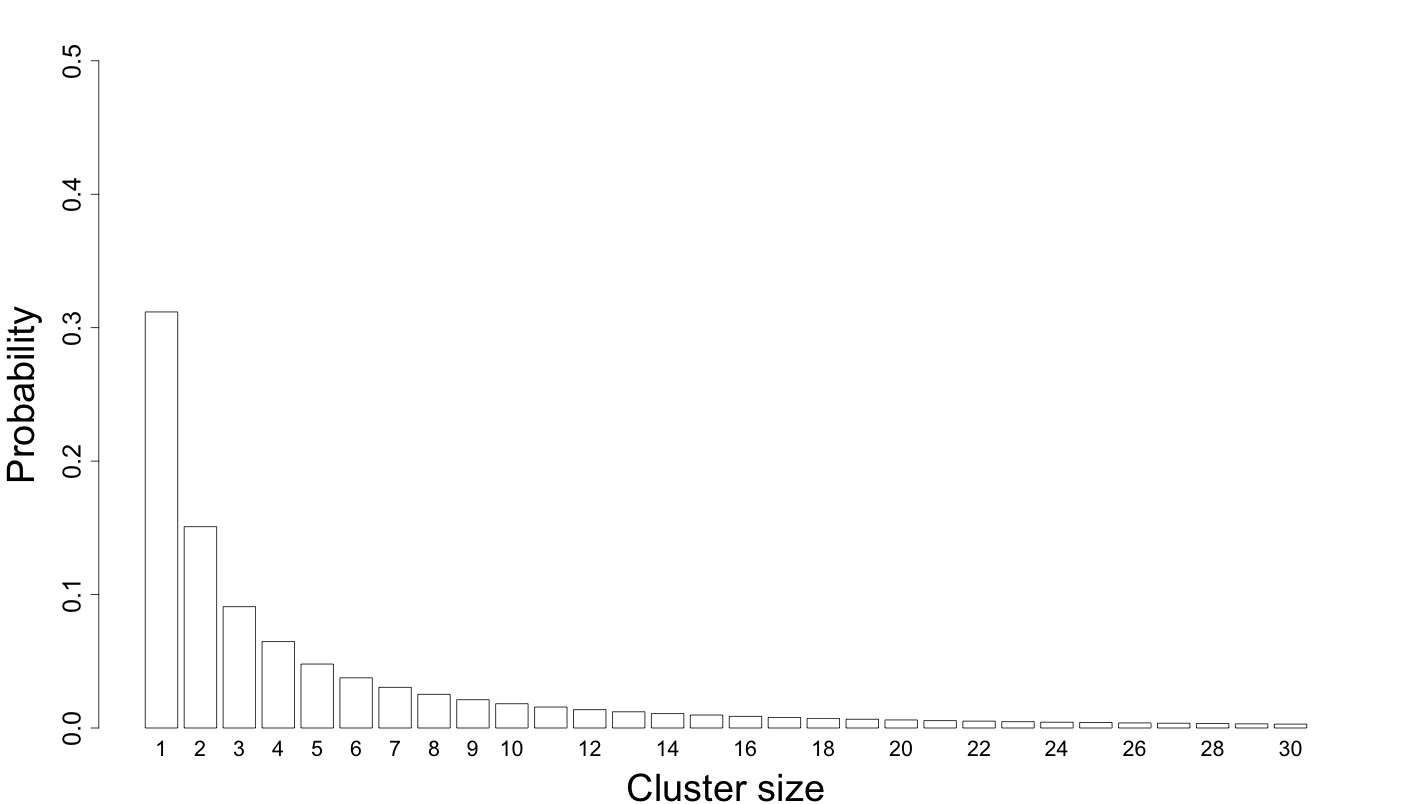}\\[.5em]
				(c) Clusters (Moore neighbors) &
				(d) Empirical distribution (Moore neighbors)
			\end{tabular}
			\caption{\label{fig:simul-2D}
				Clusters and empirical cluster size distributions for an isotropic
				Gaussian field on $\Z^2$ with covariance
				$C(t,s)=e^{-\|t-s\|^2}$ and threshold $u=0.5$.  Panels (a) and (b)
				use nearest-neighbor connectivity; panels (c) and (d) use Moore
				connectivity.  Distinct clusters above $u$ are highlighted in panels
				(a) and (c), and panels (b) and (d) show the corresponding empirical
				distributions from repeated simulations.}
		\end{center}
	\end{figure}
	
	\subsubsection{Exact root cluster distribution in $\Z^d$}
	
	Let $w_k$ denote the probability that there exists a size-$k$ cluster above $u$ with root at $o$, that is,
	\begin{equation}\label{eq:w_d}
		\begin{split}
			w_k
			&:=\sum_{D\in\mathcal{C}_k^{\rm root}}
			\P\bigl(X_D>u,\ X_{\mathcal{N}(D)}\le u\bigr),
			\qquad k=1,2,\ldots.
		\end{split}
	\end{equation}
	The quantity $w_k$ is the intensity, or expected number per lattice site, of
	size-$k$ clusters when each cluster is counted once at its root. For fixed $k$, the events
	\[
	\{X_D>u,\ X_{\mathcal{N}(D)}\le u\},
	\qquad D\in\mathcal{C}_k^{\rm root},
	\]
	are pairwise disjoint, and their union is
	$\{R_u(o),\ |C_u(o)|=k\}$.  Hence we have
	\begin{equation}\label{eq:w_d_2}
		\begin{split}
			w_k =\P\bigl(R_u(o),\ |C_u(o)|=k\bigr),
			\qquad k=1,2,\ldots.
		\end{split}
	\end{equation}
	
	\begin{theorem}[Exact root cluster size distribution]\label{thm:CSD_d}
		Let $\{X_t:t\in\Z^d\}$ be stationary, let $u\in\R$ satisfy
		Assumption~\ref{ass:finite-cluster}, and suppose
		$\rho_u=\P(R_u(o))>0$.  Then
		\begin{equation}\label{eq:CSD_d}
			\P(S_u=k)=\frac{w_k}{\rho_u},
			\qquad
			\rho_u=\sum_{j\geq1}w_j,
			\qquad k\geq1,
		\end{equation}
		where $w_k$ is given by \eqref{eq:w_d}.
	\end{theorem}
	
	\begin{proof}
		The definition \eqref{eq:CSD_def} and 
		\eqref{eq:w_d_2} give
		\[
		\P(S_u=k)
		=\frac{\P(R_u(o),|C_u(o)|=k)}{\P(R_u(o))}
		=\frac{w_k}{\rho_u}.
		\]
		By Assumption~\ref{ass:finite-cluster}, on $R_u(o)$ the cluster $C_u(o)$ has some finite size.  Hence
		\[
		R_u(o)=\bigsqcup_{j\ge1}
		\{R_u(o),\ |C_u(o)|=j\},
		\]
		together with \eqref{eq:w_d_2}, we obtain
		\[
		\rho_u = \P(R_u(o)) = \sum_{j=1}^\infty \P\bigl(R_u(o),\ |C_u(o)|=j\bigr) = \sum_{j=1}^\infty w_j,
		\]
		completing the proof.
	\end{proof}
	
	Combining \eqref{eq:w_d} with \eqref{eq:CSD_d} gives the explicit
	representation
	\begin{equation}\label{eq:root-shape-size-law}
		\P(S_u=k)=
		\frac{\displaystyle\sum_{D\in\mathcal C_k^{\rm root}}
			\P\bigl(X_D>u,\ X_{\mathcal N(D)}\le u\bigr)}
		{\displaystyle\sum_{j\geq1}\sum_{D\in\mathcal C_j^{\rm root}}
			\P\bigl(X_D>u,\ X_{\mathcal N(D)}\le u\bigr)},
		\qquad k\geq1.
	\end{equation}
	The standard Palm normalization specifies how the shape intensities are
	converted into the typical cluster law. The complete shape expression
	identifies those intensities directly from the finite dimensional joint
	distributions of the field. For each fixed $k$, the numerator is a finite
	sum, and its exterior nonexceedance constraints certify that no adjacent
	site extends the component. Thus the formula applies at a fixed threshold
	to stationary fields with arbitrary dependence, subject to the assumptions
	of Theorem~\ref{thm:CSD_d}. 
	
	\begin{remark}[Root law among finite clusters]
		\label{remark:root-finite-clusters}
		The representation in Theorem~\ref{thm:CSD_d} also has an
		interpretation without Assumption~\ref{ass:finite-cluster}.
		Retain stationarity, but allow infinite excursion components.
		Assign a lexicographic root to each finite cluster and define
		\[
		R_u^{\rm fin}(t)
		:=
		\left\{
		0<|C_u(t)|<\infty,\;
		t=\min_{\rm lex}C_u(t)
		\right\}.
		\]
		This root selection is preserved by translations, so the
		counting process of roots of finite clusters is stationary.
		
		For every finite $k$, the complete shape probabilities in
		\eqref{eq:w_d} remain valid and satisfy
		\[
		w_k
		=
		\P\bigl(R_u^{\rm fin}(o),\,|C_u(o)|=k\bigr).
		\]
		Indeed, each complete shape event identifies a finite component,
		regardless of whether infinite components exist elsewhere.
		Taking the disjoint union over all finite sizes gives
		\[
		\rho_u^{\rm fin}
		:=
		\P(R_u^{\rm fin}(o))
		=
		\sum_{j\geq1}w_j.
		\]
		Whenever $\rho_u^{\rm fin}>0$, define
		\[
		\P(S_u^{\rm fin}=k)
		:=
		\P\bigl(|C_u(o)|=k\mid R_u^{\rm fin}(o)\bigr)
		=
		\frac{w_k}{\rho_u^{\rm fin}},
		\qquad k\geq1.
		\]
		This is the root Palm size distribution among finite clusters.
		Every finite cluster contributes one root, infinite components
		are excluded, and the probabilities sum to one over the positive
		integers. Under Assumption~\ref{ass:finite-cluster}, this law
		coincides with the root cluster law in
		Theorem~\ref{thm:CSD_d}.
	\end{remark}
	
	\subsection{Consistency of empirical cluster frequencies}
	\label{sec:empirical-cluster-counts}
	We now demonstrate that Theorem~\ref{thm:CSD_d} is consistent with the empirical cluster size distribution computed in simulations when expanding observation windows.
	
	Recall from \eqref{eq:translated-root-event} that
	$R_u(t)$ is the event that $t$ is a cluster root.  Define
	\begin{equation}\label{eq:empirical-root-indicators}
		Y_{t,k}=\mathbbm{1}\{R_u(t),|C_u(t)|=k\},
		\qquad
		Y_t=\mathbbm{1}\{R_u(t)\}.
	\end{equation}
	The indicator $Y_{t,k}$ contributes one
	precisely when $t$ is the unique root of a full-lattice excursion cluster of
	size $k$.  For a finite window $\Lambda$, the \textit{ideal root counts} are
	\begin{equation}\label{eq:ideal-root-counts}
		N_k^*(\Lambda)=\sum_{t\in\Lambda}Y_{t,k},
		\qquad
		N^*(\Lambda)=\sum_{t\in\Lambda}Y_t.
	\end{equation}
	Since every finite cluster has exactly one root,
	$N_k^*(\Lambda)$ therefore counts the size-$k$ full-lattice clusters whose
	roots lie in $\Lambda$ (part of the cluster may lie outside $\Lambda$), with each such cluster counted exactly once. Similarly, $N^*(\Lambda)$ is the total number of full-lattice excursion clusters, of all sizes, whose roots lie in $\Lambda$. Under Assumption~\ref{ass:finite-cluster}, one has $N^*(\Lambda)=\sum_{k=1}^{\infty}N_k^*(\Lambda)$.

	We next define the clusters that can be identified using only the observed
	values $\{X_t:t\in\Lambda\}$. First, let $\mathscr{C}_u(\Lambda)$ denote the collection of observed clusters restricted on $\Lambda$:
	\begin{equation}\label{eq:clusters-restriction}
		\mathscr{C}_u(\Lambda)
		=
		\{\text{\rm connected components of } E_u\cap\Lambda\}.
	\end{equation}
	Thus, every $D\in \mathscr{C}_u(\Lambda)$ satisfies
	$D\subset\Lambda$ and is connected through lattice sites in $\Lambda$, although it may be the observed portion of a cluster
	that continues outside $\Lambda$.
	
	Second, call an observed component $D$ an \emph{interior cluster} (with
	respect to $\Lambda$) if its entire exterior neighbor set is contained in the observation window $\Lambda$. The collection of interior clusters is
	\begin{equation}\label{eq:observable-interior-components}
		\mathscr{C}_u^\circ(\Lambda)
		=
		\{D\in \mathscr{C}_u(\Lambda):\mathcal N(D)\subset\Lambda\}.
	\end{equation}
	This interior condition makes $D$ a genuine full-lattice excursion cluster.
	Indeed, $\mathcal N(D)\subset\Lambda$ implies that every
	exterior neighbor of $D$ is observed and is at or below $u$.  Hence the
	cluster cannot continue outside $D$. Define the \textit{observable counts} or \textit{empirical counts}
	\begin{equation}\label{eq:observable-cluster-counts}
		\begin{split}
			N_k^\circ(\Lambda)
			&=\sum_{D\in \mathscr{C}_u^\circ(\Lambda)}
			\mathbbm{1}\{|D|=k\}=\#\{\text{size-$k$ interior clusters in $\Lambda$}\},\\
			\qquad
			N^\circ(\Lambda)&=| \mathscr{C}_u^\circ(\Lambda)|=\#\{\text{all interior clusters in $\Lambda$}\}.
		\end{split}
	\end{equation}
	
	Let $d_G$ be the graph distance associated with the neighborhood
	system. For an integer $r\geq0$, write
	\begin{equation*}
		\partial_r\Lambda
		=\{t\in\Lambda:d_G(t,\Lambda^c)\leq r\}.
	\end{equation*}
	We use nested rectangular windows centered at the origin of the form
	\begin{equation}\label{eq:empirical-rectangular-windows}
		\Lambda_n
		=\prod_{i=1}^d\{-m_{i,n},\ldots,m_{i,n}\},
		\qquad m_{i,n}\uparrow\infty\quad(i=1,\ldots,d).
	\end{equation}
	In particular,
	for each fixed $r$, $|\partial_r\Lambda_n|/|\Lambda_n| \to 0$. The following result, whose proof is given in the Appendix, shows the consistency of empirical cluster frequencies.
	\begin{theorem}[Empirical representation and convergence]
		\label{thm:empirical-cluster-counts}
		Let $\{X_t:t\in\Z^d\}$ be stationary and ergodic, let
		Assumption~\ref{ass:finite-cluster} hold, and suppose
		$\rho_u>0$. Then, for fixed $k\geq1$,
		\begin{align}
			\frac{N_k^*(\Lambda_n)}{|\Lambda_n|}
			&\overset{\rm a.s.}{\longrightarrow}w_k,
			&
			\frac{N^*(\Lambda_n)}{|\Lambda_n|}
			&\overset{\rm a.s.}{\longrightarrow}\rho_u,
			\label{eq:ideal-count-limits}\\
			\widetilde w_{k,n}
			:=\frac{N_k^\circ(\Lambda_n)}{|\Lambda_n|}
			&\overset{\rm a.s.}{\longrightarrow}w_k,
			&
			\widetilde\rho_{u,n}
			:=\frac{N^\circ(\Lambda_n)}{|\Lambda_n|}
			&\overset{\rm a.s.}{\longrightarrow}\rho_u.
			\label{eq:w-empirical}
		\end{align}
		Moreover, $N^\circ(\Lambda_n)>0$ eventually almost surely, and
		\begin{equation}\label{eq:empirical}
			\frac{N_k^\circ(\Lambda_n)}{N^\circ(\Lambda_n)}
			=\frac{\widetilde w_{k,n}}{\widetilde\rho_{u,n}}
			\overset{\rm a.s.}{\longrightarrow}
			\frac{w_k}{\rho_u}
			=\P(S_u=k).
		\end{equation}
		In fact, the whole empirical probability mass function is consistent in
		total variation:
		\begin{equation}\label{eq:empirical-total-variation}
			\frac{1}{2}\sum_{k=1}^{\infty}
			\left|
			\frac{N_k^\circ(\Lambda_n)}{N^\circ(\Lambda_n)}
			-\P(S_u=k)
			\right|
			\overset{\rm a.s.}{\longrightarrow}0.
		\end{equation}
		For every finite $\Lambda$, the ideal counts also satisfy
		\begin{equation}\label{eq:ideal-count-expectations}
			\E N_k^*(\Lambda)=|\Lambda|w_k,
			\qquad
			\E N^*(\Lambda)=|\Lambda|\rho_u.
		\end{equation}
	\end{theorem}

	The convergence in \eqref{eq:empirical}, together with \eqref{eq:observable-cluster-counts}, indicates that, as the window grows, the relative frequency of clusters of size $k$ converges to the theoretical cluster size probability in Theorem~\ref{thm:CSD_d}. This verifies that our theoretical formula agrees with the empirical
	cluster size distribution computed in simulations.

	\subsection{Examples}
	
	\subsubsection{The one-dimensional specialization}
	
	Let $d=1$.  The nearest and Moore neighborhoods coincide, and the
	lexicographic order is the usual order on $\Z$.  As shown in
	\eqref{eq:root-shapes-1d}, $D_k=\{0,1,\ldots,k-1\}$ is the unique rooted
	connected shape of size $k$, and its exterior neighbor set is
	$\{-1,k\}$.  Therefore, the general weight \eqref{eq:w_d} reduces to
	\begin{equation}\label{eq:w_1D}
		\begin{split}
			w_k
			&=\P\bigl(X_{-1}\le u,\ X_0>u,\ X_1>u,\ldots,
			X_{k-1}>u,\ X_k\le u\bigr).
		\end{split}
	\end{equation}
	Moreover, the root event is simply
	$R_u(o)=\{X_{-1}\le u,X_0>u\}$.  Hence
	\begin{equation}\label{eq:w_1D_sum}
		\sum_{k=1}^{\infty}w_k
		=\P(X_{-1}\le u,X_0>u).
	\end{equation}

	Theorem~\ref{thm:CSD_d} now immediately gives the following one-dimensional
	formula.
	
	\begin{corollary}[One-dimensional specialization]\label{cor:CSD_1D}
		Let $\{X_t:t\in\Z\}$ be stationary, and let $u\in\R$ satisfy
		Assumption~\ref{ass:finite-cluster}.  If
		$\P(X_{-1}\le u,X_0>u)>0$, then, for $k=1,2,\ldots$,
		\begin{equation}\label{eq:CSD_1D}
			\P(S_u=k)
			=
			\frac{
				\P\bigl(X_{-1}\le u,\ X_0>u,\ X_1>u,\ldots,
				X_{k-1}>u,\ X_k\le u\bigr)}
			{\P(X_{-1}\le u,X_0>u)}.
		\end{equation}
	\end{corollary}
	
	\begin{proof}
		Substitute \eqref{eq:w_1D} and \eqref{eq:w_1D_sum} into the general formula
		\eqref{eq:CSD_d}.
	\end{proof}
	
	Thus, in one dimension the cluster size is the number of consecutive
	exceedances beginning at a transition from below to above the threshold.
	Figure~\ref{fig:simul-1D} provides a visual illustration, while
	\eqref{eq:empirical} gives the corresponding empirical relative frequency interpretation.
	
	\begin{example}[One-dimensional white noise]\label{example:WN_1D}
		Let $X_t$, $t\in\Z$, be an independent and identically distributed process
		with common cumulative distribution function $F$.  Set
		$p=F(u)$ and $q=1-p$, and assume $0<p<1$.  By
		\eqref{eq:CSD_1D},
		\begin{equation}\label{eq:WN_1D}
			\P(S_u=k)=\frac{p^2q^k}{pq}=pq^{k-1},
			\qquad k=1,2,\ldots.
		\end{equation}
		Hence the cluster size has a geometric distribution with success parameter
		$p$.
	\end{example}
	
	\subsubsection{A two-dimensional illustration}
	
	To illustrate the shape sum in \eqref{eq:w_d}, consider nearest-neighbor
	connectivity on $\Z^2$.  The rooted collections for $k\leq3$ are displayed
	immediately after the general definition in \eqref{eq:root3}.  They show that
	even for small $k$ several distinct rooted shapes may contribute to $w_k$;
	under Moore connectivity, diagonal adjacency produces still more shapes.
	
	Figure~\ref{fig:simul-2D} shows clusters and empirical cluster size
	distributions for a stationary Gaussian field on $\Z^2$ under both
	connectivity rules.  Because Moore connectivity includes diagonal sites, it
	generally produces larger connected components than nearest-neighbor
	connectivity.

	\begin{example}[Two-dimensional white noise]\label{example:WNF}
		Let $X_t$, $t\in\Z^2$, be an independent and identically distributed field
		with common CDF $F$, and set $p=F(u)$ and $q=1-p$.  Inspecting the rooted
		shapes in \eqref{eq:root3} and their exterior boundaries gives, under
		nearest-neighbor connectivity,
		\[
		w_1=p^4q,\qquad
		w_2=2p^6q^2,\qquad
		w_3=(2p^8+4p^7)q^3.
		\]
		Under Moore connectivity,
		\[
		w_1=p^8q,\qquad
		w_2=(2p^{10}+2p^{12})q^2,\qquad
		w_3=(6p^{12}+8p^{14}+4p^{15}+2p^{16})q^3.
		\]
		In general, the combinatorial complexity of
		$\mathcal{C}_k^{\rm root}$ grows rapidly with $k$.  Thus, closed-form
		expressions for $w_k$ are rarely available, although $w_k$ can be evaluated
		by enumerating shapes for moderate $k$ or approximated by Monte Carlo.
	\end{example}
	
	\begin{remark}[Computation of $w_k$]
		For large $k$, the probabilities $w_k$ in \eqref{eq:w_d} are difficult to
		compute analytically, and the normalizing constant
		$\sum_{j\ge1}w_j$ in \eqref{eq:CSD_d} is rarely available in closed form.
		The Monte Carlo estimators in Section \ref{sec:MC} provide alternatives.
		
		If the random field is isotropic, symmetry can substantially reduce the
		computational burden.  The probability
		$\P(X_D>u,X_{\mathcal{N}(D)}\le u)$ is the same for shapes related by a
		lattice rotation or reflection, so one may compute a single representative
		probability and multiply by its multiplicity.  For example, under
		nearest-neighbor connectivity on $\Z^2$, the two elements of
		$\mathcal{C}_2^{\rm root}$ are two orientations of the same geometric
		shape, while the six elements of $\mathcal{C}_3^{\rm root}$ in
		\eqref{eq:root3} reduce to two geometric types: straight triominoes and
		$L$-triominoes.
	\end{remark}
	
	\subsection{Site-sampled clusters and bias}
	
	For $k\ge1$, let
	\[
	\mathcal{C}_k^{\rm site}
	=\{D\subset\Z^d:|D|=k,\ D\text{ is connected, and }o\in D\}.
	\]
	Here the origin may occupy any one of the $k$ sites of the cluster.  Translating
	each rooted shape so that each of its sites is placed at $o$ gives $|\mathcal{C}_k^{\rm site}|=k|\mathcal{C}_k^{\rm root}|$.
	Let
	\begin{equation}\label{eq:w-inside}
		\begin{split}
			w_k^{\rm site}
			&:=\P\bigl(X_o>u,\ |C_u(o)|=k\bigr)=\sum_{D\in\mathcal{C}_k^{\rm site}}
			\P\bigl(X_D>u,X_{\mathcal{N}(D)}\le u\bigr).
		\end{split}
	\end{equation}
	By stationarity, each cluster of size $k$ contributes once to the rooted count
	and $k$ times to the site-containing count.  Hence
	\begin{equation}\label{eq:w-inside-w}
		w_k^{\rm site}=kw_k.
	\end{equation}
	
	\begin{corollary}\label{cor:CSD_d}
		Under the assumptions of Theorem~\ref{thm:CSD_d},
		\[
		\P(S_u=k)
		=
		\frac{w_k^{\rm site}/k}
		{\sum_{j=1}^{\infty}w_j^{\rm site}/j},
		\qquad k=1,2,\ldots,
		\]
		where $w_k^{\rm site}$ is defined in \eqref{eq:w-inside}.
	\end{corollary}
	
	Let $\rho_u^{\rm site}:=\P(X_o>u)$ denote the probability that a lattice site exceeds the threshold $u$. By \eqref{eq:w-inside} and \eqref{eq:w-inside-w},
	\[
	\rho_u^{\rm site} = \sum_{j\ge 1} w_j^{\rm site} = \sum_{j\ge 1}jw_j.
	\]
	Since $\P(S_u=j)=w_j/\rho_u$, the mean size under the typical root cluster
	law satisfies $\E S_u=\sum_{j\geq1}jw_j/\rho_u=\rho_u^{\rm site}/\rho_u$. The size of the cluster containing an exceedance site, denoted $S_u^{\rm site}$, is described by conditioning on $X_o>u$. Its probability mass function is
	\begin{equation*}
		\begin{aligned}
			\P(S_u^{\rm site}=k)
			&:=\P\bigl(|C_u(o)|=k\mid X_o>u\bigr)\\
			&=\frac{k w_k}{\rho_u^{\rm site}}
			=\frac{k}{\E S_u}\P(S_u=k),
			\qquad k=1,2,\ldots.
		\end{aligned}
	\end{equation*}
	
	Thus, site sampling weights each cluster in proportion to its size:
	a cluster of size $k$ can be selected through any of its $k$ exceedance
	sites. Sampling through cluster roots instead gives every cluster one
	representative, regardless of its size. Corollary~\ref{cor:CSD_d} corrects this size bias to recover the typical root cluster law from the site law.

	The relation \eqref{eq:w-inside-w} also leads to unbiased site-sampled
	spatial estimators of $w_k$. We develop their exact algorithms,
	finite-sample unbiasedness, control of the omitted normalization tail, and
	estimation of the spatial Monte Carlo variance in Section~\ref{sec:MC}.
	
	\section{The peak-based cluster size distribution}\label{sec:peak-based-cluster}
	When the random field $X_t$ is nonstationary, the cluster size distribution in \eqref{eq:CSD_def} becomes challenging to analyze. This difficulty arises because, in the nonstationary setting, cluster laws may vary with location because the marginal distributions or dependence structure vary across the domain. A location-specific root conditioned law remains meaningful, but its normalizer generally involves connectivity beyond a fixed neighborhood. Peak conditioning provides a convenient alternative whose conditioning event and normalizer involve only the specified site and its neighbors.
	
	A second motivation comes from the fact that every cluster contains at least one local maximum (typically only a few), which can serve as a natural representative of the cluster. Moreover, for high thresholds $u$, clusters tend to be dominated by a single prominent peak. Since peak analysis and peak-over-threshold methods are well developed \cite{Cheng:2020,CS15,CS18,CS17,Fondeville:2018,Huser:2014}, it is natural and valuable to study the peak-based cluster size distribution from both theoretical and applied perspectives. Specifically, given that a local maximum exceeding $u$ is observed at $t$, we consider the distribution of the size of the cluster containing that peak. 
	
	\begin{definition}\label{def:peak-cluster-size} Let $\{X_t: t\in \Z^d\}$ be a discrete random field and suppose Assumption~\ref{ass:finite-cluster} holds. The distribution of the peak-based cluster size at $t\in \Z^d$, denoted $S_u^{\rm peak}(t)$, is defined as 
		\begin{equation}\label{eq:CSD_peak_1}
			\begin{split}
				\P(S_u^{\rm peak}(t)=k) = \P( |C_u(t)|=k \ | \ t \ \text{\rm is a local maximum and } X_t>u), \quad k\ge 1,
			\end{split}
		\end{equation}
		provided the conditional event has a positive probability. 
	\end{definition}
	
	The peak-based formulation is inherently local: unlike the rooted cluster distribution, its normalization depends only on the finite-dimensional distribution of the field in a neighborhood of the peak location. This locality makes the peak-based approach particularly suitable for nonstationary random fields.
	
	In signal detection problems, one is often interested in peaks that are not only high in amplitude but also substantial in spatial extent. This is particularly relevant in neuroimaging applications, where extended activation patterns are more scientifically interpretable than isolated spikes \cite{Chumbley:2010,Poline:1997,Worsley:1996a,Zhang:2009}. A peak-based formulation of the cluster size distribution provides a natural and efficient framework for quantifying such effects and assessing their significance, through $p$-values, when detecting spatially extended signals.

	\subsection{General $d$-dimensional formulation}
	Let $\{X_t:t\in\Z^d\}$ be a real-valued random field and fix $u\in\R$. For $t\in\Z^d$, define the peak-over-threshold event
	\begin{equation}\label{eq:peak-event}
		A_u^{\rm peak}(t)
		:=
		\left\{X_t>u,\
		X_t\ge \max_{s\in\mathcal N(t)}X_s\right\},
	\end{equation}
	and write
	\begin{equation}\label{eq:peak-intensity-local}
		\rho_u^{\rm peak}(t):=\P\bigl(A_u^{\rm peak}(t)\bigr).
	\end{equation}
	Thus, $A_u^{\rm peak}(t)$ occurs precisely when $t$ is a local maximum
	and exceeds $u$. The inequality in \eqref{eq:peak-event} defines a
	\emph{weak} local maximum (for simplicity, we omit ``weak'' throughout). This convention guarantees that every finite
	cluster contains a peak even when the field has atoms.  If ties occur with
	positive probability, however, a flat plateau may contribute several peaks.
	When $\P(X_t=X_s)=0$ for adjacent distinct sites $t,s$, such as for the
	nondegenerate Gaussian fields considered below, weak and strict local maxima
	agree almost surely. If the field is stationary, then $\rho_u^{\rm peak}(t)$ does not depend on $t$, and we simply write it as $\rho_u^{\rm peak}$.
	
	Suppose Assumption~\ref{ass:finite-cluster} holds and
	$\rho_u^{\rm peak}(t)>0$.  Then the peak-based cluster size distribution
	at $t$ in \eqref{eq:CSD_peak_1} can be written as
	\begin{equation}\label{eq:CSD_peak}
		\P\bigl(S_u^{\rm peak}(t)=k\bigr) =
		\P\bigl(|C_u(t)|=k\mid A_u^{\rm peak}(t)\bigr),
		\qquad k=1,2,\ldots.
	\end{equation}
	For a nonstationary field, this distribution may depend on $t$.  For a
	stationary field it does not, and we then write $S_u^{\rm peak}$ instead of
	$S_u^{\rm peak}(t)$.

	For $k\geq1$, define the deterministic collection of connected shapes of size
	$k$ containing the distinguished site $t$ by
	\begin{equation}\label{eq:C-peak}
		\mathcal C_k^{\rm peak}(t)
		:=
		\left\{
		D\subset\Z^d:
		|D|=k,\ D\text{ is connected, and }t\in D
		\right\}.
	\end{equation}
	The superscript ``peak'' records that $t$ is the candidate peak location; the
	shapes themselves are deterministic and do not depend on the field values.
	Translation gives
	\begin{equation}\label{eq:C-peak-cardinality}
		|\mathcal C_k^{\rm peak}(t)|
		=|\mathcal C_k^{\rm site}|
		=k|\mathcal C_k^{\rm root}|.
	\end{equation}
	
	For example, in $\Z^2$ with nearest-neighbor connectivity,
	\begin{equation*}
		\begin{split}
			\mathcal C_1^{\rm peak}(t) = \bigl\{\{t\}\bigr\}, \quad
			\mathcal C_2^{\rm peak}(t) = \bigl\{\{t,t+e_1\}, \{t,t-e_1\},
			\{t,t+e_2\}, \{t,t-e_2\}\bigr\}.
		\end{split}
	\end{equation*}
	There are $18$ shapes in $\mathcal C_3^{\rm peak}(t)$.  Under Moore
	connectivity, diagonal shapes are also present; in particular,
	$\mathcal C_2^{\rm peak}(t)$ contains four additional diagonal pairs.
	
	For $D\in\mathcal C_k^{\rm peak}(t)$, the event that $D$ is exactly the
	excursion cluster containing $t$ and that $t$ is a peak above $u$ is
	\begin{equation}\label{eq:peak-shape-event}
		\{C_u(t)=D\}\cap A_u^{\rm peak}(t)
		=
		\left\{
		X_D>u,\ X_{\mathcal N(D)}\leq u,\
		X_t\ge \max_{s\in\mathcal N(t)}X_s
		\right\}.
	\end{equation}
	Define the size-$k$ peak intensity at $t$ by
	\begin{equation}\label{eq:w-peak}
		\begin{split}
			w_k^{\rm peak}(t)
			&:=\sum_{D\in\mathcal C_k^{\rm peak}(t)}
			\P\left(
			X_D>u,\ X_{\mathcal N(D)}\leq u,\
			X_t\ge \max_{s\in\mathcal N(t)}X_s
			\right).
		\end{split}
	\end{equation}
	Specifically, $w_k^{\rm peak}(t)$ is the probability that the specified site $t$ is
	a peak above $u$ and that its associated excursion cluster has size $k$. If $X$ is stationary, then $w_k^{\rm peak}(t)$ does not depend on $t$, and we simply write it as $w_k^{\rm peak}$.
	
	\begin{theorem}[Peak-based cluster size formula]
		\label{thm:peak-CSD}
		Let $\{X_t:t\in\Z^d\}$ be a random field, not necessarily stationary, and
		let $u\in\R$ satisfy Assumption~\ref{ass:finite-cluster}.  Fix
		$t\in\Z^d$ and suppose $\rho_u^{\rm peak}(t)$ defined in \eqref{eq:peak-intensity-local} is positive. Then, for every $k\geq1$,
		\begin{equation}\label{eq:peak-CSD}
			\P\bigl(S_u^{\rm peak}(t)=k\bigr)
			=
			\frac{w_k^{\rm peak}(t)}
			{\rho_u^{\rm peak}(t)}, \qquad \rho_u^{\rm peak}(t)
			=\sum_{j=1}^{\infty}w_j^{\rm peak}(t).
		\end{equation}
		where $w_k^{\rm peak}(t)$ is given by \eqref{eq:w-peak}.  
	\end{theorem}
	
	\begin{proof}
		By expression \eqref{eq:CSD_peak}, we obtain
		\begin{equation}\label{eq:CSD_peak_ratio}
			\P\bigl(S_u^{\rm peak}(t)=k\bigr) = \frac{\P\bigl(A_u^{\rm peak}(t),\ |C_u(t)|=k\bigr)}{\P(A_u^{\rm peak}(t))},
			\qquad k=1,2,\ldots.
		\end{equation}
		For a fixed $k$, the events on the left hand side of
		\eqref{eq:peak-shape-event}, indexed by
		$D\in\mathcal C_k^{\rm peak}(t)$, are pairwise disjoint.  Their union over all possible $D\in\mathcal C_k^{\rm peak}(t)$ is
		exactly $A_u^{\rm peak}(t)\cap\{|C_u(t)|=k\}$.
		Taking probabilities gives 
		\[
		w_k^{\rm peak}(t) =\P\bigl(A_u^{\rm peak}(t),\ |C_u(t)|=k\bigr).
		\]
		Dividing by
		$\P(A_u^{\rm peak}(t))=\rho_u^{\rm peak}(t)$ and using
		\eqref{eq:CSD_peak_ratio} proves the first equality in \eqref{eq:peak-CSD}.
		
		By Assumption~\ref{ass:finite-cluster}, the cluster $C_u(t)$ has a finite
		size a.s. on $A_u^{\rm peak}(t)$.  Consequently,
		\[
		A_u^{\rm peak}(t)
		=
		\bigsqcup_{j=1}^{\infty}
		\left(A_u^{\rm peak}(t)\cap\{|C_u(t)|=j\}\right).
		\]
		Taking probabilities proves the second equality in \eqref{eq:peak-CSD}.
	\end{proof}
	
	Theorem~\ref{thm:peak-CSD} is local in its normalization: the denominator
	$\rho_u^{\rm peak}(t)$ involves only the joint distribution of
	$\{X_s:s\in\{t\}\cup\mathcal N(t)\}$.  This is the principal advantage over
	the rooted cluster distribution, whose normalizing constant generally
	contains a global connectivity correction.  The numerator in
	\eqref{eq:w-peak}, however, still becomes combinatorially demanding as $k$
	increases.
	
	\begin{remark}[Peak law among finite clusters]
		\label{remark:peak-finite-clusters}
		Without Assumption~\ref{ass:finite-cluster}, the representation
		in Theorem~\ref{thm:peak-CSD} remains valid for every finite
		cluster size. Fix $t$ with $\rho_u^{\rm peak}(t)>0$. Then
		\[
		\sum_{k\geq1}
		\P\bigl(S_u^{\rm peak}(t)=k\bigr)
		=
		\P\bigl(|C_u(t)|<\infty\mid A_u^{\rm peak}(t)\bigr)
		\leq1.
		\]
		Any missing mass is because the
		cluster associated with the peak is infinite.
		
		To obtain a normalized law among finite clusters, define
		\[
		\rho_u^{\rm peak,fin}(t)
		:=
		\P\bigl(A_u^{\rm peak}(t),\,|C_u(t)|<\infty\bigr)
		=
		\sum_{j\geq1}w_j^{\rm peak}(t).
		\]
		Whenever $\rho_u^{\rm peak,fin}(t)>0$, let
		\[
		\P\bigl(S_u^{\rm peak,fin}(t)=k\bigr)
		:=
		\P\bigl(|C_u(t)|=k
		\mid A_u^{\rm peak}(t),\,|C_u(t)|<\infty\bigr)
		=
		\frac{w_k^{\rm peak}(t)}
		{\rho_u^{\rm peak,fin}(t)},
		\quad k\geq1.
		\]
		These probabilities sum to one and describe peak sampling
		conditional on the associated cluster being finite.
		Under Assumption~\ref{ass:finite-cluster}, this law coincides
		with the original peak law. In general, its normalizer includes
		cluster finiteness and depends on the field beyond the local
		neighborhood.
	\end{remark}
	
	\subsection{Examples}
	
	Consider first the one-dimensional case $d=1$.  If a size-$k$ cluster
	contains the specified peak site $t$, let $j\in\{0,\ldots,k-1\}$ be the
	number of cluster sites strictly to the right of $t$.  The corresponding
	interval is
	\begin{equation*}
		D_{k,j}(t)
		:=\{t-k+j+1,\ldots,t+j\},
		\qquad j=0,\ldots,k-1,
	\end{equation*}
	and $\mathcal C_k^{\rm peak}(t)
	=\{D_{k,j}(t):j=0,\ldots,k-1\}$.
	The two exterior neighbors of $D_{k,j}(t)$ are $t-k+j$ and $t+j+1$.
	Substitution into the general formula \eqref{eq:w-peak} gives
	\begin{equation}\label{eq:peak-GP}
		\begin{split}
			\P(S_u^{\rm peak}(t)=k)
			={1\over\rho_u^{\rm peak}(t)}
			\sum_{j=0}^{k-1}\P\bigl(&X_{t-k+j}\leq u,
			X_{t-k+j+1}>u,\ldots,X_{t+j}>u,\\[-1mm]
			&X_{t+j+1}\leq u,
			X_t\ge \max\{X_{t-1},X_{t+1}\}\bigr),
		\end{split}
	\end{equation}
	where $\rho_u^{\rm peak}(t) = \P\left(X_t>u, \, X_t\ge \max\{X_{t-1},X_{t+1}\}\right)$.
	In particular, for $k=1$,
	\begin{equation}\label{eq:peak1}
		\P(S_u^{\rm peak}(t)=1)
		=
		\P\bigl(X_{t-1}\leq u,X_{t+1}\leq u
		\mid A_u^{\rm peak}(t)\bigr).
	\end{equation}
	For $k=2$, the cluster lies either to the right or to the left of $t$, so
	\begin{equation}\label{eq:peak2}
		\begin{split}
			\P(S_u^{\rm peak}(t)=2)
			={1\over\rho_u^{\rm peak}(t)}\bigl[&
			\P(X_{t-1}\leq u,X_t\ge X_{t+1}>u,X_{t+2}\leq u)\\
			&+\P(X_{t-2}\leq u,X_t\ge X_{t-1}>u,X_{t+1}\leq u)
			\bigr].
		\end{split}
	\end{equation}
	If the law of the process is invariant under reflection about $t$, the terms
	in \eqref{eq:peak-GP} indexed by $j$ and $k-1-j$ coincide.  
	
	\begin{example}[One-dimensional white noise]
		\label{example:peak-GWN} 
		Let $X_t$, $t\in \Z$, be i.i.d. with common continuous CDF $F$. Let $p=F(u)\in (0,1)$ and $q=1-p$. If $Y_1, Y_2$ and $Y_3$ are i.i.d. with CDF $F$, then 
		\[
		\rho_u^{\rm peak} = \P(Y_1>u, Y_1\ge Y_2, Y_1\ge Y_3) = \int_u^\infty F(x)^2 dF(x) = \frac{1-p^3}{3}.
		\]
		It follows from \eqref{eq:peak1} and \eqref{eq:peak2} that
		\begin{equation*}
			\begin{split}
				\P(S_u^{\rm peak}=1) = \frac{3p^2q}{1-p^3},\quad
				\P(S_u^{\rm peak}=2) = \frac{3p^2q^2}{1-p^3}.
			\end{split}
		\end{equation*}
		Similarly, $\P(S_u^{\rm peak}=3)=4p^2q^3/(1-p^3)$. In general, we have
		\begin{equation}\label{eq:peak-GWN-1D}
			\P(S_u^{\rm peak}=k) = \begin{cases}
				\frac{3p^2q}{1-p^3}, \quad  &k=1,\\
				\frac{(k+1)p^2q^k}{1-p^3}, \quad  &k\ge 2.
			\end{cases}
		\end{equation}
	\end{example}

	\begin{example}[Two-dimensional white noise]
		\label{example:WNF-peak}
		Let $\{X_t:t\in\Z^2\}$ be i.i.d. with a continuous CDF $F$, and set
		$p=F(u)$ and $q=1-p$.  Under nearest-neighbor connectivity, a site is
		compared with four neighbors, and hence
		\[
		\rho_u^{\rm peak}
		=\int_u^\infty F(x)^4\,dF(x)
		=\frac{1-p^5}{5}.
		\]
		Enumerating the shapes in $\mathcal C_k^{\rm peak}(t)$ and distinguishing
		whether the candidate peak has one or two neighbors inside the cluster
		gives
		\begin{equation*}
			\begin{split}
				w_1^{\rm peak} = p^4q, \quad w_2^{\rm peak} = 2p^6q^2, \quad w_3^{\rm peak} = \frac{8}{3}(2p^7+p^8)q^3.
			\end{split}
		\end{equation*}
		Taking the ratio yields the following probabilities:
		\begin{equation*}
			\begin{split}
				\P(S_u^{\rm peak}=1) = \frac{5p^4q}{1-p^5}, \quad \P(S_u^{\rm peak}=2) = \frac{10p^6q^2}{1-p^5}, \quad 
				\P(S_u^{\rm peak}=3) = \frac{40(2p^7+p^8)q^3}{3(1-p^5)}.
			\end{split}
		\end{equation*}
		
		Under Moore connectivity, the site has eight neighbors, so
		$\rho_u^{\rm peak}=(1-p^9)/9$, and
		\begin{equation*}
			\begin{split}
				w_1^{\rm peak} = p^8q, \quad w_2^{\rm peak} = 2(p^{10}+p^{12})q^2,\quad
				w_3^{\rm peak} = \frac{4}{3}(5p^{12}+8p^{14}+4p^{15}+2p^{16})q^3.
			\end{split}
		\end{equation*}
		Consequently,
		\begin{equation*}
			\begin{split}
				\P(S_u^{\rm peak}=1) &= \frac{9p^8q}{1-p^9}, \quad \P(S_u^{\rm peak}=2) = \frac{18(p^{10}+p^{12})q^2}{1-p^9}, \\
				\P(S_u^{\rm peak}=3) &= \frac{12(5p^{12}+8p^{14}+4p^{15}+2p^{16})q^3}{1-p^9}.
			\end{split}
		\end{equation*}
		Nevertheless, these weights can be evaluated by shape enumeration for
		moderate $k$ or estimated using the spatial Monte Carlo method of
		Section~\ref{sec:MC}.
	\end{example}
	
	\subsection{Empirical peak counts and consistency}
	\label{sec:empirical-peak-counts}
	
	We now assume that the field is stationary.  
	For $t\in\Z^d$, define the peak indicators
	\begin{equation}\label{eq:empirical-peak-indicators}
		Z_{t,k}
		=\mathbbm{1}\{A_u^{\rm peak}(t),|C_u(t)|=k\},
		\qquad
		Z_t=\mathbbm{1}\{A_u^{\rm peak}(t)\}.
	\end{equation}
	For a finite window $\Lambda$, define the ideal peak counts
	\begin{equation}\label{eq:ideal-peak-counts}
		N_k^{{\rm peak},*}(\Lambda)
		=\sum_{t\in\Lambda}Z_{t,k},
		\qquad
		N^{{\rm peak},*}(\Lambda)
		=\sum_{t\in\Lambda}Z_t.
	\end{equation}
	Thus, $N_k^{{\rm peak},*}(\Lambda)$ counts the peak sites in $\Lambda$ whose
	associated full-lattice clusters have size $k$, whereas
	$N^{{\rm peak},*}(\Lambda)$ counts all peak sites above $u$ in $\Lambda$.
	These are counts of \emph{peaks}, not counts of distinct clusters.  In
	particular, if one excursion cluster contains several local maxima, it is
	represented once for each of those maxima in the peak-based distribution.
	Under Assumption~\ref{ass:finite-cluster}, $N^{{\rm peak},*}(\Lambda)
	=\sum_{j=1}^{\infty}N_j^{{\rm peak},*}(\Lambda)$.
	
	As in Section~\ref{sec:empirical-cluster-counts}, the superscript $*$ means
	that a peak is counted according to its full-lattice cluster, even if that
	cluster crosses the boundary of $\Lambda$.  The ideal counts may therefore
	depend on unobserved values outside $\Lambda$.
	
	To obtain observable counts, recall the interior cluster
	collection $\mathscr C_u^\circ(\Lambda)$ from
	\eqref{eq:observable-interior-components}.  For
	$D\in\mathscr C_u^\circ(\Lambda)$, define its set of observed peak sites by
	\begin{equation*}
		\operatorname{Peak}_u(D)
		:=
		\left\{
		t\in D:X_t\ge \max_{s\in\mathcal N(t)}X_s
		\right\}.
	\end{equation*}
	Every neighbor of every $t\in D$ is observed because
	$D\cup\mathcal N(D)\subset\Lambda$.  Hence
	$\operatorname{Peak}_u(D)$ and $|D|$ are both determined by the observations
	inside $\Lambda$.  Define
	\begin{equation}\label{eq:observable-peak-counts}
		\begin{split}
			N_k^{{\rm peak},\circ}(\Lambda)
			&=
			\sum_{D\in\mathscr C_u^\circ(\Lambda)}
			\mathbbm{1}\{|D|=k\}\,|\operatorname{Peak}_u(D)|\\
			&=\#\{\text{observed peaks with associated size-$k$ interior clusters in $\Lambda$}\};
			\\
			N^{{\rm peak},\circ}(\Lambda) 
			&=
			\sum_{D\in\mathscr C_u^\circ(\Lambda)}
			|\operatorname{Peak}_u(D)|\\
			&=\#\{\text{all observed peaks with associated interior clusters in $\Lambda$}\}.
		\end{split}
	\end{equation}
	
	Let $\Lambda_n$ be the expanding rectangular windows in
	\eqref{eq:empirical-rectangular-windows}, so that their
	boundaries satisfy $|\partial_r\Lambda_n|/|\Lambda_n| \to 0$. The following theorem follows directly from similar arguments for proving Theorem \ref{thm:empirical-cluster-counts}; the detailed proof is provided in the Appendix.
	
	\begin{theorem}[Consistency of empirical peak frequencies]
		\label{thm:empirical-peak-counts}
		Let $\{X_t:t\in\Z^d\}$ be stationary and ergodic, let
		Assumption~\ref{ass:finite-cluster} hold, and suppose
		$\rho_u^{\rm peak}>0$.  Then, for fixed $k\geq1$,
		\begin{align}
			\frac{N_k^{{\rm peak},*}(\Lambda_n)}{|\Lambda_n|}
			&\overset{\rm a.s.}{\longrightarrow}w_k^{\rm peak},
			&
			\frac{N^{{\rm peak},*}(\Lambda_n)}{|\Lambda_n|}
			&\overset{\rm a.s.}{\longrightarrow}\rho_u^{\rm peak},
			\nonumber
			\\
			\widetilde w_{k,n}^{\rm peak}
			:=\frac{N_k^{{\rm peak},\circ}(\Lambda_n)}{|\Lambda_n|}
			&\overset{\rm a.s.}{\longrightarrow}w_k^{\rm peak},
			&
			\widetilde \rho_{u,n}^{\rm peak}
			:=\frac{N^{{\rm peak},\circ}(\Lambda_n)}{|\Lambda_n|}
			&\overset{\rm a.s.}{\longrightarrow}\rho_u^{\rm peak}.
			\label{eq:w-peak-empirical}
		\end{align}
		Moreover, $N^{{\rm peak},\circ}(\Lambda_n)>0$ eventually almost surely,
		and
		\begin{equation}\label{eq:empirical-peak-ratio}
			\frac{N_k^{{\rm peak},\circ}(\Lambda_n)}
			{N^{{\rm peak},\circ}(\Lambda_n)} = \frac{\widetilde w_{k,n}^{\rm peak}}{\widetilde \rho_{u,n}^{\rm peak}}
			\overset{\rm a.s.}{\longrightarrow}
			\frac{w_k^{\rm peak}}{\rho_u^{\rm peak}}
			=\P(S_u^{\rm peak}=k).
		\end{equation}
		In fact, the whole empirical probability mass function is consistent in total variation:
		\begin{equation*}
			\frac{1}{2}\sum_{k=1}^{\infty}
			\left|
			\frac{N_k^{{\rm peak},\circ}(\Lambda_n)}
			{N^{{\rm peak},\circ}(\Lambda_n)}
			-\P(S_u^{\rm peak}=k)
			\right|
			\overset{\rm a.s.}{\longrightarrow}0.
		\end{equation*}
		For every finite $\Lambda$, the ideal counts satisfy the exact expectation
		identities
		\begin{equation*}
			\E N_k^{{\rm peak},*}(\Lambda)=|\Lambda|w_k^{\rm peak},
			\qquad
			\E N^{{\rm peak},*}(\Lambda)=|\Lambda|\rho_u^{\rm peak}.
		\end{equation*}
	\end{theorem}
	
	Combining \eqref{eq:observable-peak-counts} and the convergence in \eqref{eq:empirical-peak-ratio}, we see that, as the window grows, the relative frequency of peaks in size-$k$ clusters converges to the theoretical peak-based cluster size probability in Theorem~\ref{thm:peak-CSD}. This verifies that our theoretical formula agrees with the empirical
	peak-based cluster size distribution computed in simulations.
	
	For a nonstationary field, the analogous fixed-site result follows directly
	from the strong law applied to independent field realizations. That is, the ratio of the number of peak events at $t$ with size-$k$ cluster to the total number of peak events at $t$ converges almost surely to
	$\P(S_u^{\rm peak}(t)=k)$.

	\subsection{Root versus peak sampling}\label{sec:high-u}
	
	For a stationary field, the only difference between root and peak sampling is
	the number of votes assigned to each cluster: one at its root versus one at
	each of its weak peaks.  The two laws coincide whenever every finite cluster
	has exactly one peak.  At high Gaussian thresholds on a fixed lattice an even
	stronger mechanism operates: both laws concentrate on
	singleton clusters because neighboring simultaneous exceedances are
	exponentially negligible relative to a marginal exceedance.  We quantify the
	comparison before returning to this interpretation in
	Corollary~\ref{cor:ld-singleton}.
	
	For a threshold $u$, define, on the root event $R_u(o)$,
	\begin{equation}\label{eq:number-peaks-root-cluster}
		M_u(o)
		:=
		\sum_{t\in C_u(o)}
		\mathbbm{1}\{A_u^{\rm peak}(t)\},
	\end{equation}
	the number of local maxima in the excursion cluster rooted at the
	origin. Since $C_u(o)$ is finite on $R_u(o)$, one has $M_u(o)\geq1$. The proof of the following result is given in the Appendix.
	
	\begin{theorem}[Root--peak comparison]
		\label{thm:root-peak-comparison}
		Suppose that the field is stationary, Assumption~\ref{ass:finite-cluster}
		holds, and $\rho_u=\P(R_u(o))>0$. Then the following change of measure holds for every $k\ge 1$ with $w_k>0$:
		\begin{equation}\label{eq:change-of-measure}
			\P(S_u^{\rm peak}=k) = \frac{\E\left[
				M_u(o)\mid R_u(o), |C_u(o)|=k
				\right]}{\E\left[
				M_u(o)\mid R_u(o) \right]}\P(S_u=k).
		\end{equation}
		Moreover,
		\begin{equation}\label{eq:root-peak-TV-bound}
			d_{\rm TV}\!\left(\mathcal L(S_u^{\rm peak}),
			\mathcal L(S_u)\right)
			\leq
			\frac{\E\left[M_u(o)-1\mid R_u(o)\right]}
			{\E\left[M_u(o)\mid R_u(o)\right]}
			=
			1-\frac{\rho_u}{\rho_u^{\rm peak}}.
		\end{equation}
		Hence the root and peak-based cluster size distributions are identical when $M_u(o)=1$ almost surely on
		$R_u(o)$, and they are asymptotically equivalent if
		\[
		\E\left[M_u(o)-1\mid R_u(o)\right]\to 0,
		\qquad\text{as }u\to\infty.
		\]
	\end{theorem}
	The root-peak change of measure in \eqref{eq:change-of-measure} can be interpreted as
	\[
	\P(S_u^{\rm peak}=k) = \frac{\text{\rm mean peak count among clusters of size $k$}}{\text{\rm mean peak count among all clusters}}\times \P(S_u=k).
	\]
	
	The following result, with proof given in the Appendix, shows that the distance between the root and peak-based cluster size distributions of Gaussian random fields vanishes exponentially fast as $u\to \infty$.
	\begin{corollary}[Gaussian high-threshold equivalence]
		\label{cor:gaussian-root-peak}
		Let $\{X_t:t\in\Z^d\}$ be a stationary, centered, unit-variance Gaussian
		field. Suppose that 
		Assumption~\ref{ass:finite-cluster} holds for all sufficiently large $u$,
		and that $r_*=
		\max_{s\in\mathcal N(o)}
		\operatorname{Corr}(X_o,X_s)\in (-1, 1)$.
		Then there exists $\alpha>0$ such that
		\[
		d_{\rm TV}\!\left(\mathcal L(S_u^{\rm peak}),
		\mathcal L(S_u)\right) \le e^{-\alpha u^2}
		\to 0,
		\qquad \text{\rm as } u\to\infty.
		\]
	\end{corollary}

	\section{Monte Carlo estimation}
	\label{sec:MC}
	Using independent field realizations, we estimate the unnormalized root and
	peak intensities by spatial Monte Carlo averaging. Their normalizations yield
	the $K$-truncated root Palm law and the peak-sampled
	law. A deterministic guard band removes finite-window bias
	for all cluster sizes up to a cutoff $K$. Throughout this section the field is
	stationary, Assumption~\ref{ass:finite-cluster} holds, and the relevant
	normalizers $\rho_u$ and $\rho_u^{\rm peak}$ are positive.
	
	For a fixed cutoff $K\geq1$, define
	\begin{align*}
		\rho_K:=\sum_{k=1}^K w_k,
		\qquad \tau_K:=\sum_{k>K}w_k=\rho_u-\rho_K.
	\end{align*}
	By \eqref{eq:peak-event} and \eqref{eq:peak-intensity-local}, the peak normalizer $\rho_u^{\rm peak}$ is a finite-dimensional probability
	involving only a site and its neighbors. It is therefore assumed to be evaluated directly from their joint distribution. The theoretical results below treat this value as exact.
	
	\subsection{Guarded spatial estimation}
	
	\subsubsection{Exact computation with guard bands}
	\label{sec:mc-guard}
	
	Let $G$ be the lattice graph and $d_G$ its graph distance.  For
	$A\subset\Z^d$ and $r\geq0$, define 
	\begin{equation*}
		A^{\oplus r}:=\left\{s\in\Z^d: \inf_{a\in A}d_G(s,a)\leq r\right\},
	\end{equation*}
	and write it simply $t^{\oplus r}$ if $A=\{t\}$. Thus $A^{\oplus r}\setminus A$ is the distance-$r$ guard band around
	$A$.
	
	For a finite simulation window $V\subset\Z^d$ and $t\in V$, let $C_u(t\mid V)$ be the connected component of $t$ restricted on $E_u\cap V$, and set $C_u(t\mid V)=\varnothing$ when $X_t\leq u$.
	Equivalently,
	\begin{equation*}
		C_u(t\mid V)
		=
		\left\{s\in E_u\cap V:
		\text{there is an excursion path in $V$ connecting $t$ to $s$}
		\right\}.
	\end{equation*}
	The component $C_u(t\mid V)$ uses only observations in $V$ and may be smaller
	than the full-lattice component $C_u(t)$ if the latter is not closed inside
	$V$.
	
	\begin{lemma}[Exact finite determination]
		\label{lem:mc-guard-band}
		If $t^{\oplus K}\subset V$, then, for every $1\leq k\leq K$,
		\begin{align*}
			\{|C_u(t)|=k\}=\{|C_u(t\mid V)|=k\},
			\qquad
			\{|C_u(t)|>K\}=\{|C_u(t\mid V)|>K\}.
		\end{align*}
		Consequently, observations on any $V\supset \Lambda^{\oplus K}$ determine exactly,
		for every $t\in\Lambda$, whether $t$ is a nonexceedance, belongs to a
		cluster of size $1,\ldots,K$, or belongs to a cluster of size greater than
		$K$.  They also determine $A_u^{\rm peak}(t)$.
	\end{lemma}

	Thus $\Lambda$ is the scoring window and any finite
	$V\supset\Lambda^{\oplus K}$ may serve as the larger simulation window. Connected components are labeled with the usual lattice adjacency on $V$, without periodic wrap-around.
	
	\subsubsection{Single-origin estimators}
	\label{sec:mc-estimators}
	
	For $1\leq k\leq K$, define the root and peak site scores
	\begin{equation}\label{eq:xi}
		\xi_{t,k}
		:=\frac1k\mathbbm{1}\{|C_u(t)|=k\},
		\qquad
		\xi_{t,k}^{\rm peak}
		:=\mathbbm{1}\{A_u^{\rm peak}(t),\ |C_u(t)|=k\}.
	\end{equation}
	By stationarity, \eqref{eq:w-inside-w} and \eqref{eq:w-peak} give $\E\xi_{t,k}=w_k$ and $\E\xi_{t,k}^{\rm peak}=w_k^{\rm peak}$.
	The factor $1/k$ removes the size bias caused by observing a size-$k$
	cluster from each of its $k$ sites.  No such factor is used for peaks, since
	each peak is a sampling location.
	
	Let $X^{(1)},\ldots,X^{(M)}$ be independent field realizations. A superscript $(m)$ denotes the quantity computed
	from realization $m$. The single-origin estimators for $w_k$ and $w_k^{\rm peak}$, obtained by simulating on $V\supset o^{\oplus K}$, are respectively
	\begin{equation*}
		\begin{aligned}
			\widehat w_k^{\rm SO}
			&=\frac1{Mk}\sum_{m=1}^M
			\mathbbm{1}\{|C_u^{(m)}(o\mid V)|=k\},\\
			\widehat w_k^{{\rm peak},{\rm SO}}
			&=\frac1M\sum_{m=1}^M
			\mathbbm{1}\{A_u^{{\rm peak},(m)}(o),\
			|C_u^{(m)}(o\mid V)|=k\}.
		\end{aligned}
	\end{equation*}

	\subsubsection{Spatial estimators}
	
	The single-origin estimator uses only the excursion component associated with one prespecified site and therefore leaves most of the information in a simulated field unused. We next introduce a spatial procedure that extracts more information from each realization by using all sites of a finite scoring window $\Lambda\subset\Z^d$. To ensure that cluster sizes up to $K$ are determined exactly, the field is simulated on a finite window $V\supset \Lambda^{\oplus K}$. For one realization define
	\begin{align}
		W_{m,k}
		&=\frac{1}{k|\Lambda|}\sum_{t\in\Lambda}
		\mathbbm{1}\{|C_u^{(m)}(t\mid V)|=k\},
		\label{eq:MC2}\\
		W_{m,k}^{{\rm peak}}
		&=\frac{1}{|\Lambda|}\sum_{t\in\Lambda}
		\mathbbm{1}\{A_u^{{\rm peak},(m)}(t),\ |C_u^{(m)}(t\mid V)|=k\}.
		\label{eq:MC3}
	\end{align}
	Their averages over independent realizations give the following \textit{spatial Monte Carlo estimators} for $w_k$ and $w_k^{\rm peak}$:
	\begin{equation}\label{eq:wk-estimate}
		\widehat w_k:=\frac1M\sum_{m=1}^M W_{m,k},
		\qquad
		\widehat w_k^{\rm peak}:=\frac1M\sum_{m=1}^M W_{m,k}^{\rm peak}.
	\end{equation}
	The spatial peak-based probability estimator is therefore
	\begin{equation*}
		\widehat p_k^{\rm peak}
		:=\widehat w_k^{\rm peak}/\rho_u^{\rm peak}.
	\end{equation*}
	The choice $\Lambda=\{o\}$ recovers the single-origin estimator; larger
	windows use more of each simulated realization.
	
	Rather than determining $C_u(t \mid V)$ separately for every
	$t\in\Lambda$, we label the connected components of the excursion set
	$E_u\cap V$ once. The resulting component labels and sizes provide all the quantities needed to evaluate the sums in \eqref{eq:MC2}--\eqref{eq:MC3}. Let
	$\mathscr C_u^{(m)}(V)$ denote the collection of labeled excursion components in
	realization $m$.  For $D\in\mathscr C_u^{(m)}(V)$, let
	\[
	a_D=|D\cap\Lambda|,
	\qquad
	b_D=\#\{t\in D\cap\Lambda:
	A_u^{{\rm peak},(m)}(t)\}.
	\]
	Only components satisfying $D\cap\Lambda\neq\varnothing$ need to be processed. Lemma~\ref{lem:mc-guard-band} guarantees that every component intersecting
	$\Lambda$ is classified correctly as having size $k\leq K$ or size greater
	than $K$, and hence gives the following algorithm.
	\begin{enumerate}[label=(\roman*)]
		\item Generate the field on $V\supset\Lambda^{\oplus K}$, and label its connected excursion components on $E_u\cap V$.
		\item For each component $D$ with size $k\leq K$, add $a_D/k$ to the root size-$k$ total and add $b_D$ to the peak size-$k$ total.
		\item For each component $D$ with size greater than $K$, add $a_D$ to the root overflow total.
		\item Divide all totals by $|\Lambda|$ and average them over the $M$
		independent realizations.
	\end{enumerate}
	Applying this algorithm, the scores in \eqref{eq:MC2}--\eqref{eq:MC3} become
	\begin{align}\label{eq:W2}
		W_{m,k}
		=
		\frac1{|\Lambda|}
		\sum_{\substack{D\in\mathscr C_u^{(m)}(V)\\
				|D|=k}}
		\frac{|D\cap\Lambda|}{k}, \qquad 
		W_{m,k}^{\rm peak}
		=
		\frac1{|\Lambda|}
		\sum_{\substack{D\in\mathscr C_u^{(m)}(V)\\
				|D|=k}}
		b_D.
	\end{align}
	In general, a size-$k$ component $D$ contributes $|D\cap\Lambda|/|D|$ to the unnormalized root total. For root law tail control, also define and record
	\begin{equation}\label{eq:mc-overflow-estimators}
		\begin{split}
			T_m&:=\frac1{|\Lambda|}\sum_{t\in\Lambda}
			\mathbbm{1}\{|C_u^{(m)}(t\mid V)|>K\} = \frac1{|\Lambda|}
			\sum_{\substack{D\in\mathscr C_u^{(m)}(V)\\|D|>K}}
			|D\cap\Lambda|,\\
			\widehat r_K&:=\frac1M\sum_{m=1}^M T_m, \qquad r_K:=\P(|C_u(o)|>K).
		\end{split}
	\end{equation}

	\begin{remark}[Peak-based estimation]
		\label{remark:MC-peak}
		The estimator \eqref{eq:MC3} is an unconditional spatial average under
		the original field law, not simulation conditional on a peak.  For a
		nonstationary field, the single-origin construction remains valid at a
		fixed site $t$, with the quantities
		$w_k^{\rm peak}(t)$ and $\rho_u^{\rm peak}(t)$ depending on $t$.
	\end{remark}
	
	\subsection{Unbiasedness, consistency, and truncation control}
	\label{sec:mc-unbiased}
	
	Define 
	\[
	\widehat\rho_K:=\sum_{k=1}^K\widehat w_k.
	\]
	\begin{theorem}[Unbiasedness and consistency]
		\label{thm:mc-unbiased}
		Suppose that the field is stationary and that
		Assumption~\ref{ass:finite-cluster} holds.  For every finite nonempty
		$\Lambda$, every $M\geq1$, and every $1\leq k\leq K$,
		\begin{equation}\label{eq:mc-unbiasedness}
			\E\widehat w_k=w_k,
			\qquad
			\E\widehat w_k^{\rm peak}=w_k^{\rm peak},
			\qquad
			\E\widehat\rho_K=\rho_K,
			\qquad
			\E\widehat r_K=r_K.
		\end{equation}
		All four estimators converge a.s. to their targets as
		$M\to\infty$.  If $\rho_K>0$ and $\rho_u^{\rm peak}>0$, then
		\begin{equation}\label{eq:mc-ratio-consistency}
			\frac{\widehat w_k}{\widehat\rho_K}
			\to\frac{w_k}{\rho_K},
			\qquad
			\widehat p_k^{\rm peak}
			\to
			\P(S_u^{\rm peak}=k), \quad {\rm a.s}.
		\end{equation}
		Moreover, because $\rho_u^{\rm peak}$ is nonrandom,
		$\widehat p_k^{\rm peak}$ is unbiased for
		$\P(S_u^{\rm peak}=k)$ for every $M$.
	\end{theorem}

	The root ratio in \eqref{eq:mc-ratio-consistency} estimates the truncated
	law $w_k/\rho_K$, rather than $\P(S_u=k)=w_k/\rho_u$. 	Using $\P(|C_u(o)|=j)=j w_j$ yields
	\[
	r_K
	=\sum_{j>K}\P(|C_u(o)|=j)
	=\sum_{j>K}j w_j
	\geq(K+1)\sum_{j>K}w_j=(K+1)\tau_K.
	\]
	Consequently, the overflow probability $r_K$ gives the bound
	\begin{align}\label{eq:tau-B}
		0\leq\tau_K
		\leq B_K:=\frac{r_K}{K+1}.
	\end{align}

	When $\rho_K>0$, let $p_{k,K}=w_k/\rho_K$ for $k\leq K$ and
	$p_{k,K}=0$ for $k>K$; this is the renormalized $K$-truncated law.
	The quantity $B_K$ is deterministic but unknown; simply replacing $r_K$ by
	$\widehat r_K$ does not produce an upper bound by realizations. The
	following finite-sample certificate makes the truncation control operational.
	
	\begin{proposition}[Finite-sample truncation certificate]
		\label{prop:finite-sample-truncation}
		Suppose $\rho_K>0$.  Let $R_m=\sum_{k=1}^K W_{m,k}$, so that
		$0\leq R_m\leq1$ and $\widehat\rho_K=M^{-1}\sum_mR_m$.  For
		$\delta\in(0,1)$, let
		\begin{equation*}
			\begin{aligned}
				\varepsilon_{M,\delta}
				=\sqrt{\frac{\log(2/\delta)}{2M}},\quad
				r_{K,U}=\min\{1,\widehat r_K+\varepsilon_{M,\delta}\},\quad
				\rho_{K,L}=\max\{0,\widehat\rho_K-\varepsilon_{M,\delta}\}.
			\end{aligned}
		\end{equation*}
		Define $B_{K,U}:=r_{K,U}/(K+1)$. Then, with probability at least
		$1-\delta$,
		\begin{equation*}
			\tau_K\leq B_{K,U},\qquad
			d_{\rm TV}\!\left(p_{\cdot,K},\mathcal L(S_u)\right)
			\leq\frac{B_{K,U}}{\rho_{K,L}+B_{K,U}},
		\end{equation*}
		whenever the displayed denominator is positive; otherwise, use the trivial
		upper bound $1$.
	\end{proposition}
	
	\begin{remark}[Finite-component interpretation without nonpercolation]
		\label{remark:mc-finite-components}
		Lemma~\ref{lem:mc-guard-band} does not require all components to be finite:
		an infinite component is classified as larger than $K$. For finite $k$,
		the complete-shape intensities $w_k$ and $w_k^{\rm peak}$ remain well defined,
		and the guarded scores remain unbiased for those intensities. Root only
		finite components and let
		\[
		\rho_u^{\rm fin}=\sum_{j\geq1}w_j,
		\qquad \tau_K^{\rm fin}=\sum_{j>K}w_j,
		\qquad p_k^{\rm fin}=w_k/\rho_u^{\rm fin},
		\]
		when $\rho_u^{\rm fin}>0$. Stationarity gives
		\[
		r_K=\sum_{j>K}j w_j+\P\{|C_u(o)|=\infty\}
		\geq (K+1)\tau_K^{\rm fin}.
		\]
		Thus the proof of Proposition~\ref{prop:finite-sample-truncation}
		also bounds the distance between $p_{\cdot,K}$ and $p_{\cdot}^{\rm fin}$.
		It does not certify absence of infinite components. For peak sampling,
		$w_k^{\rm peak}/\rho_u^{\rm peak}$ is still the conditional probability
		of a finite size-$k$ component, but these probabilities can sum to less
		than one, with the remaining mass at infinite size. 
	\end{remark}

	For many fields, including Gaussian fields under suitable regularity conditions, large excursion clusters are rare, and the overflow probability $r_K$ and the tail intensity $\tau_K$
	decrease rapidly to zero as $K\to\infty$; see
	Remark~\ref{remark:ld-mc-truncation}.
	
	The effect of the root Palm tail on normalization can be expressed as
	\begin{equation*}
		\rho_K\leq\rho_u\leq\rho_K+B_K,
		\qquad
		\frac{w_k}{\rho_K+B_K}
		\leq\P(S_u=k)
		\leq\frac{w_k}{\rho_K},
		\quad 1\leq k\leq K.
	\end{equation*}
	For this truncated law, the complete truncation error is
	\begin{equation}\label{eq:mc-exact-TV-truncation}
		\begin{split}
			d_{\rm TV}\!\left(p_{\cdot,K},\mathcal L(S_u)\right)
			&=\frac12\sum_{k=1}^{\infty} \left|
			p_{k,K}
			-\P(S_u=k)\right|\\
			&=\frac{\tau_K}{\rho_u}
			=\frac{\tau_K}{\rho_K+\tau_K}
			\leq\frac{B_K}{\rho_K+B_K}.
		\end{split}
	\end{equation}
	
	\begin{corollary}[Iterated consistency for the root Palm distribution]
		\label{cor:mc-two-stage-consistency}
		Whenever $\widehat\rho_K>0$, define $\widehat p_{k,K,M}=\widehat w_k/\widehat\rho_K$ for $k\leq K$ and zero for $k>K$.
		Under the assumptions of Theorem~\ref{thm:mc-unbiased}, assume also
		$\rho_K>0$.  Then
		\begin{equation*}
			\begin{split}
				\lim_{M\to\infty}
				d_{\rm TV}\!\left(\widehat p_{\cdot,K,M},\mathcal L(S_u)\right)
				&=\frac{\tau_K}{\rho_u}, \quad \text{\rm a.s.},\\
				\lim_{K\to\infty}\lim_{M\to\infty}
				d_{\rm TV}\!\left(\widehat p_{\cdot,K,M},\mathcal L(S_u)\right)
				&=0, \quad \text{\rm a.s.}
			\end{split}
		\end{equation*}
	\end{corollary}

	\begin{remark}[Refining the truncation allowance]
		If the upper bound at $K$ is too conservative, choose $L>K$ as a new guard
		band.  Then
		\begin{equation*}
			\sum_{j=K+1}^{L}w_j
			\leq\tau_K
			\leq\sum_{j=K+1}^{L}w_j+\frac{r_L}{L+1}.
		\end{equation*}
		The finite sum is estimated unbiasedly by
		$\sum_{j=K+1}^{L}\widehat w_j$, while $r_L$ is estimated by the spatial
		overflow score at cutoff $L$.  Increasing $L$ therefore replaces part of the
		upper allowance by explicitly estimated cluster weights.
	\end{remark}
	
	\subsection{Monte Carlo variance of spatial estimators}
	\label{sec:mc-spatial-variance}
	
	We now quantify the Monte Carlo error of the spatial estimator $\widehat w_k$ and make its dependence on the size and geometry of $\Lambda$ explicit. The same
	argument applies to $\widehat w_k^{\rm peak}$ after replacing
	$\xi_{t,k}$ by $\xi_{t,k}^{\rm peak}$. 
	
	For $h\in\Z^d$, define $\gamma_k(h)
	=\operatorname{Cov}(\xi_{o,k}, \xi_{o+h,k})$
	and 
	\begin{equation*}
		N_\Lambda(h)
		=\#\{(s,t)\in\Lambda^2:t-s=h\}
		=|\Lambda\cap(\Lambda-h)|,
		\qquad
		a_\Lambda(h)=\frac{N_\Lambda(h)}{|\Lambda|}.
	\end{equation*}
	
	\begin{proposition}[Window-dependent spatial variance]
		\label{prop:mc-window-variance}
		Suppose that the field is stationary.  For every finite nonempty
		$\Lambda$, every $1\leq k\leq K$, 
		\begin{equation}\label{eq:mc-window-variance}
			\begin{split}
				\operatorname{Var}(\widehat w_k)
				=\frac1{M|\Lambda|^2}
				\sum_{h\in\Z^d}N_\Lambda(h)\gamma_k(h)
				=\frac1{M|\Lambda|}
				\sum_{h\in\Z^d}a_\Lambda(h)\gamma_k(h).
			\end{split}
		\end{equation}
		If $\sum_{h\in\Z^d}|\gamma_k(h)|<\infty$,
		then
		\begin{equation}\label{eq:mc-window-variance-bound}
			\operatorname{Var}(\widehat w_k)
			\leq\frac1{M|\Lambda|}
			\sum_{h\in\Z^d}|\gamma_k(h)|.
		\end{equation}
		Consequently, the variance is $O((M|\Lambda|)^{-1})$ for fixed $k$.
	\end{proposition}
	
	The preceding identity is useful for theory, but, for $M\geq2$, independent field realizations also give a directly implementable variance estimator that
	automatically retains all spatial dependence within $\Lambda$:
	\begin{equation*}
		\widehat{\operatorname{Var}}(\widehat w_k)
		=\frac{1}{M(M-1)}\sum_{m=1}^M
		(W_{m,k}-\widehat w_k)^2.
	\end{equation*}
	The same formula applies to $W_{m,k}^{\rm peak}$ and $T_m$.
	More generally, applying the sample covariance formula to the
	replication-level vector of all reported scores estimates their full
	Monte Carlo covariance matrix.  Thus independent field realizations are the sampling units for uncertainty assessment.
	
	For peak probabilities with an exact normalizer
	$\rho_u^{\rm peak}>0$, recall
	$\widehat p_k^{\rm peak}
	=\widehat w_k^{\rm peak}/\rho_u^{\rm peak}$.
	The covariance matrix of these probability estimates is the
	corresponding weight covariance matrix divided by
	$(\rho_u^{\rm peak})^2$. In particular, for $M\geq2$, an unbiased variance estimator is
	\begin{equation*}
		\widehat{\operatorname{Var}}(\widehat p_k^{\rm peak})
		=
		\frac{1}{M(M-1)(\rho_u^{\rm peak})^2}
		\sum_{m=1}^M
		\left(W_{m,k}^{\rm peak}-\widehat w_k^{\rm peak}\right)^2.
	\end{equation*}
	
	For a rectangular scoring window $\Lambda=\prod_{i=1}^d\{1,\ldots,n_i\}$ with $|\Lambda|=\prod_{i=1}^dn_i$, the overlap count is $N_\Lambda(h)=\prod_{i=1}^d(n_i-|h_i|)_+$ where $(x)_+:=\max\{x,0\}$.
	Along regular expanding windows $\{\Lambda_n\}$ satisfying
	$|\Lambda_n|\to\infty$ and $a_{\Lambda_n}(h)\to 1$ for every fixed $h$,
	we have
	\begin{equation*}
		M|\Lambda_n|\operatorname{Var}(\widehat w_{k,n})
		\to \sum_{h\in\Z^d}\gamma_k(h),
		\qquad \text{\rm as } n\to\infty.
	\end{equation*}
	
	\section{Cluster tails and high-threshold large deviations}
	\label{sec:large-deviations}
	
	We consider two complementary regimes: $k\to\infty$ at fixed threshold and
	$u\to\infty$ at fixed cluster size.  The first yields conservative
	exponential tail bounds useful for Monte Carlo truncation; the second yields
	exact $u^2$-scale logarithmic asymptotics for Gaussian fields.
	
	One comparison will be used in both regimes.  A size-$k$ cluster contributes once to the root weight and once for each of its peaks to the peak weight.
	Every finite size-$k$ cluster has between one and $k$ weak peaks, so
	\begin{equation}\label{eq:ld-weight-sandwich}
		w_k\leq w_k^{\rm peak}\leq k w_k.
	\end{equation}
	In the fixed-threshold regime, $(\log k)/k\to0$; in the high-threshold
	regime, $k$ is fixed and $(\log k)/u^2\to0$.  Thus the factor is negligible
	on the relevant exponential scale in either regime.  
	
	\subsection{Large-cluster bounds at a fixed threshold}
	\label{sec:ld-large-k}
	
	In this subsection, $u$ is fixed and the field is stationary.  We give a
	sufficient condition for exponential cluster tails and then verify it for a
	class of Gaussian fields.
	
	\subsubsection{A general exponential bound}
	\label{sec:ld-general-bound}
	
	Exponential decay of cluster size tails in subcritical independent
	percolation is classical \cite{AizenmanBarsky87,Grimmett99}.  The result below establishes a verifiable sufficient condition for
	exponential decay of excursion cluster size tails in dependent random
	fields. The derived explicit bounds quantify the probability mass
	omitted by truncation, connecting dependence assumptions to computational
	accuracy and providing a theoretical basis for selecting the Monte Carlo
	cutoff and controlling the resulting approximation error.
	
	Recall that $\mathcal C_k^{\rm site}$ is the collection of size-$k$ clusters that contain the origin; and $|\mathcal C_k^{\rm site}|$ denotes the number of clusters it contains. Define the maximal vertex degree of the lattice graph by $\Delta:=|\mathcal N(o)|$. A depth-first contour traversal of a deterministic spanning tree of
	$D\in\mathcal C_k^{\rm site}$ has length $2(k-1)$ and determines $D$.
	Since at each step
	there are at most $\Delta$ choices,
	\begin{equation}\label{eq:ld-animal-count}
		|\mathcal C_k^{\rm site}|
		\leq \Delta^{2(k-1)}\leq\beta^k,
		\qquad \text{\rm where } \beta:=\Delta^2.
	\end{equation}
	For nearest-neighbor adjacency on $\Z^d$, $\Delta=2d$ and we have
	$\beta=4d^2$.  In $d=1$, the exact count is
	$|\mathcal C_k^{\rm site}|=k$, so the sharper choice $\beta=2$ is valid.
	
	\begin{assumption}[Uniform cluster exceedance bound]
		\label{ass:ld-connected-exceedance}
		There exists $\psi_u>0$ such that, for every $k\geq1$,
		\begin{equation}\label{eq:ld-joint-exceedance}
			\sup_{D\in\mathcal C_k^{\rm site}}
			\P(X_t>u\text{ for every }t\in D)
			\leq e^{-\psi_u k}.
		\end{equation}
	\end{assumption}

	The next theorem expresses the competition between the exponential number
	of possible cluster shapes and the probability of simultaneous exceedances
	on any one shape.
	
	\begin{theorem}[Exponential bound for large clusters]
		\label{thm:ld-general}
		Suppose Assumption~\ref{ass:ld-connected-exceedance} holds and
		$\psi_u>\log \beta$.  Let
		\begin{equation}\label{eq:ld-cu}
			c_u:=\psi_u-\log \beta>0.
		\end{equation}
		Then, for every $k\geq1$,
		\begin{align}
			\P\{|C_u(o)|\geq k\}
			&\leq e^{-c_u k}.
			\label{eq:ld-site-tail}
		\end{align}
		In particular, Assumption~\ref{ass:finite-cluster} holds.  If
		$\rho_u,\rho_u^{\rm peak}>0$, then
		\begin{align}
			\P(S_u\geq k)
			\leq\frac{e^{-c_u k}}{k\rho_u}, \qquad
			\P(S_u^{\rm peak}\geq k)
			\leq\frac{e^{-c_u k}}{\rho_u^{\rm peak}}.
			\label{eq:ld-peak-tail}
		\end{align}
	\end{theorem}
	
	The condition $\psi_u>\log \beta$ has a direct interpretation:
	$\psi_u$ is the probabilistic cost per site of maintaining simultaneous
	exceedances, whereas $\log \beta$ is an upper bound on the combinatorial growth
	per site allowed by the count of connected shapes.  Large clusters have an
	exponentially small probability whenever the former dominates the latter.
	
	\subsubsection{Gaussian truncation bounds}
	\label{sec:ld-gaussian}
	
	For a stationary Gaussian field, absolute summability of the covariance
	provides a simple sufficient criterion for the joint exceedance condition.
	
	\begin{corollary}[Gaussian large cluster bound]
		\label{cor:ld-gaussian}
		Let $X$ be a centered stationary Gaussian field with covariance
		$\Gamma(h):=\operatorname{Cov}(X_o,X_h)$, and suppose $\kappa=\sum_{h\in\Z^d}|\Gamma(h)|\in (0,\infty)$.
		If $u>\sqrt{2\kappa\log \beta}>0$,
		then \eqref{eq:ld-site-tail}--\eqref{eq:ld-peak-tail} hold with $c_u=u^2/(2\kappa)-\log \beta>0$.
	\end{corollary}

	\begin{remark}[Implication for Monte Carlo truncation]
		\label{remark:ld-mc-truncation}
		Under Theorem~\ref{thm:ld-general}, taking $k=K+1$ gives $r_K=\P\{|C_u(o)|>K\}
		\leq e^{-c_u(K+1)}$.
		Consequently, by \eqref{eq:tau-B}, $\tau_K\leq B_K
		\leq e^{-c_u(K+1)}/(K+1)$,
		and the total variation error in
		\eqref{eq:mc-exact-TV-truncation} obeys
		\begin{equation*}
			d_{\rm TV}\!\left(p_{\cdot,K},\mathcal L(S_u)\right)
			\leq
			\frac{e^{-c_u(K+1)}}{(K+1)\rho_u}.
		\end{equation*}
		The peak-based tail probability is bounded directly by $\P(S_u^{\rm peak}>K)
		\leq e^{-c_u(K+1)}/\rho_u^{\rm peak}$.
		Thus the bound provides a conservative guide for selecting
		the Monte Carlo cutoff $K$.
	\end{remark}
	
	\subsection{Gaussian high-threshold large deviations at fixed size}
	\label{sec:ld-high-threshold}
	
	We now fix $k$ and let $u\to\infty$.  In this
	subsection, let $X$ be a centered stationary Gaussian field with
	$\sigma^2=\operatorname{Var}(X_o)>0$.  Assume that the covariance matrix on
	every finite nonempty set of distinct sites is positive definite and that
	Assumption~\ref{ass:finite-cluster} holds for every $u\geq u_0$, for some
	$u_0<\infty$.  Positive definiteness gives an everywhere positive density for
	each finite-dimensional distribution and $\operatorname{Corr}(X_s,X_t) \in (-1,1)$ whenever $s\ne t$. To display the dependence on the threshold,
	write the weights as $w_k(u)$ and $w_k^{\rm peak}(u)$ in this subsection.
	
	For $D\in\mathcal C_k^{\rm root}$, note that $D^{\oplus 1}=D\cup\mathcal N(D)$; let $\Sigma_{D^{\oplus 1}}$ be the covariance matrix of
	$(X_t)_{t\in D^{\oplus 1}}$, and define the closed polyhedron
	\begin{equation}\label{eq:ld-shape-polyhedron}
		\mathcal K_D
		:=
		\left\{
		x=(x_t)_{t\in D^{\oplus 1}}:
		x_t\geq1\ \text{for }t\in D,
		\,
		x_s\leq1\ \text{for }s\in\mathcal N(D)
		\right\}.
	\end{equation}
	The rate associated with the rooted shape $D$ is
	\begin{equation*}
		I_D
		:=
		\inf_{x\in\mathcal K_D}
		\frac12x^\top\Sigma_{D^{\oplus 1}}^{-1}x,
	\end{equation*}
	and the size-$k$ rate is
	\begin{equation*}
		I_k:=\min_{D\in\mathcal C_k^{\rm root}}I_D.
	\end{equation*}
	Because $\Sigma_{D^{\oplus1}}$ is positive definite, the objective is
	strictly convex.  The set $\mathcal K_D$ is nonempty, closed,
	and convex.  Hence the infimum is attained at a unique point
	$x_D^*\in\mathcal K_D$, and $I_D\in(0,\infty)$.  Thus $I_D$ can be
	computed as a finite-dimensional convex quadratic program.  The minimum
	defining $I_k$ is over a finite collection, so $I_k$ is the Gaussian cost
	of the least unlikely rooted cluster shape of size $k$.  The exterior
	constraints in \eqref{eq:ld-shape-polyhedron} certify that $D$ is the complete
	cluster, rather than merely a connected set of exceedance sites, and in
	general they cannot be omitted from the optimization.  
	
	\begin{theorem}[Gaussian high-threshold large deviations]
		\label{thm:ld-high-threshold}
		Under the preceding assumptions, for every fixed $k\geq1$,
		\begin{align}
			\lim_{u\to\infty}
			-\frac1{u^2}\log w_k(u)
			&=
			\lim_{u\to\infty}
			-\frac1{u^2}\log w_k^{\rm peak}(u)
			=I_k,
			\label{eq:ld-weight-high-u}\\
			\lim_{u\to\infty}
			-\frac1{u^2}\log\rho_u
			&=
			\lim_{u\to\infty}
			-\frac1{u^2}\log\rho_u^{\rm peak}
			=\frac1{2\sigma^2}.
			\label{eq:ld-normalizer-high-u}
		\end{align}
		Consequently,
		\begin{equation}\label{eq:ld-pmf-high-u}
			\begin{split}
				\lim_{u\to\infty}
				-\frac1{u^2}\log\P(S_u=k)
				&=
				\lim_{u\to\infty}
				-\frac1{u^2}\log\P(S_u^{\rm peak}=k)=I_k-\frac1{2\sigma^2}.
			\end{split}
		\end{equation}
	\end{theorem}

	\begin{corollary}[Concentration on singleton clusters]
		\label{cor:ld-singleton}
		Under the assumptions of Theorem~\ref{thm:ld-high-threshold}, we have
		\begin{equation}\label{eq:ld-singleton-rate}
			I_1=\frac1{2\sigma^2},
			\qquad
			\P(S_u=1)\to 1,
			\qquad
			\P(S_u^{\rm peak}=1)\to 1.
		\end{equation}
		Moreover, $r_*=\max_{s\in\mathcal N(o)}
		\operatorname{Corr}(X_o,X_s)\in (-1,1)$, and for every fixed $k\geq2$,
		\begin{equation}\label{eq:ld-nonsingleton-rate}
			I_k-\frac1{2\sigma^2}
			\geq
			\frac{1-r_*}{2\sigma^2(1+r_*)}>0.
		\end{equation}
	\end{corollary}
	
	\begin{remark}
		A singleton cluster is an isolated exceedance: it consists of one
		site \(t\) such that \(X_t>u\), while all neighbors of \(t\) are at
		or below \(u\). Thus, \(\P(S_u=1)\to1\) means that the size of a
		typical finite cluster converges in distribution to one. It does not
		mean that the entire excursion set contains no non-singleton
		clusters. The same interpretation applies to the peak-based
		distribution.
	\end{remark}
	
	\begin{example}[One-dimensional white noise in both regimes]
		\label{example:ld-white-noise}
		Let $p_u=F(u)$ and $q_u=1-p_u\in(0,1)$, as in
		Examples~\ref{example:WN_1D} and \ref{example:peak-GWN}.  The root Palm law
		satisfies $\P(S_u=k)=p_uq_u^{k-1}$, while for $k\geq2$ the peak-based law is
		proportional to $(k+1)q_u^k$.  Hence, for fixed $u$, both distributions have
		the exact rate
		\[
		-\lim_{k\to\infty}\frac1k\log\P(S_u=k)
		=
		-\lim_{k\to\infty}\frac1k\log\P(S_u^{\rm peak}=k)
		=-\log q_u.
		\]
		If, in addition, $F=\Phi$ and $k$ is fixed while $u\to\infty$, then
		$q_u=1-\Phi(u)$ satisfies $-u^{-2}\log q_u\to1/2$ and
		$1-p_u^3\sim3q_u$.  The same formulas therefore give
		\begin{equation}\label{eq:ld-white-noise-high-u}
			\lim_{u\to\infty}-\frac1{u^2}\log\P(S_u=k)
			=
			\lim_{u\to\infty}-\frac1{u^2}
			\log\P(S_u^{\rm peak}=k)
			=\frac{k-1}{2}.
		\end{equation}
		Thus the two regimes have different speeds: $k$ at a fixed threshold and
		$u^2$ at a fixed cluster size.
	\end{example}
	
	\begin{remark}[Scope of the two regimes]
		\label{remark:ld-scope}
		Theorem~\ref{thm:ld-general} and Corollary~\ref{cor:ld-gaussian} provide
		sufficient, generally conservative, bounds as $k\to\infty$ with $u$
		fixed.  Gaussianity alone is insufficient in that regime: long-range
		dependence or a percolating excursion regime may prevent exponential
		cluster size decay.  In contrast,
		Theorem~\ref{thm:ld-high-threshold} gives the exact $u^2$-scale rate for
		each fixed $k$ under finite-dimensional nondegeneracy.  The two limits have
		different speeds, $k$ and $u^2$, and cannot be interchanged without
		additional uniform estimates.
	\end{remark}
	
	\section{Numerical Illustrations}\label{sec:simulation}
	We use numerical examples to illustrate the guarded estimators and the
	difference between root and peak sampling.  Closed forms are generally
	unavailable, especially in two dimensions. We therefore distinguish three
	sources of reference values: analytic values when explicit formulas are available, numerical evaluations of the exact finite-dimensional formulas for moderate $k$, and a large-window simulation benchmark when direct evaluation is infeasible.
	
	We use the computational cutoff $K=100$ and display probabilities only
	for cluster sizes $1,\ldots,J$, with $J=10$. For the root cluster law,
	both the reference values and the guarded Monte Carlo estimates target
	\[
	p_{k,100}
	=\frac{w_k}{\rho_{100}},
	\qquad
	\rho_{100}=\sum_{j=1}^{100}w_j,
	\qquad 1\leq k\leq100.
	\]
	All 100 weights enter the normalization, even though only the first
	ten probabilities are plotted. The truncation error relative to the
	full law $w_k/\rho_u$ is quantified by
	\eqref{eq:mc-exact-TV-truncation} and
	Proposition~\ref{prop:finite-sample-truncation}.
	
	For peak sampling, the curves target the full probabilities
	$w_k^{\rm peak}/\rho_u^{\rm peak}$, using the local peak intensity
	$\rho_u^{\rm peak}$ as the normalizing constant. Omitting larger cluster
	sizes does not rescale the retained probabilities. Since both types of
	plots display only sizes up to $J=10$, the displayed probabilities need
	not sum to one.
	
	\begin{figure}[t]
		\centering
		\begin{tabular}{@{}c@{\hspace{0.25in}}c@{}}
			\includegraphics[
			trim=10 20 20 0,clip,width=2.5in
			]{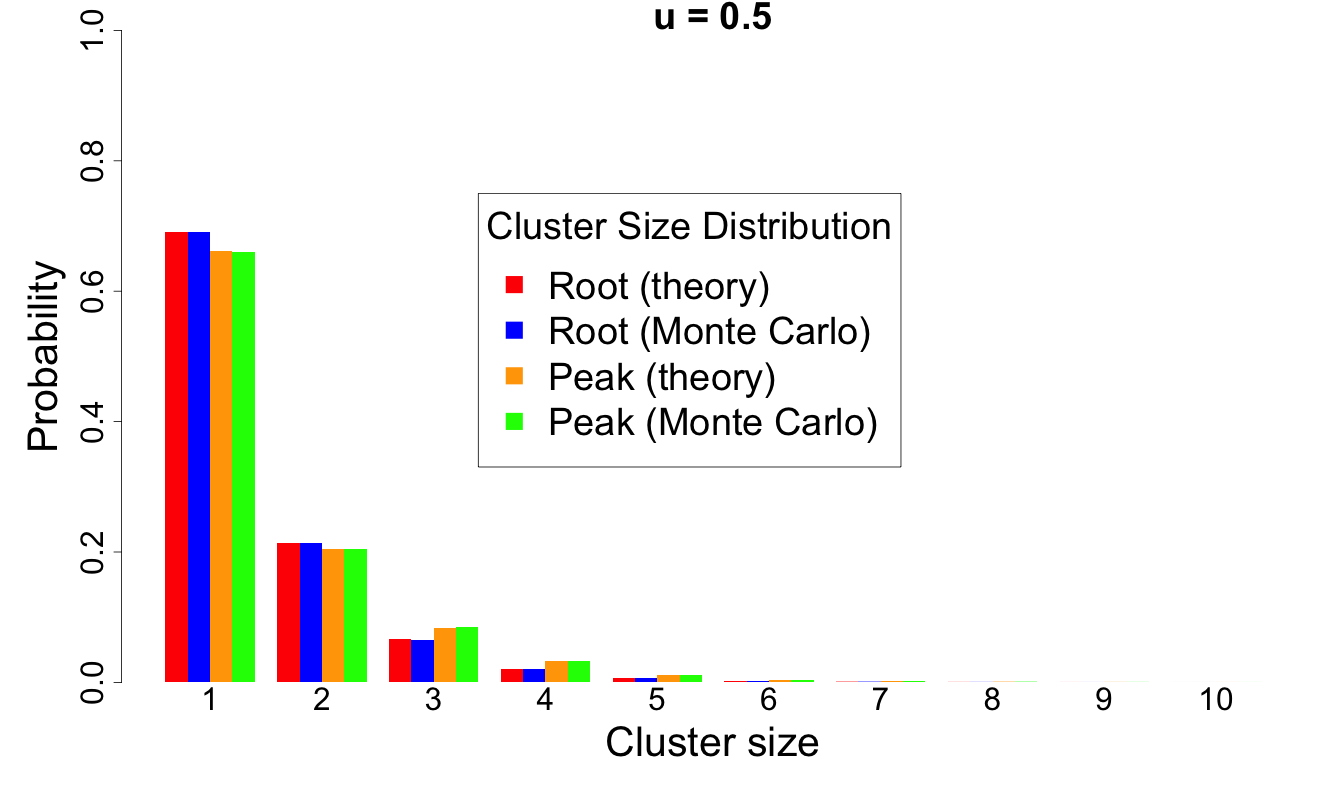}
			&
			\includegraphics[
			trim=10 20 20 0,clip,width=2.5in
			]{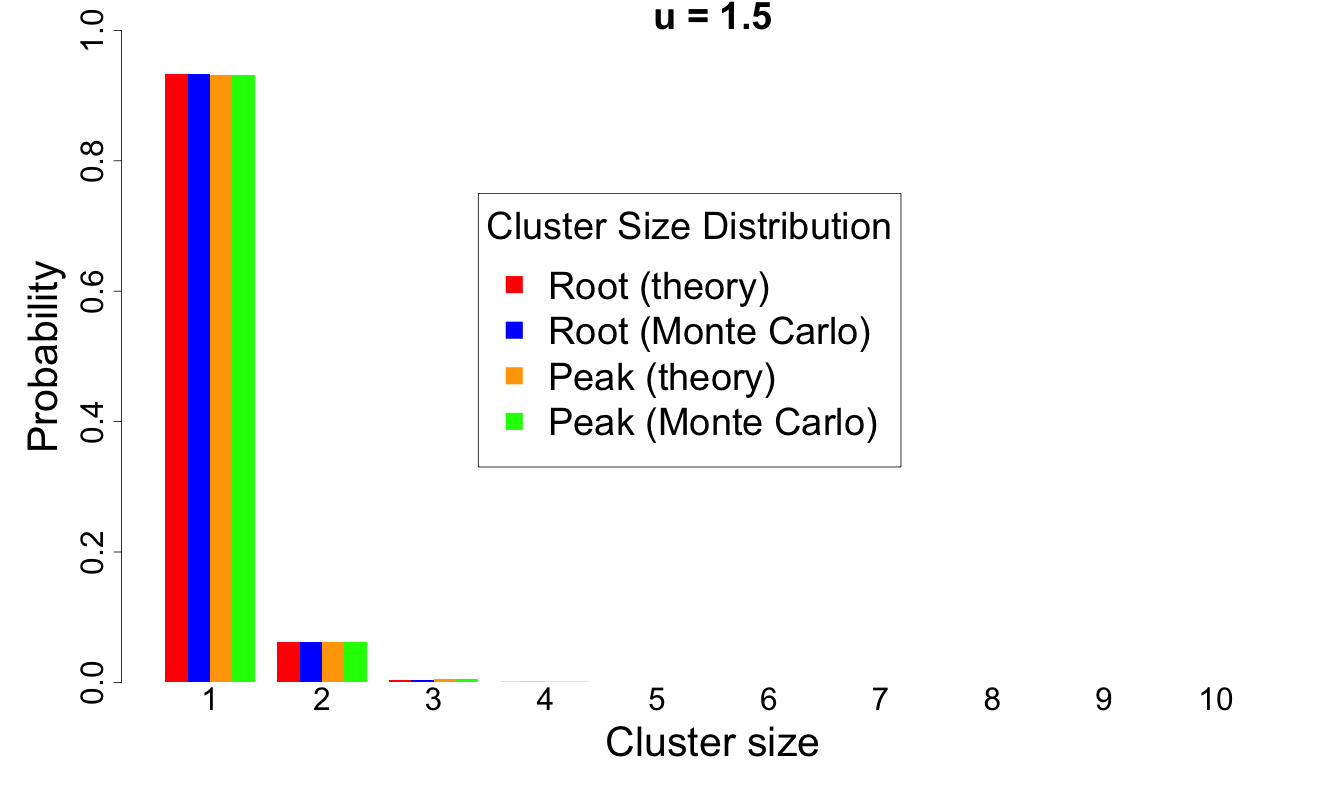}
			\\
			\multicolumn{2}{c}{
				(a) Gaussian white noise on $\mathbb{Z}$.
			}
			\\[1em]
			\includegraphics[
			trim=10 20 20 0,clip,width=2.5in
			]{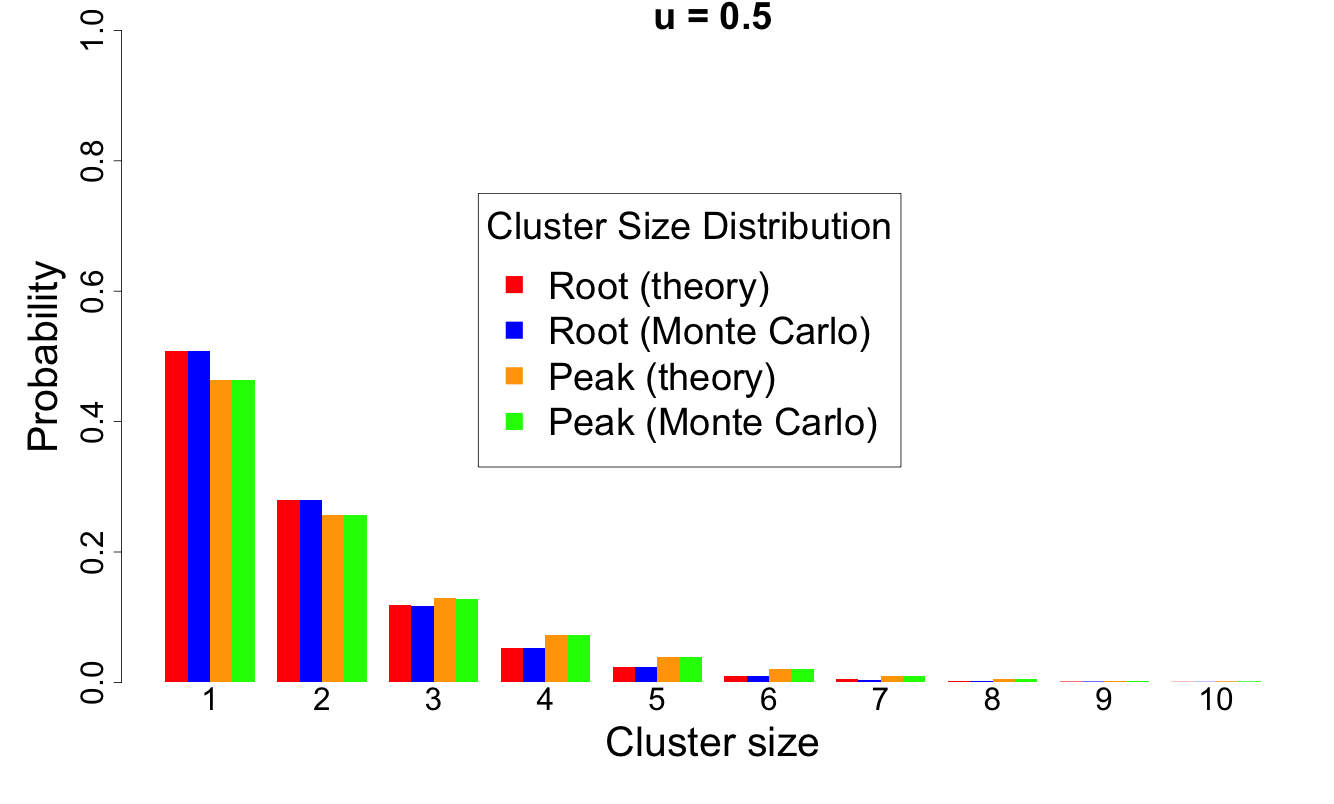}
			&
			\includegraphics[
			trim=10 20 20 0,clip,width=2.5in
			]{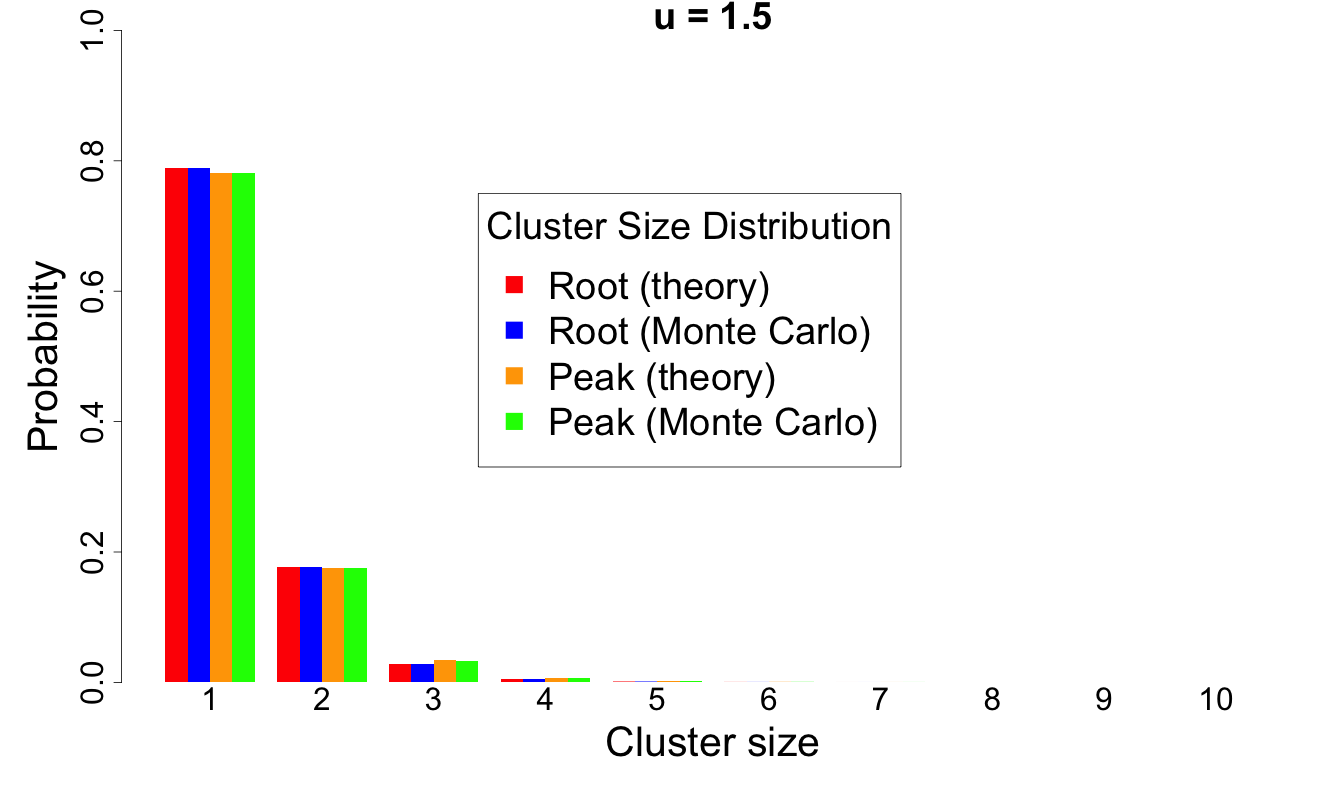}
			\\
			\multicolumn{2}{c}{
				(b) Gaussian process on $\mathbb{Z}$ with covariance
				$C(t,s)=e^{-(t-s)^2}$.
			}
		\end{tabular}
		\caption{
			Root and peak-based cluster size probability masses for
			(a) Gaussian white noise and
			(b) a Gaussian process with covariance
			$C(t,s)=e^{-(t-s)^2}$ on $\mathbb{Z}$. The theoretical values are compared with spatial Monte Carlo estimates.
			The left and right columns correspond to thresholds
			$u=0.5$ and $u=1.5$, respectively.		
		}
		\label{fig:sim-est-1D}
	\end{figure}

	\subsection{Gaussian processes on $\mathbb{Z}$}\label{sec:sim-1D}
	We consider standard Gaussian white noise and a stationary Gaussian process with
	covariance $C(t,s)=e^{-(t-s)^2}$.  For each model and threshold, the guarded
	experiment uses $M_1=10{,}000$ independent realizations and $K=100$. We take the score window 
	$\Lambda=\{101,\ldots,200\}$ and simulation window $V=\{1,\ldots,300\}$, so
	$V\supset\Lambda^{\oplus100}$ and every cluster classification contributing
	to the scores is determined without boundary ambiguity.
	
	For white noise, set $p=\Phi(u)$ and $q=1-p$.  As shown in Examples \ref{example:WN_1D} and \ref{example:peak-GWN}, the theoretical values are
	available analytically:
	\begin{align*}
		w_k&=p^2q^k,& \rho_u&=pq,\\
		w_1^{\rm peak}&=p^2q,&
		w_k^{\rm peak}&=\frac{k+1}{3}p^2q^k\quad(k\geq2),&
		\rho_u^{\rm peak}&=\frac{1-p^3}{3}.
	\end{align*}
	For comparison with the root Monte Carlo estimator, we use $\rho_{100}=\sum_{j=1}^{100}w_j$ and $p_{k,100}=w_k/\rho_{100}$.
	The full normalizer $\rho_u$ is used only to quantify the distinction between
	the truncated and full laws.
	
	For the correlated process, the theoretical values
	of $w_k$, $\rho_u$, $w_k^{\rm peak}$, and $\rho_u^{\rm peak}$ were obtained by
	numerical evaluation of \eqref{eq:w_1D}, \eqref{eq:w_1D_sum}, and the numerator
	and denominator of \eqref{eq:peak-GP}. Its root reference bars again use
	$\rho_{100}=\sum_{j=1}^{100}w_j$.
	
	Figure~\ref{fig:sim-est-1D} gives the grouped bar comparisons.  The guarded Monte Carlo estimates are visually close to their theoretical values in these examples.  The root and peak-based bars are also closer at $u=1.5$ than at $u=0.5$. This observation is qualitatively consistent with the asymptotic high-threshold result in Section~\ref{sec:high-u}.
	
	\begin{figure}[t]
		\centering
		\begin{tabular}{@{}c@{\hspace{0.25in}}c@{}}
			\includegraphics[
			trim=10 20 20 0,clip,width=2.5in
			]{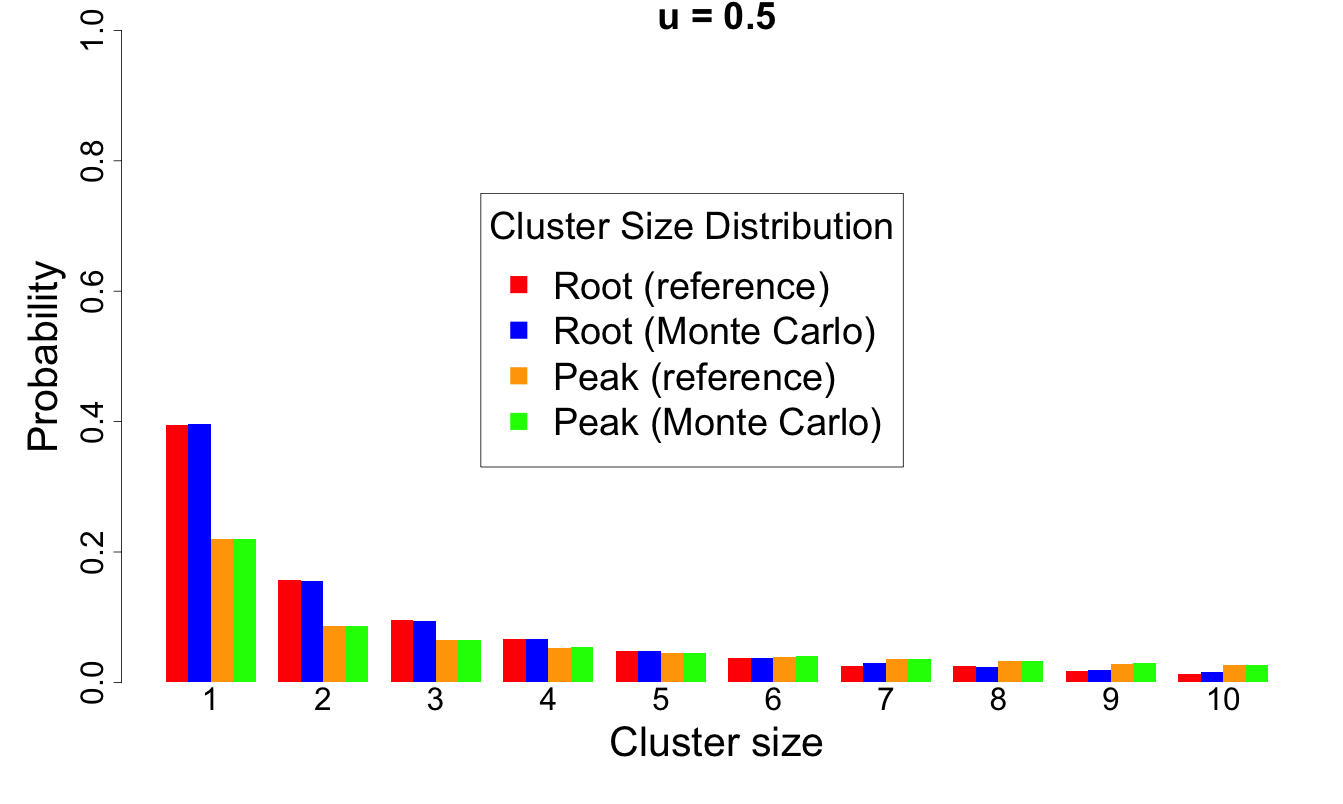}
			&
			\includegraphics[
			trim=10 20 20 0,clip,width=2.5in
			]{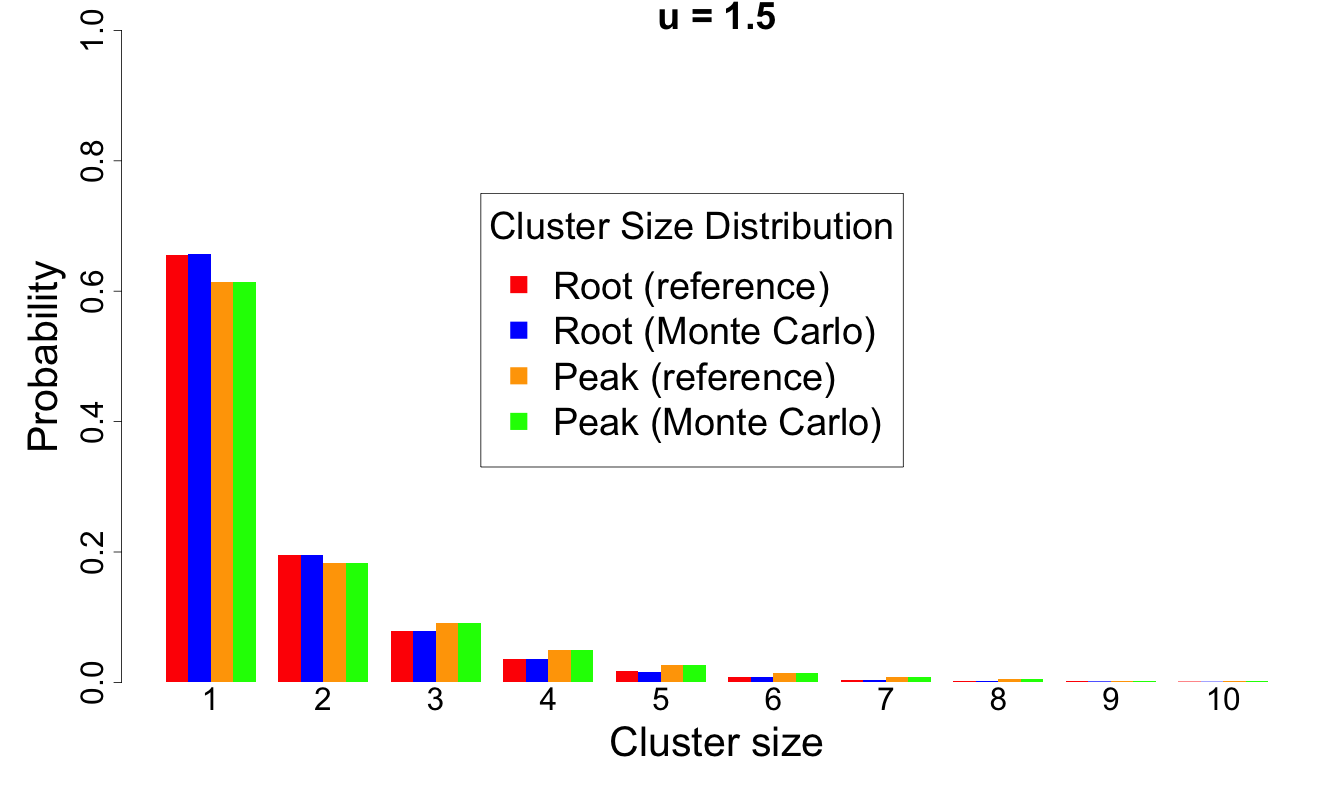}
			\\
			\multicolumn{2}{c}{
				(a) Nearest-neighbor connectivity on $\Z^2$.
			}
			\\[1em]
			\includegraphics[
			trim=10 20 20 0,clip,width=2.5in
			]{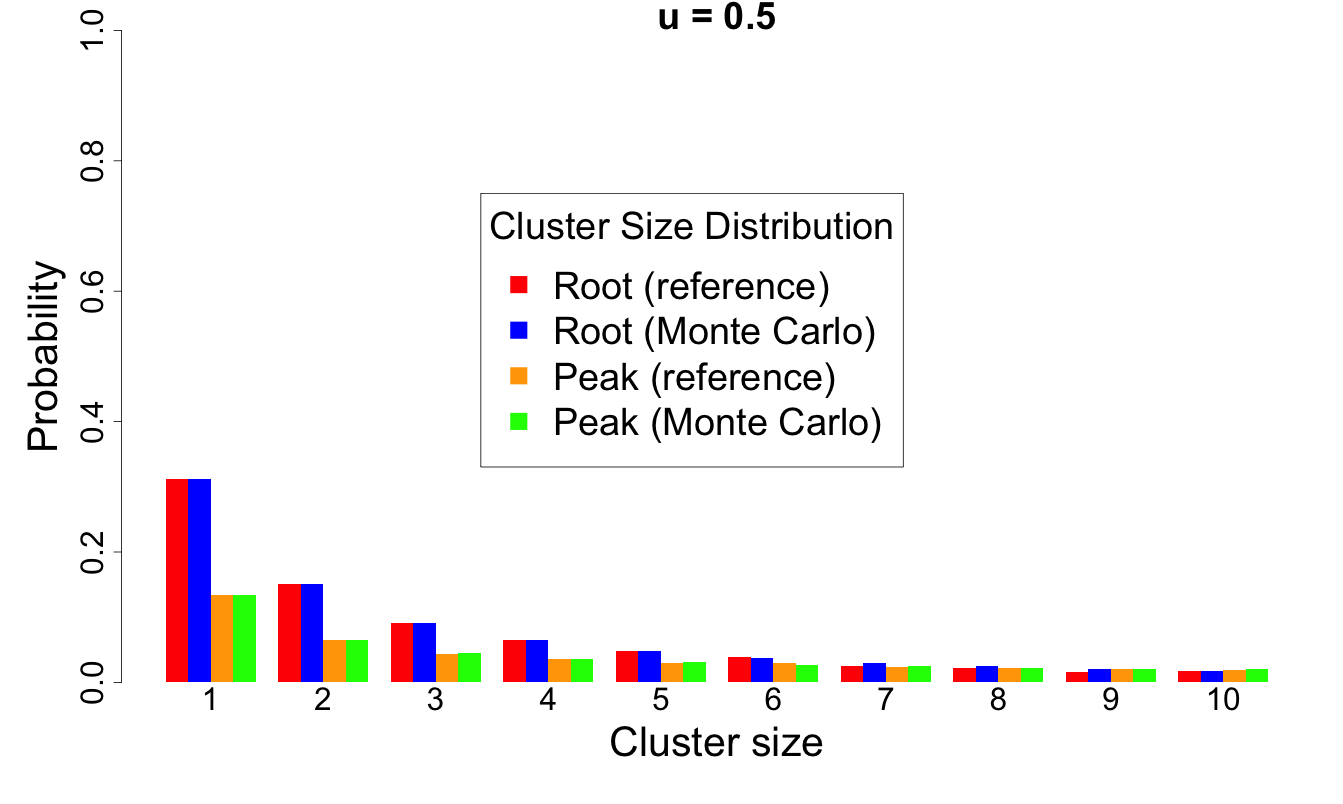}
			&
			\includegraphics[
			trim=10 20 20 0,clip,width=2.5in
			]{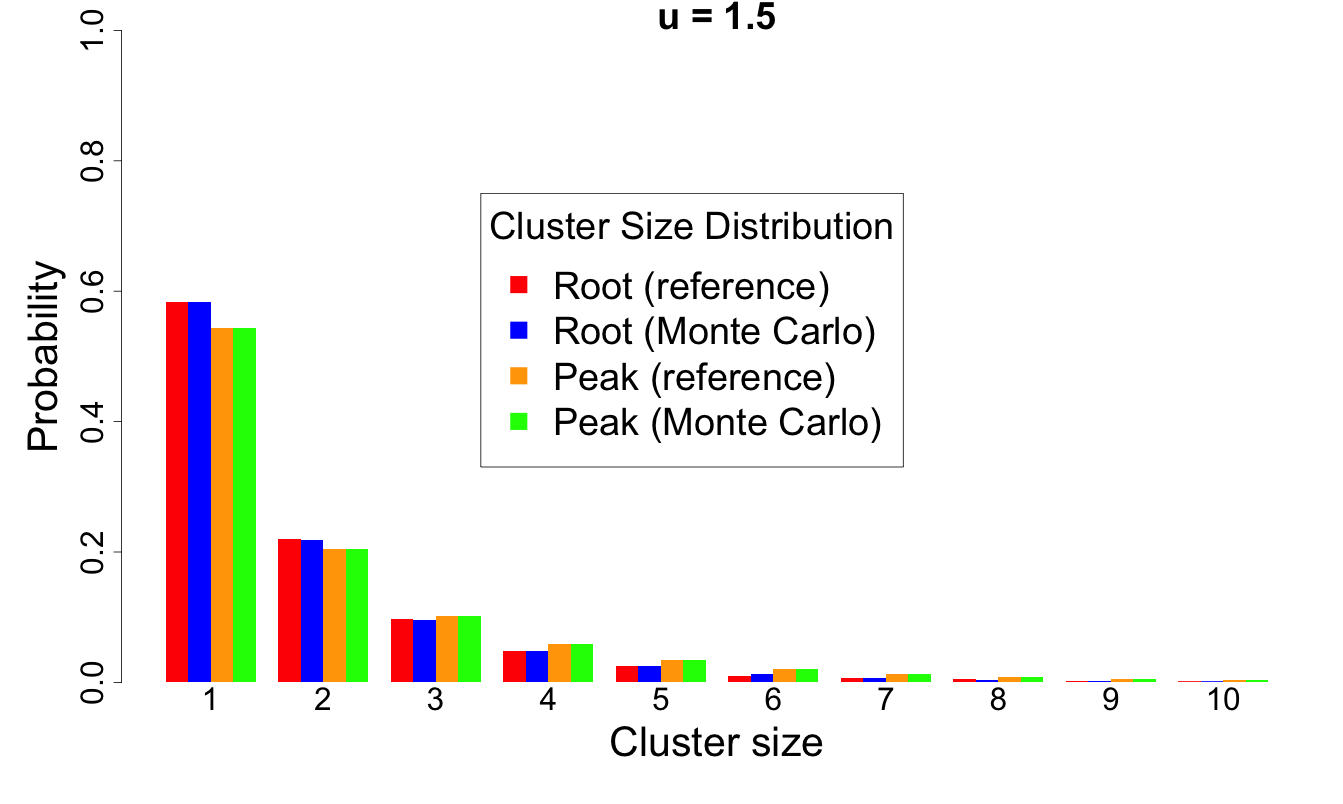}
			\\
			\multicolumn{2}{c}{
				(b) Moore connectivity on $\Z^2$.
			}
		\end{tabular}
		\caption{ 
			Root and peak-based cluster size probability masses for a Gaussian field on $\Z^2$ with covariance
			$C(t,s)=e^{-\|t-s\|^2}$ under
			(a) nearest-neighbor connectivity and 
			(b) Moore connectivity.
			The reference values are compared with spatial Monte Carlo estimates. The left and right columns correspond to thresholds
			$u=0.5$ and $u=1.5$, respectively.		
		}
		\label{fig:sim-est-2D}
	\end{figure}

	\subsection{Gaussian fields on $\mathbb{Z}^2$}\label{sec:sim-2D}
	We next consider the centered stationary Gaussian field on $\Z^2$ with
	\[
	\operatorname{Cov}(X_s,X_t)=e^{-\|t-s\|^2}.
	\]
	We use the same field and thresholds under nearest-neighbor and Moore
	adjacency; Figure~\ref{fig:sim-est-2D} collects the four comparisons.
	
	In general, reference values of $w_k$, $w_k^{\rm peak}$ and $\rho_u^{\rm peak}$ are computed from \eqref{eq:w_d}, \eqref{eq:w-peak} and \eqref{eq:peak-intensity-local}, respectively. Because exact numerical evaluation becomes costly for large cluster sizes, we numerically evaluate the exact finite-dimensional formulas for $k\le 6$ and estimate larger $k$ by empirical values. Specifically, we simulate the field on an integer lattice of size $3,000\times 3,000$. The empirical estimators of $w_k$, for $k\ge 7$, are provided in \eqref{eq:w-empirical}. Similarly, the empirical estimators for the peak-based quantities are given by \eqref{eq:w-peak-empirical}.
	
	The guarded Monte Carlo estimates use $M_2=4{,}000$ independent realizations and $K=100$. We take the score window $\Lambda=\{101,\ldots,200\}^2$ and simulation window $V=\{1,\ldots,300\}^2$. For both nearest-neighbor and Moore adjacency, $V\supset\Lambda^{\oplus100}$, so the guard condition in
	Lemma~\ref{lem:mc-guard-band} is satisfied.
	
	The results in Figure \ref{fig:sim-est-2D} show that Moore connectivity produces larger observed cluster sizes and lower probabilities for several small sizes. Across all examples presented in Figures \ref{fig:sim-est-1D} and \ref{fig:sim-est-2D}, the Monte Carlo estimated root and peak-based distributions closely match their reference counterparts. As the threshold increases from $u=0.5$ to $u=1.5$, the peak-based cluster size distribution converges toward the root distribution, consistent with the high-threshold approximation in Section \ref{sec:high-u}.
	
	\begin{table}[t]
		\centering
		\footnotesize
		\setlength{\tabcolsep}{4pt}
		\caption{Primary intensity estimates and Monte Carlo standard errors. WN denotes white noise and Corr denotes squared-exponential covariance. The peak normalizers are evaluated from the local field law. The last two columns give the peak overflow estimate beyond size 100 in two dimensions and its independent-realization standard error, as defined in Section~\ref{sec:numerical-uncertainty}. Zero overflow estimates and zero empirical standard errors indicate that no scored peak had an associated component larger than 100 in the simulated realizations.}
		\label{tab:numerical-normalizers}
		\begin{tabular}{lcccccc}
			\toprule
			Model & $u$ & $\widehat\rho_{100}$ & ${\rm SE}$ & $\rho_u^{\rm peak}$ & $\widehat P^{\rm peak}_{>100}$ & ${\rm SE}$ \\
			\midrule
			WN 1D & 0.5 & 0.213081 & $2.762\times10^{-4}$ & 0.22313258 & --- & --- \\
			WN 1D & 1.5 & 0.062263 & $2.254\times10^{-4}$ & 0.06244339 & --- & --- \\
			Corr 1D & 0.5 & 0.164756 & $2.521\times10^{-4}$ & 0.18036547 & --- & --- \\
			Corr 1D & 1.5 & 0.053180 & $2.066\times10^{-4}$ & 0.05385986 & --- & --- \\
			NN 2D & 0.5 & 0.062948 & $0.338\times10^{-4}$ & 0.11200208 & 0.001639 & $0.999\times10^{-4}$ \\
			NN 2D & 1.5 & 0.040915 & $0.270\times10^{-4}$ & 0.04373793 & 0.000000 & $0.000\times10^{-4}$ \\
			Moore 2D & 0.5 & 0.034217 & $0.309\times10^{-4}$ & 0.08059492 & 0.080852 & $7.142\times10^{-4}$ \\
			Moore 2D & 1.5 & 0.036051 & $0.225\times10^{-4}$ & 0.03851671 & 0.000000 & $0.000\times10^{-4}$ \\
			\bottomrule
		\end{tabular}
	\end{table}
	
	\begin{table}[t]
		\centering
		\footnotesize
		\setlength{\tabcolsep}{4pt}
		\caption{Per-configuration 95\% root truncation certificates at $K=100$. The displayed Total-Variation (TV) allowance controls the true truncated law versus the full root law.}
		\label{tab:numerical-truncation}
		\begin{tabular}{lcccccc}
			\toprule
			Model & $u$ & $\widehat r_{100}$ & $r_{100,U}$ & $\rho_{100,L}$ & $B_{100,U}$ & TV allowance \\
			\midrule
			WN 1D & 0.5 & 0.000000 & 0.013582 & 0.199499 & $1.345\times10^{-4}$ & 0.000674 \\
			WN 1D & 1.5 & 0.000000 & 0.013582 & 0.048682 & $1.345\times10^{-4}$ & 0.002755 \\
			Corr 1D & 0.5 & 0.000000 & 0.013582 & 0.151174 & $1.345\times10^{-4}$ & 0.000889 \\
			Corr 1D & 1.5 & 0.000000 & 0.013582 & 0.039598 & $1.345\times10^{-4}$ & 0.003385 \\
			NN 2D & 0.5 & 0.000730 & 0.022204 & 0.041474 & $2.199\times10^{-4}$ & 0.005273 \\
			NN 2D & 1.5 & 0.000000 & 0.021474 & 0.019441 & $2.127\times10^{-4}$ & 0.010818 \\
			Moore 2D & 0.5 & 0.032730 & 0.054204 & 0.012743 & $5.367\times10^{-4}$ & 0.040412 \\
			Moore 2D & 1.5 & 0.000000 & 0.021474 & 0.014577 & $2.127\times10^{-4}$ & 0.014376 \\
			\bottomrule
		\end{tabular}
	\end{table}
	
	\subsection{Sampling uncertainty, benchmark checks, and truncation}\label{sec:numerical-uncertainty}
	Table~\ref{tab:numerical-normalizers} reports the primary normalizers and
	sampling errors. Table~\ref{tab:numerical-truncation} implements
	Proposition~\ref{prop:finite-sample-truncation} with $\delta=0.05$ for
	each configuration separately. These are guarantees for each configuration,
	not a simultaneous 95\% guarantee over all eight rows. 
	The certificate controls omitted root probability mass, whereas the
	standard errors measure sampling uncertainty of reported estimates.
	
	Let $s_M(Z)$ denote the sample standard deviation of the scores
	$Z_1,\ldots,Z_M$ from independent realizations. For a cutoff $K$, define the
	peak overflow score in realization $m$ by
	\[
	T_{m,K}^{\rm peak}
	=\frac{1}{|\Lambda|}\sum_{t\in\Lambda}
	\mathbf{1}\bigl\{A_u^{{\rm peak},(m)}(t),\,
	|C_u^{(m)}(t\mid V)|>K\bigr\}.
	\]
	Under the guard condition $V\supset\Lambda^{\oplus K}$, this score
	classifies the full-lattice overflow event exactly at every scored site.
	Writing $\overline T_K^{\rm peak}=M^{-1}\sum_{m=1}^M T_{m,K}^{\rm peak}$,
	the conditional peak overflow probability and its estimate are
	\begin{align*}
		P^{\rm peak}_{>K}
		=\P\bigl\{|C_u(o)|>K\mid A_u^{\rm peak}(o)\bigr\},\quad
		\widehat P^{\rm peak}_{>K}
		=\frac{\overline T_K^{\rm peak}}{\rho_u^{\rm peak}},
		\quad
		\widehat{\operatorname{SE}}\bigl(\widehat P^{\rm peak}_{>K}\bigr)
		=\frac{s_M(T_{m,K}^{\rm peak})}{\sqrt M\,\rho_u^{\rm peak}}.
	\end{align*}
	We use the same simulation settings
	as in Sections \ref{sec:sim-1D} and \ref{sec:sim-2D}. A zero overflow estimate together with a zero empirical
	standard error means that $T_{m,100}^{\rm peak}=0$ in every simulated
	realization. These are observed simulation outcomes; they do not establish
	that $P^{\rm peak}_{>100}=0$. 
	
	The benchmark in Table~\ref{tab:numerical-benchmark} compares Monte Carlo estimates of
	singleton probabilities with their theoretical ($d=1$) and reference ($d=2$) values mentioned in Sections \ref{sec:sim-1D} and \ref{sec:sim-2D}. Parentheses contain estimated Monte Carlo standard errors, which
	describe the standard deviation of a probability estimate under
	repeated simulation of the entire experiment. 
	Each independent field realization supplies one spatial score, so
	dependence among sites within a realization is included in the
	uncertainty calculation. Let $R_m=\sum_{j=1}^{100}W_{m,j}$.
	The reported standard errors are computed as
	\begin{align*}
		\widehat{\operatorname{SE}}(\widehat p_{1,100})
		=\frac{s_M(W_{m,1}-\widehat p_{1,100}R_m)}
		{\sqrt M\,\widehat\rho_{100}},\quad
		\widehat{\operatorname{SE}}(\widehat p_1^{\rm peak})
		=\frac{s_M(W_{m,1}^{\rm peak})}
		{\sqrt M\,\rho_u^{\rm peak}}.
	\end{align*}
	The first expression is a delta method estimate that accounts for the
	random root denominator. The second treats the local peak normalizer
	as fixed, in accordance with the exact normalizer calculation in
	Section~\ref{sec:MC}.

	An additional check evaluates the representation directly. For each
	two-dimensional configuration,
	we evaluate all complete-shape Gaussian probabilities at sizes one and
	two and sum over their orientations. Table~\ref{tab:numerical-shapechecks}
	reports these intensities. A size-two connected component has exactly one
	strict local maximum, so $w_k=w_k^{\rm peak}$ for $k=1,2$ here. The
	integrations use eight independent randomizations per shape; the largest assessed standard error of a summed intensity is $2.44\times10^{-9}$. This provides a numerical evaluation independent of component labeling in simulated fields. Agreement with the guarded Monte Carlo estimates therefore
	supports the consistency of the shape probability calculation
	and the estimator implementation.
	
	\begin{table}[t]
		\centering
		\scriptsize
		\setlength{\tabcolsep}{4pt}
		\caption{Singleton probabilities. Parentheses give Monte Carlo standard errors for estimates obtained by simulation.}
		\label{tab:numerical-benchmark}
		\begin{tabular}{lccccc}
			\toprule
			Model & $u$ & Root Monte Carlo & Root reference & Peak Monte Carlo & Peak reference \\
			\midrule
			WN 1D & 0.5 & 0.69150 (0.00100) & 0.69146 & 0.66035 (0.00143) & 0.66112 \\
			WN 1D & 1.5 & 0.93269 (0.00101) & 0.93319 & 0.92999 (0.00354) & 0.93171 \\
			Corr 1D & 0.5 & 0.50988 (0.00121) & 0.50783 & 0.46575 (0.00145) & 0.46416 \\
			Corr 1D & 1.5 & 0.78987 (0.00177) & 0.78869 & 0.77989 (0.00356) & 0.78129 \\
			NN 2D & 0.5 & 0.39157 (0.00030) & 0.39157 & 0.22007 (0.00023) & 0.21991 \\
			NN 2D & 1.5 & 0.65498 (0.00037) & 0.65576 & 0.61271 (0.00056) & 0.61411 \\
			Moore 2D & 0.5 & 0.30810 (0.00038) & 0.30795 & 0.13080 (0.00021) & 0.13065 \\
			Moore 2D & 1.5 & 0.57997 (0.00040) & 0.58070 & 0.54284 (0.00054) & 0.54391 \\
			\bottomrule
		\end{tabular}
	\end{table}

	\begin{table}[t]
		\centering
		\small
		\setlength{\tabcolsep}{4pt}
		\caption{Independent complete-shape Gaussian evaluations in $\Z^2$. For sizes one and two, root and peak intensities coincide. Values are numerical integration estimates; assessed standard errors are at most $2.44\times10^{-9}$.}
		\label{tab:numerical-shapechecks}
		\begin{tabular}{lccc}
			\toprule
			Connectivity & $u$ & $w_1=w_1^{\rm peak}$ & $w_2=w_2^{\rm peak}$ \\
			\midrule
			NN & 0.5 & 0.0246305067 & 0.0097439110 \\
			NN & 1.5 & 0.0268600761 & 0.0079984215 \\
			Moore & 0.5 & 0.0105300498 & 0.0051144835 \\
			Moore & 1.5 & 0.0209494620 & 0.0078837914 \\
			\bottomrule
		\end{tabular}
	\end{table}
	
	\section{Discussion}\label{sec:discussion}
	
	The main conceptual point of this paper is that a cluster size distribution
	is inseparable from its sampling rule. The root law samples each finite
	cluster once, the site law samples in proportion to cluster size,
	and the peak law samples in proportion to the number of peaks.
	The identities developed above make these changes of measure explicit and
	prevent the three distributions from being interpreted interchangeably.
	Canonical rooting and complete shape enumeration provide the common
	representation. Their consequences are operational: reciprocal cluster
	size removes site-sampling bias, peak counts determine an exact change of
	measure, and finite guards make every retained intensity observable
	without a component-boundary approximation. The overflow score then
	turns the omitted root mass into a finite-sample allowance. For Gaussian
	fields, the same complete shape event identifies the exterior constraints
	in the high-threshold quadratic program.
	
	Stationarity is essential for interpreting a root at the origin as a typical
	cluster in an infinite field.  The fixed-site peak law remains well defined
	without stationarity, but it is then a location-specific estimand.  Estimating local cluster size distributions from a single realization
	would require pooling observations from nearby sites. This approach
	relies on local stationarity, so that the joint field distribution varies
	little across the estimation neighborhood. The spatial bandwidth must
	balance bias from combining different local laws against variance from
	using too few observed peaks. A rigorous analysis would therefore require
	conditions on spatial variation and dependence, together with a bandwidth
	regime that provides increasing local information while preserving
	approximate homogeneity.
	
	The two asymptotic regimes answer different questions.  At a fixed threshold,
	the lattice-animal argument supplies a conservative sufficient bound for
	large cluster sizes; it is not a sharp percolation criterion.  At fixed size
	and increasing Gaussian threshold, the quadratic programs retain exterior
	nonexceedance constraints and therefore describe complete components.  The
	agreement of root and peak sampling at high thresholds is driven chiefly by
	singleton concentration, rather than by a general identity between the two
	sampling schemes.
	
	The Monte Carlo guarantees concern simulation from a specified field law.
	Guard bands eliminate component classification error only up to the chosen
	cutoff, while the overflow score and its confidence bound quantify the
	remaining root Palm truncation error.  Ratio estimators are consistent but
	need not be unbiased, and uncertainty should be computed over independent
	field realizations.  
	
	Finally, these cluster laws are descriptive model summaries, not automatic
	multiple testing procedures.  For a fixed location $t$ under a fully
	specified null, the conditional tail
	\[
	\P\{|C_u(t)|\geq c\mid A_u^{\rm peak}(t)\}
	=
	\frac{\sum_{k\geq c}w_k^{\rm peak}(t)}
	{\rho_u^{\rm peak}(t)}
	\]
	is a legitimate, generally conservative discrete reference probability.
	It does not by itself account for selecting $t$ after inspecting the
	field, testing multiple peaks, estimating the null model from the same
	data, or repeatedly testing a cluster containing several peaks.
	An important direction for future work is to build on these results
	to develop procedures that account for selection and control errors
	from multiple testing.
	
	\begin{appendix}
		\section*{Appendix: Proofs}\label{appn} 
		The appendix contains proofs for Theorems \ref{thm:empirical-cluster-counts}, \ref{thm:empirical-peak-counts}, \ref{thm:root-peak-comparison} and Corollary \ref{cor:gaussian-root-peak}, as well as for all results in Sections \ref{sec:MC} and \ref{sec:large-deviations}.
		
		\subsection*{Proof of Theorem \ref{thm:empirical-cluster-counts}} The lexicographic order is invariant under translations. Consequently,
		$Y_{t,k}$ and $Y_t$ are translation-equivariant measurable functions of the
		field. Because the field is stationary and ergodic, both indicator fields are
		stationary and ergodic. Furthermore,
		\[
		\E Y_{o,k}
		=\P(R_u(o),|C_u(o)|=k)=w_k,
		\qquad
		\E Y_o=\P(R_u(o))=\rho_u.
		\]
		The pointwise ergodic theorem for the rectangular windows $\Lambda_n$ gives
		the two limits in \eqref{eq:ideal-count-limits}. Stationarity and linearity of
		expectation also give \eqref{eq:ideal-count-expectations}.
		
		It remains to show that discarding boundary affected clusters changes these
		counts by $o(|\Lambda_n|)$. On $R_u(t)$, define the closure radius
		\begin{equation}\label{eq:cluster-closure-radius}
			L_t
			=\max\{d_G(t,s):s\in C_u(t)\cup\mathcal N(C_u(t))\}.
		\end{equation}
		Set $L_t=0$ off $R_u(t)$. The graph is locally finite, and
		Assumption~\ref{ass:finite-cluster} implies that $L_t<\infty$ a.s.
		
		Fix an integer $r\geq1$. If $t$ is a cluster root, $L_t\leq r$, and
		$t\notin\partial_r\Lambda_n$, then
		\[
		C_u(t)\cup\mathcal N(C_u(t))\subset\Lambda_n.
		\]
		That cluster is therefore included in
		$ \mathscr C_u^\circ(\Lambda_n)$. Conversely, every component in
		$ \mathscr C_u^\circ(\Lambda_n)$ is a full-lattice cluster with its root in
		$\Lambda_n$. Hence, for every $k$,
		\begin{align}
			0&\leq N_k^*(\Lambda_n)-N_k^\circ(\Lambda_n)
			\leq |\partial_r\Lambda_n|
			+\sum_{t\in\Lambda_n}Y_t\mathbbm{1}\{L_t>r\},
			\label{eq:size-count-boundary-error}\\
			0&\leq N^*(\Lambda_n)-N^\circ(\Lambda_n)
			\leq |\partial_r\Lambda_n|
			+\sum_{t\in\Lambda_n}Y_t\mathbbm{1}\{L_t>r\}.
			\label{eq:total-count-boundary-error}
		\end{align}
		Divide by $|\Lambda_n|$. The first term on the right tends to zero. For fixed $r$, the pointwise
		ergodic theorem gives
		\[
		\frac1{|\Lambda_n|}\sum_{t\in\Lambda_n}
		Y_t\mathbbm{1}\{L_t>r\}
		\overset{\rm a.s.}{\longrightarrow}
		\P(R_u(o),L_o>r).
		\]
		Taking the limit superior in
		\eqref{eq:size-count-boundary-error}--\eqref{eq:total-count-boundary-error}
		and letting $r\to\infty$ yields zero, because $L_o<\infty$ almost surely
		on $R_u(o)$. Thus the observable and ideal counts have the same 
		limits, proving \eqref{eq:w-empirical}.
		
		Finally, $\widetilde\rho_{u,n}\to\rho_u>0$, so its denominator is eventually
		positive almost surely. Taking the ratio in \eqref{eq:w-empirical} and using
		Theorem~\ref{thm:CSD_d} proves \eqref{eq:empirical}. Because there are only
		countably many values of $k$, the convergence in \eqref{eq:empirical} holds
		simultaneously for all $k$ on a single event with probability one. On that event,
		both
		\[
		\left\{
		\frac{N_k^\circ(\Lambda_n)}{N^\circ(\Lambda_n)}:k\geq1
		\right\}
		\quad\text{and}\quad
		\{\P(S_u=k):k\geq1\}
		\]
		are probability mass functions. Thus Scheff\'e's
		lemma for counting measure yields
		\eqref{eq:empirical-total-variation}.
		
		\subsection*{Proof of Theorem \ref{thm:empirical-peak-counts}}
		
		The fields $\{Z_{t,k}\}$ and $\{Z_t\}$ are translation-equivariant
		measurable functions of $X$ and hence are stationary and ergodic, with
		means $w_k^{\rm peak}$ and $\rho_u^{\rm peak}$.  The pointwise ergodic
		theorem gives the two limits for the ideal counts, and stationarity gives
		the displayed expectation identities.
		
		It remains to control clusters cut by the observation boundary.  Set
		$L_t^{\rm peak}=0$ off $A_u^{\rm peak}(t)$ and, on that event, define
		\[
		L_t^{\rm peak}
		:=\max\{d_G(t,s):s\in C_u(t)\cup\mathcal N(C_u(t))\}.
		\]
		Local finiteness of the graph and Assumption~\ref{ass:finite-cluster}
		imply $L_t^{\rm peak}<\infty$ almost surely.  For every fixed integer
		$r\geq1$, a peak $t\in\Lambda_n\setminus\partial_r\Lambda_n$ with
		$L_t^{\rm peak}\leq r$ belongs to an interior cluster and is therefore
		counted observably.  Conversely, every observably counted peak is a genuine
		full-lattice peak.  Thus, for the size-$k$ and total counts alike, the
		difference between ideal and observable counts is bounded by
		\begin{equation*}
			|\partial_r\Lambda_n|
			+\sum_{t\in\Lambda_n}Z_t
			\mathbbm{1}\{L_t^{\rm peak}>r\}.
		\end{equation*}
		After division by $|\Lambda_n|$, the first term tends to zero and the
		second converges almost surely, by the ergodic theorem, to
		$\P\{A_u^{\rm peak}(o),L_o^{\rm peak}>r\}$.  Letting $r\to\infty$ shows
		that the observable and ideal counts have the same limits.  Since the
		limiting denominator is positive, taking ratios proves
		\eqref{eq:empirical-peak-ratio}.  The convergence holds simultaneously for
		all $k$, and Scheff\'e's lemma for counting measure gives total-variation
		convergence.
		
		\subsection*{Proof of Theorem \ref{thm:root-peak-comparison}}
		For fixed $k$, let every peak in a size-$k$ cluster send one unit of mass
		to the root of that cluster.  The mass-transport principle gives
		\begin{equation}\label{eq:root-peak-mass-transport}
			\begin{split}
				w_k^{\rm peak}
				&= \E\left[
				M_u(o)\mathbbm{1}\{|C_u(o)|=k\}\mathbbm{1}\{R_u(o)\}
				\right]\\
				&=
				\rho_u
				\E\left[
				M_u(o)\mathbbm{1}\{|C_u(o)|=k\}\mid R_u(o)
				\right].
			\end{split}
		\end{equation}
		Summing over $k$ yields
		\begin{equation}\label{eq:peak-root-intensity-relation}
			\rho_u^{\rm peak}
			=
			\rho_u\E\left[M_u(o)\mid R_u(o)\right].
		\end{equation}
		
		Let
		\[
		\mu_u=\E\left[M_u(o)\mid R_u(o)\right],\quad
		a_{k,u}
		=
		\E\left[
		(M_u(o)-1)\mathbbm{1}\{|C_u(o)|=k\}\mid R_u(o)
		\right].
		\]
		Then $a_{k,u}\geq0$ and $\sum_{k\geq1}a_{k,u}=\mu_u-1$.
		Taking the ratio between \eqref{eq:root-peak-mass-transport} and \eqref{eq:peak-root-intensity-relation} yields \eqref{eq:change-of-measure},
		\begin{equation*}
			\begin{split}
				\P(S_u^{\rm peak}=k)&=\frac{\E\left[
					M_u(o)\mathbbm{1}\{|C_u(o)|=k\}\mid R_u(o)
					\right]}{\E\left[M_u(o)\mid R_u(o)\right]} \\
				&= \frac{\E\left[
					M_u(o)\mid R_u(o), |C_u(o)|=k
					\right]}{\E\left[
					M_u(o)\mid R_u(o) \right]}\P(S_u=k),
			\end{split}
		\end{equation*}
		as well as
		\[
		\P(S_u^{\rm peak}=k)
		=
		\frac{\E\left[\mathbbm{1}\{|C_u(o)|=k\}\mid R_u(o)
			\right]+a_{k,u}}{\mu_u} = \frac{\P(S_u=k) + a_{k,u}}{\mu_u}.
		\]
		Therefore,
		\begin{align*}
			d_{\rm TV}\!\left(\mathcal L(S_u^{\rm peak}),
			\mathcal L(S_u)\right)\leq
			\frac{1}{2\mu_u}
			\left\{
			\sum_{k=1}^\infty a_{k,u}
			+
			(\mu_u-1)\sum_{k=1}^\infty \P(S_u=k)
			\right\}=
			\frac{\mu_u-1}{\mu_u}.
		\end{align*}
		Finally, \eqref{eq:peak-root-intensity-relation} gives the final equality
		in \eqref{eq:root-peak-TV-bound}.
		
		\subsection*{Proof of Corollary \ref{cor:gaussian-root-peak}}
		Write $\Psi(x)=1-\Phi(x)$ for the standard Gaussian tail function. Let
		\[
		q_u=\P(X_o>u)=\Psi(u)
		\]
		and
		\[
		b_u
		=
		\P\left(
		X_o>u,\
		\max_{s\in\mathcal N(o)}X_s>u
		\right).
		\]
		An exceedance at the origin with no neighboring exceedance is a singleton
		cluster. The origin is then both its root and a local maximum.
		Therefore,
		\[
		q_u-b_u\leq\rho_u.
		\]
		Moreover, Theorem~\ref{thm:root-peak-comparison} gives
		$\rho_u\leq\rho_u^{\rm peak}$, while
		$\rho_u^{\rm peak}\leq q_u$. Hence
		\[
		q_u-b_u
		\leq\rho_u
		\leq\rho_u^{\rm peak}
		\leq q_u.
		\]
		Using Theorem~\ref{thm:root-peak-comparison} again,
		\begin{equation}\label{eq:gaussian-TV-short-bound}
			d_{\rm TV}\!\left(\mathcal L(S_u^{\rm peak}),
			\mathcal L(S_u)\right)
			\leq
			1-\frac{\rho_u}{\rho_u^{\rm peak}}
			\leq
			\frac{b_u}{q_u}.
		\end{equation}
		
		For every
		$s\in\mathcal N(o)$,
		\[
		\P(X_o>u,X_s>u)
		\leq
		\Psi\left(
		\sqrt{2/(1+r_*)}\,u
		\right).
		\]
		Thus, by the union bound,
		\[
		\frac{b_u}{q_u}
		\leq
		|\mathcal N(o)|
		\frac{
			\Psi\left(
			\sqrt{2/(1+r_*)}\,u
			\right)}
		{\Psi(u)}
		\to 0,
		\]
		where the convergence follows from Mills' ratio because $r_*<1$.
		The conclusion now follows from
		\eqref{eq:gaussian-TV-short-bound} by choosing $0<\alpha<(1-r_*)/[2(1+r_*)]$.
		
		\subsection*{Proof of Lemma \ref{lem:mc-guard-band}}
		If $|C_u(t)|=k\leq K$, every vertex of the component lies in
		$t^{\oplus(k-1)}$, and every exterior neighbor lies in
		$t^{\oplus k}\subset V$.  Hence the component closes inside $V$ and
		$C_u(t\mid V)=C_u(t)$.  If $|C_u(t)|>K$, a breadth-first search finds
		$K+1$ connected exceedance sites in $t^{\oplus K}\subset V$, so
		$|C_u(t\mid V)|>K$. The converse follows from
		$C_u(t\mid V)\subset C_u(t)$.
		Finally, $A_u^{\rm peak}(t)$ depends only on $X_t$ and the values at
		$\mathcal N(t)\subset t^{\oplus 1}\subset t^{\oplus K}$.
		
		\subsection*{Proof of Theorem \ref{thm:mc-unbiased}}
		By Lemma~\ref{lem:mc-guard-band}, the cluster-size and overflow
		classifications computed on $V=\Lambda^{\oplus K}$ agree with their
		full-lattice counterparts at every site of $\Lambda$.  Therefore,
		stationarity and linearity of expectation give
		\begin{align*}
			\E W_{m,k}
			&=\frac1{|\Lambda|}\sum_{t\in\Lambda}\E\xi_{t,k}
			=w_k, \quad \E W_{m,k}^{\rm peak}
			=\frac1{|\Lambda|}\sum_{t\in\Lambda}
			\E\xi_{t,k}^{\rm peak}
			=w_k^{\rm peak}.\\
			\E T_m
			&=\frac1{|\Lambda|}\sum_{t\in\Lambda}
			\P(|C_u(t)|>K) = 
			r_K.
		\end{align*}
		Averaging over $m$ proves \eqref{eq:mc-unbiasedness}.  
		
		For each estimator, the replication-level quantities are independent and
		identically distributed across $m$ and are bounded.  The strong law of
		large numbers proves the consistency in \eqref{eq:mc-ratio-consistency}. Finally, $\E\widehat p_k^{\rm peak}=w_k^{\rm peak}/\rho_u^{\rm peak}$ follows
		from the deterministic denominator.
		
		\subsection*{Proof of Proposition \ref{prop:finite-sample-truncation}}
		Hoeffding's inequality applied to the independent variables
		$T_m\in[0,1]$ and $R_m\in[0,1]$, followed by a union bound, gives
		$r_K\leq r_{K,U}$ and $\rho_K\geq\rho_{K,L}$ simultaneously with
		probability at least $1-\delta$.  Combining these inequalities with
		\eqref{eq:tau-B} and the monotonicity of
		$x/(y+x)$ in $x\geq0$ and $y\geq0$, we obtain the desired result.
		
		\subsection*{Proof of Corollary \ref{cor:mc-two-stage-consistency}}
		For fixed $K$, Theorem~\ref{thm:mc-unbiased} gives simultaneous
		almost sure convergence of the finitely many estimated weights and their
		sum.  Hence 
		\[
		\lim_{M\to\infty}\sum_{k=1}^{\infty} \left|
		\widehat p_{k,K,M}
		- p_{k,K}
		\right| =0, \qquad \text{a.s.}
		\]
		Eq.~\eqref{eq:mc-exact-TV-truncation} yields the first limit.
		Finally, $\tau_K\downarrow0$ because
		$\rho_u=\sum_{k\geq1}w_k<\infty$, proving the second limit.

		\subsection*{Proof of Proposition \ref{prop:mc-window-variance}}
		Lemma~\ref{lem:mc-guard-band} allows the guarded scores computed on $V$
		to be replaced by their full-lattice counterparts. By \eqref{eq:xi} and \eqref{eq:MC2}, 
		\[
		W_{m,k}
		=\frac1{|\Lambda|}\sum_{t\in\Lambda}\xi_{t,k}^{(m)}.
		\]
		Stationarity gives
		\begin{align*}
			\operatorname{Var}(W_{m,k})
			=\frac1{|\Lambda|^2}
			\sum_{s,t\in\Lambda}
			\operatorname{Cov}(\xi_{s,k},\xi_{t,k})=\frac1{|\Lambda|^2}
			\sum_{h\in\Z^d}N_\Lambda(h)\gamma_k(h).
		\end{align*}
		Averaging $M$ independent realizations divides this variance by $M$,
		proving \eqref{eq:mc-window-variance}.  The bound in \eqref{eq:mc-window-variance-bound} follows from $0\leq a_\Lambda(h)\leq1$ and the triangle inequality.
		
		\subsection*{Proof of Theorem \ref{thm:ld-general}}
		If $|C_u(o)|\geq k$, a breadth-first exploration of $C_u(o)$ produces a
		connected subset $D\in\mathcal C_k^{\rm site}$ on which every site
		exceeds $u$.  Hence, by \eqref{eq:ld-animal-count},
		Assumption~\ref{ass:ld-connected-exceedance}, and the union bound,
		\[
		\P\{|C_u(o)|\geq k\}
		\leq
		\sum_{D\in\mathcal C_k^{\rm site}}
		\P(X_t>u\text{ for every }t\in D)
		\leq \beta^k e^{-\psi_u k}
		=e^{-c_u k}.
		\]
		Letting $k\to\infty$ shows that
		$\P\{|C_u(o)|=\infty\}=0$.  By stationarity and the countability of
		$\Z^d$, every excursion cluster is therefore finite almost surely under
		the assumptions of the theorem.
		Furthermore, \eqref{eq:w-inside-w} yields
		\[
		\P(S_u\geq k)
		=\frac1{\rho_u}\sum_{j\geq k}w_j
		\leq\frac1{k\rho_u}\sum_{j\geq k}jw_j
		=\frac{\P\{|C_u(o)|\geq k\}}{k\rho_u}.
		\]
		For the peak-based law,
		\[
		\P(S_u^{\rm peak}\geq k)
		=
		\frac{\P\{A_u^{\rm peak}(o),\ |C_u(o)|\geq k\}}
		{\rho_u^{\rm peak}}
		\leq
		\frac{\P\{|C_u(o)|\geq k\}}{\rho_u^{\rm peak}}.
		\]
		Combining these inequalities with \eqref{eq:ld-site-tail} proves
		\eqref{eq:ld-peak-tail}.
		
		\subsection*{Proof of Corollary \ref{cor:ld-gaussian}}
		Fix $D\in\mathcal C_k^{\rm site}$ and set
		$Y_D=\sum_{t\in D}X_t$.  By stationarity,
		\[
		\operatorname{Var}(Y_D)
		=\sum_{s,t\in D}\Gamma(t-s)
		\leq
		\sum_{s\in D}\sum_{h\in\Z^d}|\Gamma(h)|
		=\kappa k.
		\]
		The event $\{X_t>u\text{ for every }t\in D\}$ implies $Y_D>ku$.
		If $\operatorname{Var}(Y_D)=0$, this event has probability zero.
		Otherwise, the Gaussian Chernoff bound gives
		\[
		\P(X_t>u\text{ for every }t\in D)
		\leq\P(Y_D>ku)
		\leq
		\exp\!\left\{-\frac{k^2u^2}
		{2\operatorname{Var}(Y_D)}\right\}
		\leq\exp\!\left(-\frac{ku^2}{2\kappa}\right).
		\]
		Thus Assumption~\ref{ass:ld-connected-exceedance} holds with
		$\psi_u=u^2/(2\kappa)$, and
		Theorem~\ref{thm:ld-general} applies.
		
		\subsection*{Proof of Theorem \ref{thm:ld-high-threshold}}
		Fix $D\in\mathcal C_k^{\rm root}$.  The vectors
		$u^{-1}(X_t)_{t\in D^{\oplus 1}}$ satisfy a finite-dimensional Gaussian
		large deviation principle with speed $u^2$ and rate function
		\[
		Q_D(x)=\frac12x^\top\Sigma_{D^{\oplus 1}}^{-1}x.
		\]
		By \eqref{eq:C_u(o)}, the event that $D$ is the excursion cluster rooted at the origin is $\{C_u(o)=D\}$. Let
		\[
		\mathcal K_D^\circ
		:=
		\{x:x_t>1\text{ for }t\in D,
		x_s<1\text{ for }s\in\mathcal N(D)\}.
		\]
		Then
		$\{u^{-1}X_{D^{\oplus 1}}\in\mathcal K_D^\circ\}
		\subset \{C_u(o)=D\}\subset
		\{u^{-1}X_{D^{\oplus 1}}\in\mathcal K_D\}$.
		The set $\mathcal K_D^\circ$ is dense in $\mathcal K_D$: one can
		approximate a boundary point by increasing its coordinates on $D$ and
		decreasing those on $\mathcal N(D)$.  Since $Q_D$ is continuous,
		\[
		\inf_{x\in\mathcal K_D^\circ}Q_D(x)
		=
		\inf_{x\in\mathcal K_D}Q_D(x)=I_D.
		\]
		The Gaussian large deviation upper and lower bounds therefore give
		\begin{equation}\label{eq:ld-one-shape-rate}
			\lim_{u\to\infty}
			-\frac1{u^2}\log\P\{C_u(o)=D\}=I_D.
		\end{equation}
		Because $\mathcal C_k^{\rm root}$ is finite and $w_k(u)$ is the sum of
		the corresponding shape probabilities, \eqref{eq:ld-one-shape-rate}
		implies
		\[
		\lim_{u\to\infty}
		-\frac1{u^2}\log w_k(u)
		=\min_{D\in\mathcal C_k^{\rm root}}I_D=I_k.
		\]
		Since $k$ is fixed, the
		second equality in \eqref{eq:ld-weight-high-u} follows from \eqref{eq:ld-weight-sandwich}.
		
		It remains to determine the normalizing rates.  Define
		\begin{equation*}
			q_u:=\P(X_o>u),
			\qquad
			b_u:=\P\!\left(
			X_o>u,\ \max_{s\in\mathcal N(o)}X_s>u
			\right).
		\end{equation*}
		If $X_o>u$ and every neighbor of $o$ is at most $u$, then $\{o\}$ is a
		singleton cluster and $o$ is both its root and its peak.  Therefore,
		\begin{equation}\label{eq:ld-normalizer-squeeze}
			q_u-b_u\leq\rho_u\leq q_u,
			\qquad
			q_u-b_u\leq\rho_u^{\rm peak}\leq q_u.
		\end{equation}
		Positive definiteness and the finiteness of $\mathcal N(o)$ imply $r_*=\max_{s\in\mathcal N(o)}
		\operatorname{Corr}(X_o,X_s)<1$.
		For each $s\in\mathcal N(o)$,
		\[
		\P(X_o>u,X_s>u)
		\leq\P(X_o+X_s>2u)
		\leq
		\exp\!\left\{-\frac{u^2}
		{\sigma^2\{1+\operatorname{Corr}(X_o,X_s)\}}\right\}.
		\]
		A union bound and the Gaussian marginal tail rate consequently yield
		\begin{equation*}
			\begin{aligned}
				\frac{b_u}{q_u}
				&\leq
				\exp\left\{
				-\frac{1-r_*}{2\sigma^2(1+r_*)}u^2+o(u^2)
				\right\}
				\to 0,\quad 
				\lim_{u\to\infty}-\frac1{u^2}\log q_u
				=\frac1{2\sigma^2}.
			\end{aligned}
		\end{equation*}
		It follows from \eqref{eq:ld-normalizer-squeeze} that
		$\rho_u\sim q_u$ and $\rho_u^{\rm peak}\sim q_u$, proving
		\eqref{eq:ld-normalizer-high-u}.  Finally, divide the weight asymptotics
		in \eqref{eq:ld-weight-high-u} by their corresponding normalizers to obtain
		\eqref{eq:ld-pmf-high-u}.

		\subsection*{Proof of Corollary \ref{cor:ld-singleton}}
		The singleton event in \eqref{eq:ld-normalizer-squeeze} has probability
		$q_u-b_u\sim q_u$.  It follows that its rate is $1/(2\sigma^2)$ and that
		both normalized singleton probabilities tend to one.  If $k\geq2$, a
		connected size-$k$ cluster containing $o$ contains an exceedance neighbor
		of $o$.  Thus $w_k(u)\leq b_u$, and the pair bound in the proof of
		Theorem~\ref{thm:ld-high-threshold} gives
		\[
		I_k\geq\frac1{\sigma^2(1+r_*)}.
		\]
		Subtracting the normalizer rate $1/(2\sigma^2)$ proves
		\eqref{eq:ld-nonsingleton-rate}.
	\end{appendix}

	
	\begin{funding}
		The authors were supported by NSF Grant DMS-2220523 and Simons Foundation Collaboration Grant \#854127. 
	\end{funding}


\begin{thebibliography}{99}
		
		\bibitem{Adler81}
		Adler, R. J. (1981). {\it The Geometry of Random Fields}. Wiley,
		New York.
		
		\bibitem{AT07}
		Adler, R. J. and Taylor, J. E. (2007). {\it Random Fields and Geometry}. Springer, New York.
		
		\bibitem{AizenmanBarsky87}
		Aizenman, M. and Barsky, D. J. (1987). Sharpness of the phase transition
		in percolation models. \emph{Commun. Math. Phys.}, \textbf{108}, 489--526.
		
		\bibitem{AntunovicVeselic08}
		Antunovi\'c, T. and Veseli\'c, I. (2008). Sharpness of the phase transition
		and exponential decay of the subcritical cluster size for percolation on
		quasi-transitive graphs. \emph{J. Stat. Phys.}, \textbf{130}, 983--1009.
		
		\bibitem{Bansal:2018}
		Bansal, R. and Peterson, B. S. (2018). Cluster-level statistical inference in fMRI datasets: The unexpected behavior of random fields in high dimensions. {\it Magn Reson Imaging},
		{\bf 49},  101--115.
		
		\bibitem{Bardeen:1986}
		Bardeen, J. M., Bond, J. R., Kaiser, N. and Szalay, A. S. (1986). The statistics of peaks of Gaussian random fields. {\it Astrophys. J.},
		{\bf 304},  15--61.
		
		\bibitem{Benjamini:2019}
		Benjamini, Y., Taylor, J. and Irizarry, R. A. (2019). Selection-corrected statistical inference for region detection with high-throughput assays. \emph{Journal of the American Statistical Association}, {\bf 114}, 1351--1365.
		
		\bibitem{BroadbentHammersley57}
		Broadbent, S. R. and Hammersley, J. M. (1957). Percolation processes. I.
		Crystals and mazes. \emph{Proc. Cambridge Philos. Soc.}, \textbf{53},
		629--641.
		
		\bibitem{Cheng:2020}
		Cheng, D., Cammarota, V., Fantaye, Y., Marinucci, D. and Schwartzman, A. (2020). Multiple testing of local maxima for detection of peaks on the (celestial) sphere. {\it Bernoulli}, {\bf 26},  31--60.
		
		\bibitem{CS15}
		Cheng, D. and Schwartzman, A. (2015). Distribution of the height of local maxima of Gaussian random fields. {\it Extremes}, {\bf 18}, 213--240.
		
		\bibitem{CS18}
		Cheng, D. and Schwartzman, A. (2018). Expected number and height distribution of critical points of smooth isotropic Gaussian random fields. {\it Bernoulli}, {\bf 24}, 3422--3446.
		
		\bibitem{CS17}
		Cheng, D. and Schwartzman, A. (2017). Multiple testing of local maxima for detection of peaks in random fields. {\it Ann. Stat.}, {\bf 45}, 529--556.
		
		\bibitem{Chiles:2012}
		Chiles, J.-P. and Delfiner, P. (2012). {\it Geostatistics: Modeling Spatial Uncertainty}. Wiley, New York.
		
		\bibitem{Chumbley:2010}
		Chumbley, J. R., Worsley, K., Flandin, G. and Friston, K. J. (2010). Topological FDR for neuroimaging.
		\emph{Neuroimage}, {\bf 49}, 3057--3064.
		
		\bibitem{CoeurjollyEtAl17}
		Coeurjolly, J.-F., M{\o}ller, J. and Waagepetersen, R. (2017). A tutorial
		on Palm distributions for spatial point processes. \emph{Int. Stat. Rev.},
		\textbf{85}, 404--420.
		
		\bibitem{Cooley:2007}
		Cooley, D., Nychka, D. and Naveau, P. (2007). Bayesian spatial modeling of extreme precipitation return levels. \emph{J. Am. Statist. Assoc.}, {\bf 102}, 824--840.
		
		\bibitem{Davison:2012}
		Davison, A. C., Padoan, S. A. and Ribatet, M. (2012). Statistical modeling of spatial extremes. \emph{Statistical Science}, {\bf 27}, 161--168.
		
		\bibitem{DemboZeitouni10}
		Dembo, A. and Zeitouni, O. (2010). \emph{Large Deviations Techniques and
			Applications}, 2nd ed. Springer, Berlin.
		
		\bibitem{Fondeville:2018}
		de Fondeville, R. and Davison, A. C. (2018). High-dimensional peaks-over-threshold inference. \textit{Biometrika}, \textbf{105}, 575--592.
		
		\bibitem{Friston:1994}
		Friston, K. J.,  Worsley, K. J., Frackowiak, R. S., Mazziotta, J. C. and Evans, A. C.  (1994). Assessing the significance of focal activations using their spatial extent. \textit{Human Brain Mapping}, \textbf{1(3)}, 210--220.
		
		\bibitem{Goeman:2023}
		Goeman, J. J., Gorecki, P., Monajemi, R., Chen, X., Nichols, T. E. and Weeda, W. (2023). Cluster extent inference revisited: quantification and localisation of brain activity. \textit{Journal of the Royal Statistical Society Series B}, \textbf{85}, 1128--1153.
		
		\bibitem{Grimmett99}
		Grimmett, G. (1999). \emph{Percolation}, 2nd ed. Springer, Berlin.
		
		\bibitem{Huser:2014}
		Huser, R. and Davison, A. C. (2014). Space–time modelling of extreme events. \textit{JRSSB}, \textbf{76}, 439--461.
		
		\bibitem{Last06}
		Last, G. (2006). Stationary partitions and Palm probabilities.
		\emph{Adv. Appl. Probab.}, \textbf{38}, 602--620.
		
		\bibitem{LastThorisson09}
		Last, G. and Thorisson, H. (2009). Invariant transports of stationary
		random measures and mass-stationarity. \emph{Ann. Probab.}, \textbf{37},
		790--813.
		
		\bibitem{Lindgren82}
		Lindgren, G. (1982). Wave characteristics distributions for Gaussian waves -- wave-length, amplitude and steepness. {\it Ocean Engng.}, {\bf 9},  411--432.
		
		\bibitem{MP11}
		Marinucci, D. and Peccati, G. (2011). {\it Random Fields on the Sphere. Representation, Limit Theorems and Cosmological Applications}. Cambridge University Press.
		
		\bibitem{Muirhead24}
		Muirhead, S. (2024). Percolation of strongly correlated Gaussian fields II.
		Sharpness of the phase transition. \emph{Ann. Probab.}, \textbf{52},
		838--881.
		
		\bibitem{Planinic23}
		Planini\'c, H. (2023). Palm theory for extremes of stationary regularly
		varying time series and random fields. \emph{Extremes}, \textbf{26},
		45--82.
		
		\bibitem{Poline:1997}
		Poline, J. B., Worsley, K. J., Evans, A. C. and Friston, K. J. (1997). Combining spatial extent and peak intensity to test for activations in functional imaging. \emph{Neuroimage}, {\bf 5}, 83--96.
		
		\bibitem{Robert09}
		Robert, C. Y. (2009). Inference for the limiting cluster size distribution
		of extreme values. \emph{Ann. Statist.}, \textbf{37}, 271--310.
		
		\bibitem{RonnNielsenStehr23}
		R{\o}nn-Nielsen, A. and Stehr, M. (2023). Extremal clustering and cluster
		counting for spatial random fields. \emph{Bernoulli}, \textbf{29},
		2771--2796.
		
		\bibitem{StaufferAharony92}
		Stauffer, D. and Aharony, A. (1992). \emph{Introduction to Percolation
			Theory}, 2nd ed. Taylor \& Francis, London.
		
		\bibitem{SykesEssam64}
		Sykes, M. F. and Essam, J. W. (1964). Exact critical percolation
		probabilities for site and bond problems in two dimensions.
		\emph{J. Math. Phys.}, \textbf{5}, 1117--1127.
		
		\bibitem{Tan:2021}
		Tan, X., Wu, X., and Liu, B. (2021). Global changes in the spatial extents of precipitation extremes. \emph{Environmental Research Letters}, {\bf 16}, 054017.
		
		\bibitem{Taylor:2007}
		Taylor, J. E. and Worsley, K. J. (2007). Detecting sparse signals in random fields, with an application to brain mapping. \emph{Journal of the American Statistical Association}, {\bf 102}, 913--928.
		
		\bibitem{Thorisson07}
		Thorisson, H. (2007). The Palm-duality for random subsets of
		$d$-dimensional grids. \emph{Adv. Appl. Probab.}, \textbf{39}, 318--325.
		
		\bibitem{Vanmarcke:2010}
		Vanmarcke, E. (2010). \textit{Random Fields: Analysis and Synthesis.} World Scientific.
		
		\bibitem{Worsley:1996a}
		Worsley, K. J., Marrett, S., Neelin, P., Vandal, A. C., Friston, K. J. and Evans, A. C. (1996). A unified statistical approach for determining significant signals in images of cerebral activation. \emph{Human Brain Mapping}, {\bf 4}, 58--73.
		
		\bibitem{Zhang:2009}
		Zhang, H.,  Nichols, T. E. and Johnson, T. D. (2009). Cluster mass inference via random field theory. \emph{Neuroimage}, {\bf 44}, 51--61.
		
		\bibitem{Zhong:2025}
		Zhong, P., Brunner, M., Opitz, T. and Huser, R. (2025). Spatial modeling and future projection of extreme precipitation extents. \emph{Journal of the American Statistical Association}, {\bf 120}, 80--95.
		
	\end{thebibliography}
\end{document}